\documentclass[11pt, twoside, leqno]{article}

\usepackage{amssymb}
\usepackage{amsmath}
\usepackage{amsthm}
\usepackage{color}

\allowdisplaybreaks

\def\rr{{\mathbb R}}
\def\rn{{\mathbb{R}^n}}

\def\zz{{\mathbb Z}}

\def\nn{{\mathbb N}}

\def\cs{{\mathcal S}}

\def\mi{{\mathrm I}}

\def\fz{\infty }
\def\az{\alpha}
\def\bz{\beta}
\def\dz{\delta}

\def\lz{\lambda}

\def\lf{\left}
\def\r{\right}

\def\hs{\hspace{0.25cm}}
\def\ls{\lesssim}
\def\gs{\gtrsim}

\def\noz{\nonumber}
\def\wz{\widetilde}
\def\wh{\widehat}

\def\com{\complement}

\def\dis{\displaystyle}

\def\supp{\mathop\mathrm{\,supp\,}}

\def\lpq{{L^{p,q}(\rn)}}

\newtheorem{theorem}{Theorem}[section]
\newtheorem{lemma}[theorem]{Lemma}
\newtheorem{corollary}[theorem]{Corollary}
\newtheorem{proposition}[theorem]{Proposition}

\theoremstyle{definition}
\newtheorem{remark}[theorem]{Remark}
\newtheorem{definition}[theorem]{Definition}
\renewcommand{\appendix}{\par
   \setcounter{section}{0}%
   \setcounter{subsection}{0}%
   \setcounter{subsubsection}{0}%
   \gdef\thesection{\@Alph\c@section}%
   \gdef\thesubsection{\@Alph\c@section.\@arabic\c@subsection}%
   \gdef\theHsection{\@Alph\c@section.}%
   \gdef\theHsubsection{\@Alph\c@section.\@arabic\c@subsection}%
   \csname appendixmore\endcsname
 }

\numberwithin{equation}{section}

\begin{document}

\arraycolsep=1pt

\title{\bf\Large Anisotropic Hardy-Lorentz Spaces and
Their Applications
\footnotetext{\hspace{-0.35cm} 2010 {\it
Mathematics Subject Classification}. Primary  42B35;
Secondary 46E30, 42B25, 42B20.
\endgraf {\it Key words and phrases.} Lorentz space,
anisotropic Hardy-Lorentz space, expansive dilation,
Calder\'{o}n reproducing formula, grand maximal function,
atom, molecule, Calder\'{o}n-Zygmund  operator.
\endgraf This project is supported by the National
Natural Science Foundation of China
(Grant Nos.~11571039 and 11471042). }}
\author{Jun Liu, Dachun Yang\,\footnote{Corresponding author}\ \
and Wen Yuan}
\date{}
\maketitle

\vspace{-0.8cm}

\begin{center}
\begin{minipage}{13cm}
{\small {\bf Abstract}\quad
Let $p\in(0,1]$, $q\in(0,\infty]$ and  $A$ be a general
expansive matrix on $\mathbb{R}^n$. The authors
introduce the anisotropic Hardy-Lorentz space
$H^{p,q}_A(\mathbb{R}^n)$ associated with $A$ via the
non-tangential grand maximal function and then establish its
various real-variable characterizations
in terms of the atomic or  the molecular decompositions, the radial
or the non-tangential maximal functions, or the finite atomic decompositions.
All these characterizations except the $\infty$-atomic characterization are
new even for the classical isotropic Hardy-Lorentz spaces on $\mathbb{R}^n$.
As applications, the authors first
prove that $H^{p,q}_A(\mathbb{R}^n)$
is an intermediate space between $H^{p_1,q_1}_A(\mathbb{R}^n)$
and $H^{p_2,q_2}_A(\mathbb{R}^n)$ with
$0<p_1<p<p_2<\infty$ and $q_1,\,q,\,q_2\in(0,\infty]$, and also
between $H^{p,q_1}_A(\mathbb{R}^n)$
and $H^{p,q_2}_A(\mathbb{R}^n)$ with
$p\in(0,\infty)$ and $0<q_1<q<q_2\leq\infty$
in the real method of interpolation.
The authors then establish a criterion on the boundedness of sublinear
operators from $H^{p,q}_A(\mathbb{R}^n)$ into a quasi-Banach
space; moreover, the authors obtain the boundedness of $\delta$-type
Calder\'{o}n-Zygmund operators from $H^p_A(\mathbb{R}^n)$ to
the weak Lebesgue space $L^{p,\infty}(\mathbb{R}^n)$ (or $H^{p,\infty}_A(\mathbb{R}^n)$)
in the critical case,
from $H_A^{p,q}(\mathbb{R}^n)$ to $L^{p,q}(\mathbb{R}^n)$ (or $H_A^{p,q}(\mathbb{R}^n)$)
with $\delta\in(0,\frac{\ln\lambda_-}{\ln b}]$,
$p\in(\frac1{1+\delta},1]$ and $q\in(0,\infty]$, as well as the boundedness of some
Calder\'{o}n-Zygmund operators from $H_A^{p,q}(\mathbb{R}^n)$
to $L^{p,\infty}(\mathbb{R}^n)$, where $b:=|\det A|$,
$\lambda_-:=\min\{|\lambda|:\ \lambda\in\sigma(A)\}$ and $\sigma(A)$
denotes the set of all eigenvalues of $A$.}
\end{minipage}
\end{center}

\vspace{0.0cm}

\section{Introduction\label{s1}}

\hskip\parindent
The study of Lorentz spaces originated from Lorentz \cite{l50,l51}
in the early 1950's. As a generalization of $L^p(\rn)$, Lorentz spaces
are known to be an intermediate space of Lebesgue spaces
in the real interpolation method; see \cite{c64,lp64,p63}. For a  systematic
treatment of Lorentz spaces as well as their dual spaces, we refer the reader to
Hunt \cite{h66}, Cwikel \cite{c75} and Cwikel and Fefferman
\cite{cf80,cf84}; see also \cite{bs88,bl76,lg08,s76,sw71}. It is well known that,
due to their fine structures, Lorentz spaces play an irreplaceable role in the study
on various critical or endpoint analysis problems from many different research fields
and there exists a lot of literatures on this subject, here we only mention
several recent papers from harmonic analysis (see, for example,
\cite{osttw12,mtt03,st01,tw01}) and partial differential equations (see, for example,
\cite{iiy15,mr15,p15}).

On the other hand, the theory of Hardy spaces has been well developed
and these spaces play an important role in many branches of analysis and
partial differential equations; see, for example,
\cite{clms93,fs72,fs82,lg09,m94,s94,s93,sto89}. It is well known that
Hardy spaces are good substitutes of Lebesgue spaces when $p\in(0,1]$,
particularly, for the study on the boundedness of maximal functions and singular integral operators.
Moreover, for the study on the boundedness of operators,
the weak Hardy space $H^{1,\fz}(\rn)$ is also a good substitute of $L^{1,\fz}(\rn)$.
Recall that the weak
Hardy spaces $H^{p,\fz}(\rn)$ with $p\in(0,1)$ were first introduced by
Fefferman, Rivi\'{e}re and Sagher \cite{frs74} in 1974, which naturally appeared
as the intermediate spaces of Hardy spaces $H^p(\mathbb{R}^n)$ with $p\in(0,1]$ via
the real interpolation. Later on, to find out the biggest
space from which the Hilbert transform
is bounded to the weak Lebesgue space $L^{1,\fz}(\rn)$,
Fefferman and Soria \cite{fs87} introduced the weak Hardy space
$H^{1,\fz}(\rn)$, in which they
also obtained the $\fz$-atomic decomposition of $H^{1,\fz}(\rn)$
and the boundedness of some Calder\'on-Zygmund operators from
$H^{1,\fz}(\rn)$ to $L^{1,\fz}(\rn)$.
In 1994, \'{A}lvarez \cite{a94} considered the Calder\'on-Zygmund
theory related to $H^{p,\fz}(\rn)$ with $p\in(0,1]$, while
Liu \cite{l91} studied the weak Hardy spaces $H^{p,\fz}(\rn)$ with $p\in(0,1]$
on homogeneous groups. Nowadays, it is well known that the weak Hardy spaces
$H^{p,\fz}(\rn)$, with $p\in(0,1]$, play
a key role when studying the boundedness of operators in the critical case;
see, for example, \cite{a94,a98,dl03,g92,dls06,w13,dlq07}.
Moreover, it is known that the weak Hardy spaces
$H^{p,\fz}(\rn)$ are special cases of the Hardy-Lorentz spaces $H^{p,q}(\rn)$
which, when $p=1$ and $q\in(1,\fz)$, were introduced and investigated
by Parilov \cite{p05}. In 2007, Abu-Shammala and Torchinsky \cite{wa07}
studied the Hardy-Lorentz spaces $H^{p,q}(\rn)$ for the full
range $p\in(0,1]$ and $q\in(0,\fz]$, and obtained their $\fz$-atomic
characterization,  real interpolation properties over parameter $q$, and
the boundedness of singular integrals and some other operators on these spaces.
In 2010, Almeida and Caetano \cite{ac10} studied the generalized Hardy spaces
which include the classical Hardy-Lorentz spaces $H^{p,q}(\rn)$ investigated in
\cite{wa07} as special cases.  To be more precise, Almeida and Caetano \cite{ac10} obtained
some maximal characterizations of these generalized Hardy spaces,
some real interpolation results with function parameter and, as applications,
they studied the behavior of some classical operators in this generalized setting.

As the series of works (see, for example, \cite{frs74,fs87,a94,l91,p05,wa07,ac10}) reveal,
the Hardy-Lorentz spaces (as well the weak Hardy spaces)
serve as a more subtle research object than the usual Hardy
spaces when considering the boundedness of singular integrals, especially,
in some endpoint cases, due to the fact that these function spaces own finer structures.
However, the real-variable
theory of these spaces is still not complete.
For example,  the $r$-atomic, with $r\in(1,\fz)$, or the molecular characterizations,
the characterizations in terms of the
radial or the non-tangential maximal functions, and the finite atomic characterizations
of Hardy-Lorentz spaces are still unknown.

On the other hand, from  1970's, there has been an increasing interesting in extending classical
function spaces arising in harmonic analysis from Euclidean spaces to
anisotropic settings and some other domains; see, for example,
\cite{ct75,ct77,fs82,st87,t83,t92}. The study of function spaces on $\rn$
associated with anisotropic dilations was originally started from these celebrated
articles \cite{c77,ct75,ct77} of Calder\'{o}n and Torchinsky on anisotropic Hardy spaces.
Since then, the theory of  anisotropic function spaces  was well developed
by many authors; see, for example, \cite{fs82,s93,t83}. In 2003, Bownik
\cite{mb03} introduced and investigated the anisotropic Hardy spaces
associated with general expansive dilations, which were extended to the
weighted setting in \cite{blyz08}. For further developments of function
spaces on the anisotropic Euclidean spaces, we refer the reader to
\cite{mb07,bh06,blyz08,dl08,lby14,lbyy12,lbyy14,t06}.

To give a complete theory of Hardy-Lorentz
spaces and also to establish this theory in a more general antitropic setting,
in this article, we systematically develop a
theory of Hardy-Lorentz spaces associated with  anisotropic dilations $A$.
To be precise, in this article, for all $p\in(0,1]$ and $q\in(0,\infty]$, we
introduce the anisotropic Hardy-Lorentz spaces $H^{p,q}_A(\mathbb{R}^n)$
associated with a general expansive matrix $A$ via the non-tangential grand
maximal function. Then
we characterize  $H^{p,q}_A(\mathbb{R}^n)$
in terms of the atomic or the molecular decompositions, the
radial or the non-tangential maximal functions, or the finite atomic decompositions.
All these results except the $\infty$-atomic characterization are new even
for the classical isotropic Hardy-Lorentz spaces on $\rn$.
As applications, we first
prove that the space $H^{p,q}_A(\mathbb{R}^n)$
is an intermediate space between $H^{p_1,q_1}_A(\mathbb{R}^n)$
and $H^{p_2,q_2}_A(\mathbb{R}^n)$ with
$0<p_1<p<p_2<\infty$ and $q_1,\,q,\,q_2\in(0,\infty]$,
and also between $H^{p,q_1}_A(\mathbb{R}^n)$
and $H^{p,q_2}_A(\mathbb{R}^n)$ with
$p\in(0,\infty)$ and $0<q_1<q<q_2\leq\infty$
in the real method of interpolation.
We then establish a criterion on the boundedness of sublinear
operators from $H^{p,q}_A(\mathbb{R}^n)$ into a quasi-Banach
space. Moreover, we obtain the boundedness of $\delta$-type
Calder\'{o}n-Zygmund operators from $H^p_A(\mathbb{R}^n)$ to
the weak Lebesgue space $L^{p,\infty}(\mathbb{R}^n)$ (or
$H^{p,\infty}_A(\mathbb{R}^n)$) in the critical case,
from $H_A^{p,q}(\mathbb{R}^n)$ to $L^{p,q}(\mathbb{R}^n)$
(or $H_A^{p,q}(\mathbb{R}^n)$) with
$\delta\in(0,\frac{\ln\lambda_-}{\ln b}]$,
$p\in(\frac1{1+\delta},1]$ and $q\in(0,\infty]$, as well as the boundedness of some
Calder\'{o}n-Zygmund operators from $H_A^{p,q}(\mathbb{R}^n)$
to $L^{p,\infty}(\mathbb{R}^n)$.

To be precise, this article is organized as follows.

In Section \ref{s2}, we first present some basic notions and notation
appearing in this article, including Lorentz spaces and  their properties
and also some known facts on expansive dilations in \cite{mb03}.
Then we introduce the anisotropic Hardy-Lorentz spaces
$H^{p,q}_A(\rn)$, with $p\in(0,1]$ and $q\in(0,\fz]$, via the non-tangential maximal
function (see Definition \ref{d-ahls} below). These anisotropic
Hardy-Lorentz spaces  include the classical Hardy
spaces of Fefferman and Stein (\cite{fs72}), the classical Hardy-Lorentz
spaces of Abu-Shammala and Torchinsky  (\cite{wa07}), the anisotropic
Hardy spaces of Bownik  (\cite{mb03}) and the anisotropic weak Hardy
spaces of Ding and Lan (\cite{dl08}) as special cases.
Some basic properties of $H^{p,q}_A(\rn)$ are also obtained in this
section (see Propositions \ref{sp2} and \ref{sp3} below).

Section \ref{s3} is devoted to the atomic and the molecular characterizations of
$H^{p,q}_A(\rn)$. These characterizations of $H^{p,q}_A(\rn)$
are obtained  by using the Calder\'{o}n-Zygmund decomposition associated
with non-tangential grand maximal functions on anisotropic $\rn$ from
\cite[p.\,9, Lemma 2.7]{mb03}, as well as a criterion for affirming some functions
being in Lorentz spaces $L^{p,q}(\rn)$ from \cite[Lemma 1.2]{wa07}.
Recall that, for the classical Hardy-Lorentz spaces $H^{p,q}(\rn)$, only
their $\fz$-atomic characterizations are known (see \cite{wa07}).
Thus, the $r$-atomic characterizations of $H^{p,q}_A(\rn)$, with $r\in(1,\fz)$, presented in
Theorem \ref{tt1} below are new even for the classical Hardy-Lorentz
spaces. We also point out that the molecular characterizations
in Theorem \ref{tt2} below are new even when $p=q$ for the anisotropic Hardy
spaces $H^p_A(\rn)$ with $p\in(0,1]$. Moreover, the approach used to
prove the $r$-atomic characterization in this article is much more complicated than
that for the $\fz$-atomic characterization in \cite{wa07}. Indeed,
in the proof of the $\fz$-atomic characterization,
an $L^\fz(\rn)$ estimate of an infinite combination of $\fz$-atoms can be easily
obtained by the size condition and the finite overlapping property of $\fz$-atoms
(see \cite[p.\,291]{wa07}),
but this approach fails for the corresponding  $L^r(\rn)$ estimate with $r\in(1,\fz)$.
To overcome this difficulty, we employ a distributional estimate (see \eqref{te21} below)
instead of the $L^r(\rn)$ estimate in this article, which relies on some
subtle applications of the boundedness of the grand
maximal function on $L^r(\rn)$ and the finite overlapping property
of $r$-atoms.

In Section \ref{s4}, we characterize $H^{p,q}_A(\rn)$ by means of
the radial or the non-tangential maximal functions (see Theorem \ref{ft1} below).
To this end, via the Aoki-Rolewicz theorem (see \cite{ta42, sr57}),
we first prove that the $L^{p,q}(\rn)$ quasi-norm of the tangential
maximal function $T^{N(K,L)}_\varphi(f)$ can be controlled by that of the
non-tangential maximal function $M^{(K,L)}_\varphi(f)$
for all $f\in\cs'(\rn)$
(see Lemma \ref{fl3} below), where $K$ is the truncation level,
$L$ is the decay level and $\cs'(\rn)$ denotes the set of all tempered distributions
on $\rn$.
Then we obtain the boundedness of the maximal function
$M_\mathcal{F}f$ on $L^{p,q}(\rn)$ with $p\in(1,\fz)$ and $q\in(0,\fz]$
(see Lemma \ref{fl4} below), where $M_\mathcal{F}f$ is defined as in
\eqref{te59} below. As a consequence of Lemma \ref{fl4}, both the non-tangential grand
maximal function  and the
Hardy-Littlewood maximal function given by \eqref{te58} are also bounded
on $L^{p,q}(\rn)$ (see Remark \ref{fr1} below).
We point out that Lemmas \ref{fl3} and \ref{fl4}, Remark \ref{fr1} and the
monotone property of the non-increasing rearrangement
(see \cite[Proposition\ 1.4.5(8)]{lg08})
play a key role in proving Theorem \ref{ft1}.

In Section \ref{s7}, we obtain the finite atomic decomposition
characterizations of $H_A^{p,q}(\rn)$.
In what follows, $C_c^\fz(\rn)$
denotes the space of all smooth functions with compact supports.
For any admissible anisotropic triplet $(p,r,s)$,
via proving that $H_A^{p,q}(\rn)\cap L^r(\rn)$, with $r\in[1,\fz]$,
and $H_A^{p,q}(\rn)\cap C_c^\fz(\rn)$
are dense in $H_A^{p,q}(\rn)$ (see Lemma \ref{sevenl4} below),
we establish the finite atomic decomposition
characterizations of $H_A^{p,q}(\rn)$ (see Theorem \ref{sevent1} below). This
extends \cite[Theorem 3.1 and Remark 3.3]{msv08} and \cite[Theorem 5.6]{gly08}
to the present setting
of anisotropic Hardy-Lorentz spaces.

Section \ref{s6} is devoted to the interpolation of $H^{p,q}_A(\rn)$
and the boundedness of Calder\'on-Zygmund operators.
As an application, in Subsection \ref{s6.1}, we show that $H^{p,q}_A(\rn)$
is an intermediate space between $H^{p_1,q_1}_A(\mathbb{R}^n)$
and $H^{p_2,q_2}_A(\mathbb{R}^n)$ with
$0<p_1<p<p_2<\infty$ and $q_1,\,q,\,q_2\in(0,\infty]$, and also
between $H^{p,q_1}_A(\mathbb{R}^n)$
and $H^{p,q_2}_A(\mathbb{R}^n)$ with
$p\in(0,\infty)$ and $0<q_1<q<q_2\leq\infty$
in the sense of real interpolation (see Theorem \ref{sixt2} below),
whose isotropic version
includes \cite[Theorem 2.5]{wa07} as a special case (see Remark \ref{sixr4}(ii) below).
In Subsection \ref{s6.2}, by using the atomic characterization of $H_A^p(\rn)$,
we first obtain the boundedness of $\dz$-type Calder\'on-Zygmund operators
from $H_A^p(\rn)$
to the weak Lebesgue space $L^{p,\fz}(\rn)$ (or $H_A^{p,\fz}(\rn)$)
in the critical case (see Theorem \ref{sixt3} below).
In this case, even for the classical isotropic setting,
$\dz$-type Calder\'on-Zygmund operators are not bounded from $H^p(\rn)$ to itself.
Moreover, for all $p\in(0,1]$ and $q\in(p,\fz]$, employing the atomic characterizations of
$H^{p,q}_A(\rn)$, we also obtain the boundedness of some Calder\'{o}n-Zygmund
operators from $H_A^{p,q}(\rn)$ to
$L^{p,\fz}(\rn)$ (see Theorem \ref{sixt1} below).
In addition, as an application of the finite atomic decomposition characterizations
for $H_A^{p,q}(\rn)$
(see Theorem \ref{sevent1} below)
obtained in Section \ref{s7}, we establish a criterion on the boundedness of
sublinear operators from $H_A^{p,q}(\rn)$ into a quasi-Banach space
(see Theorem \ref{sevent2} below), which is of independent interest;
by this criterion, we further conclude that,
if $T$ is a sublinear operator and maps all $(p,r,s)$-atoms
with $r\in(1,\fz)$ (or all continuous $(p,\fz,s)$-atoms) into uniformly bounded
elements of some quasi-Banach space $\mathcal{B}$, then $T$ has a unique
bounded sublinear  extension from $H_A^{p,q}(\rn)$ into $\mathcal{B}$
(see Corollary \ref{sevenc1} below).
This extends the corresponding results of Meda et al. \cite{msv08},
Yang-Zhou \cite{yz09} and Grafakos et al. \cite{gly08} to the present setting.
Finally, via the criterion established in Theorem \ref{sevent2}, we also obtain
the boundedness of $\delta$-type
Calder\'{o}n-Zygmund operators
from $H_A^{p,q}(\mathbb{R}^n)$ to $L^{p,q}(\mathbb{R}^n)$ (or $H_A^{p,q}(\mathbb{R}^n)$)
with $\delta\in(0,\frac{\ln\lambda_-}{\ln b}]$,
$p\in(\frac1{1+\delta},1]$ and $q\in(0,\infty]$ (see Theorem \ref{sixt4} below).

We point out that we also obtain the Littlewood-Paley characterizations
of anisotropic Hardy-Lorentz spaces $H^{p,q}_A(\rn)$,
respectively, in terms of the Lusin-area functions,
the Littlewood-Paley $g$-functions or the $g_\lambda^*$-functions;
to restrict the length of this article, we present these characterizations in \cite{lyy15}.
More applications of these anisotropic Hardy-Lorentz spaces $H^{p,q}_A(\rn)$
are expectable.

Finally, we make some conventions on notation. Throughout this article, we always let
$\mathbb{N}:=\{1,2,\ldots\}$ and $\mathbb{Z}_+:=\{0\}\cup\mathbb{N}$.
We denote by $C$ a \emph{positive constant} which is independent of the main
parameters, but its value may change from line to line.
\emph{Constants with subscripts}, such as $C_1$, are the same
in different occurrences. We also use $C_{(\alpha,\beta,\ldots)}$
to denote a positive constant depending on the indicated parameters
$\alpha$, $\beta,\ldots$.
For any multi-index
$\beta:=(\beta_1,\ldots,\beta_n)\in\mathbb{Z}_+^n$,
let
$|\beta|:=\beta_1+\cdots+\beta_n$ and $\partial^\beta:=(\frac\partial{\partial x_1})^{\beta_1}
\cdots(\frac\partial{\partial x_n})^{\beta_n}.$
We use $f\ls g$ to denote $f\leq Cg$ and, if $f\ls g\ls f$,
we then write $f\sim g$. For every index $r\in[1,\fz]$, we use $r'$ to
denote its \emph{conjugate index},
that is, $1/r+1/r'=1$. Moreover, for any set $F\subset\rn$, we denote
by $\chi_F$ its \emph{characteristic function}, by $F^\complement$ the
set $\rn\setminus F$, and by $\sharp F$ the cardinality of $F$.
The symbol $\lfloor s\rfloor$, for any $s\in\rn$, denotes the biggest
integer less than or equal to $s$.

\section{Anisotropic Hardy-Lorentz spaces \label{s2}}

\hskip\parindent
In this section, we introduce the  anisotropic Hardy-Lorentz spaces
via grand maximal functions and give out some basic properties of these spaces.

First we recall the definition of Lorentz spaces. Let $p\in(0,\fz)$
and $q\in(0,\fz]$. The \emph{Lorentz space} $L^{p,q}(\rn)$
is defined to be the space of all measurable functions $f$ with finite $L^{p,q}(\rn)$
quasi-norm $\|f\|_{L^{p,q}(\rn)}$ given by
\begin{eqnarray*}
\|f\|_{L^{p,q}(\rn)}:=\left\{
\begin{array}{cl}
&\lf\{\dis\frac{q}{p}\dis\int_0^\fz\lf
[t^{\frac{1}{p}}{f^\ast(t)}\r]
^q\frac{dt}{t}\r\}^{\frac{1}{q}}
\ \ \ {\rm when}\ \ \ q\in(0,\fz),\\
&\dis\sup_{t\in(0,\fz)}\lf[t^
\frac{1}{p}f^\ast(t)\r]\ \ \ \ \ \ \
\ \ \ \ \ \ \ \,{\rm when}\ \ \ q=\infty,
\end{array}\r.
\end{eqnarray*}
where $f^\ast$ denotes the non-increasing rearrangement
of $f$, namely,
\begin{eqnarray*}
f^\ast(t):=\{\alpha\in(0,\fz):\
d_f(\alpha)\leq t\}, \quad t\in(0,\fz).
\end{eqnarray*}
Here and hereafter, for any $\az\in(0,\fz)$,
$d_f(\alpha):=|\{x\in\rn:\ |f(x)|>\alpha\}|$.
It is well known that, if $q\in(0,\fz)$,
\begin{equation}\label{se6}
\|f\|_{L^{p,q}(\rn)}
\sim\lf\{\int_0^\fz\alpha^{q-1}\lf[d_f(\alpha)\r]
^{\frac{q}{p}}\,d\az\r\}^{\frac{1}{q}}
\sim \lf\{\sum_{k\in\mathbb{Z}}\lf[2^k
\lf\{d_f(2^k)\r\}^{\frac{1}{p}}\r]^q\r\}^{\frac{1}{q}}
\end{equation}
and
\begin{equation}\label{se7}
\|f\|_{L^{p,\infty}(\rn)}
\sim\sup_{\alpha\in(0,\fz)}
\lf\{\alpha\lf[d_f(\alpha)\r]^{\frac1p}\r\}
\sim\sup_{k\in\mathbb{Z}}
\lf\{2^k\lf[d_f(2^k)\r]^{\frac{1}{p}}\r\};
\end{equation}
see \cite{lg08}. By
\cite[Remark\ 1.4.7]{lg08}, for all $p,\,r\in(0,\fz)$, $q\in(0,\fz]$ and all measurable
functions $g$, we have
\begin{eqnarray}\label{se22}
\lf\||g|^r\r\|_{L^{p,q}(\rn)}
=\lf\|g\r\|_{L^{pr,qr}(\rn)}^r.
\end{eqnarray}

Now let us recall the notion of
expansive matrixes in \cite{mb03}.

\begin{definition}\label{d-em}
An \emph{expansive matrix}
(for short, a \emph{dilation})
is an $n\times n$ real matrix $A$
such that all eigenvalues $\lambda$
of $A$ satisfy $|\lambda|>1$.
\end{definition}

Throughout this article,
we always let $A$ be a fixed dilation
and $b:=|\det A|$. By \cite[p.\,6, (2.7)]{mb03},
we know that $b\in(1,\fz)$.
Let $\lambda_-$ and $\lambda_+$
be \emph{positive numbers} such that
$$1<\lambda_-<\min\{|\lambda|:\
\lambda\in\sigma(A)\}
\leq\max\{|\lambda|:\
\lambda\in\sigma(A)\}<\lambda_+,$$
where $\sigma(A)$ denotes the set of all eigenvalues
of $A$. Then there exists a positive constant
$C$, independent of $x$ and $j$, such that, for all $x\in\rn$,
when $j\in\mathbb{Z}_+$,
\begin{eqnarray}\label{se5}
C^{-1}\lf(\lambda_-\r)^j|x|\leq|A^jx|
\leq C\lf(\lambda_+\r)^j|x|
\end{eqnarray}
and, when $j\in\mathbb{Z}\setminus\mathbb{Z}_+$,
\begin{eqnarray}\label{se19}
C^{-1}\lf(\lambda_+\r)^j|x|\leq|A^jx|
\leq C\lf(\lambda_-\r)^j|x|.
\end{eqnarray}
In the case when $A$ is diagonalizable over
$\mathbb{C}$, we can even take
$\lambda_-:=\min\{|\lambda|:\
\lambda\in\sigma(A)\}$
and
$\lambda_+:=\max\{|\lambda|:\
\lambda\in\sigma(A)\}$.
Otherwise, we need to choose them
sufficiently close to these equalities
according to what we need in our arguments.

It was proved in \cite[p.\,5, Lemma 2.2]{mb03} that,
for a given dilation $A$, there exists an open
ellipsoid $\Delta$ and $r\in(1,\infty)$ such that
$\Delta\subset r\Delta\subset A\Delta$, and one
can additionally assume that $|\Delta|=1$,
where $|\Delta|$ denotes the \emph{n-}dimensional
Lebesgue measure of the set $\Delta$.
Let $B_k:=A^k\Delta$ for all $k\in\zz$.
An ellipsoid $x+B_k$ for some $x\in\rn$ and $k\in\mathbb{Z}$
is called a \emph{dilated ball}. Let $\mathfrak{B}$ be the set of all such
dilated balls, namely,
\begin{eqnarray}\label{se14}
\mathfrak{B}:=\lf\{x+B_k:\ x\in\rn,\ k\in\mathbb{Z}\r\}.
\end{eqnarray}
Then $B_k$ is open,
$B_k\subset rB_k\subset B_{k+1}$
and $|B_k|=b^k$.
Throughout this article, let $\tau$ be
the \emph{minimal integer} such that
$r^\tau\geq2$. Then, for all $k\in\mathbb{Z}$,
it holds true that
\begin{eqnarray}\label{se1}
B_k+B_k\subset B_{k+\tau},
\end{eqnarray}
\begin{eqnarray}\label{se2}
B_k+(B_{k+\tau})^\complement
\subset (B_k)^\complement,
\end{eqnarray}
where $E+F$ denotes the algebraic sum
$\{x+y:\ x\in E,\,y\in F\}$ of sets
$E,\,F\subset\mathbb{R^\emph{n}}$.

Define the
\emph{step homogeneous quasi-norm}
$\rho$ on $\rn$
associated to $A$ and $\Delta$ as
\begin{equation}\label{se3}
\rho(x):=\left\{
\begin{array}{cl}
b^j&\ \ \ \ {\rm when}\ \ \ x\in
B_{j+1}\backslash B_j,\\
0&\ \ \ \ {\rm when}\ \ \ x=0.
\end{array}\r.
\end{equation}
Obviously, for all $k\in\zz$,
$B_k=\{x\in\rn:\ \rho(x)<b^k\}$.
By \eqref{se1} and \eqref{se2}, we know
that, for all $x,\,y\in\rn$,
\begin{equation}\label{se4}
\max\{1,\rho(x+y)\}\leq b^\tau
\lf(\max\{1,\rho(x)\}\r)\lf(\max\{1,\rho(y)\}\r)
\end{equation}
and, for all $j\in\zz_+$ and $x\in\rn$,
$\max\{1,\rho(A^jx)\}
\leq b^j\max\{1,\rho(x)\};$
see \cite[p.\,8]{mb03}.
Moreover, $(\rn,\,\rho,\,dx)$ is a space
of homogeneous type in the sense of
Coifman and Weiss \cite{cw71,cw77}, here and hereafter,
$dx$ denotes the \emph{n-}dimensional
Lebesgue measure.

Recall that the homogeneous quasi-norm
induced by $A$ was introduced in
\cite[p.\,6, Definition 2.3]{mb03} as follows.

\begin{definition}\label{sd2}
A \emph{homogeneous quasi-norm} associated
with a dilation $A$ is a measurable mapping
$\rho:\ \rn \rightarrow [0,\infty]$ satisfying that

(i) $\rho(x)=0 \Longleftrightarrow x=\vec0_n$,
here and hereafter, $\vec0_n:=(0,\ldots,0)\in\rn$;

(ii) $\rho(Ax)=b\rho(x)$ for all $x\in\rn$;

(iii) $\rho(x+y)\leq H[\rho(x)+\rho(y)]$
for all $x,\ y\in\rn$, where $H$ is a positive constant
no less than $1$.
\end{definition}

In the standard dyadic case
$A:=2{\rm I}_{n\times n}$, $\rho(x):=|x|^n$
for all $x\in\rn$ is an example of the
homogeneous quasi-norm associated
with $A$, here and hereafter,
${\rm I}_{n\times n}$ denotes the $n\times n$
\emph{unit matrix} and $|\cdot|$ the
Euclidean norm in $\rn$. It was proved in
\cite[p.\,6, Lemma 2.4]{mb03} that all homogeneous
quasi-norms associated with $A$ are equivalent.
Therefore, in what follows, we always use the
step homogeneous quasi-norm induced by the
given dilation $A$ for convenience.

A $C^\infty(\rn)$
function $\varphi$ is said to belong to the Schwartz class
$\cs(\rn)$ if, for every integer $\ell\in\zz_+$ and
multi-index $\alpha$,
$\|\varphi\|_{\alpha,\ell}:=
\sup_{x\in\rn}[\rho(x)]^\ell
|\partial^\alpha\varphi(x)|<\infty$.
The dual space of $\cs(\rn)$, namely, the space of all
tempered distributions on $\rn$ equipped with the weak-$\ast$ topology, is denoted by
$\cs'(\rn)$. For any $N\in\mathbb{Z}_+$, define
$\cs_N(\rn)$ as
$$\cs_N(\rn):=\{\varphi\in\cs(\rn):\
\|\varphi\|_{\alpha,\ell}\leq1,\
|\alpha|\leq N\ ,\ \ell\leq N\};$$
equivalently,
\begin{eqnarray}\label{se24}
\ \ \ \varphi\in\cs_N(\rn)\Longleftrightarrow
\|\varphi\|_{\cs_N(\rn)}:=\sup_{|\alpha|\leq N}
\sup_{x\in\rn}\lf[\lf|\partial^\alpha
\varphi(x)\r|\max\lf\{1,\lf[
\rho(x)\r]^N\r\}\r]\leq1.
\end{eqnarray}
In what follows, for
$\varphi\in\cs(\rn),k\in\mathbb{Z}$ and
$x\in\rn$, let
$\varphi_k(x):=b^{-k}\varphi(A^{-k}x)$.

\begin{definition}\label{d-mf}
Let $\varphi\in\cs(\rn)$ and $f\in\cs'(\rn)$. The
\emph{non-tangential maximal function} $M_\varphi(f)$ of
$f$ with respect to $\varphi$ is defined as
\begin{eqnarray}\label{se25}
M_\varphi(f)(x):= \sup_{y\in x+B_k,
k\in\mathbb{Z}}|f\ast\varphi_k(y)|,
\ \ \ \ \ \forall\ x\in\rn.
\end{eqnarray}
The \emph{radial maximal function} $M_\varphi^0(f)$ of $f$
with respect to $\varphi$ is defined as
\begin{equation}\label{se10}
M_\varphi^0(f)(x):= \sup_{k\in\mathbb{Z}}
|f\ast\varphi_k(x)|,\ \ \ \ \ \forall\ x\in\rn.
\end{equation}
For $N\in\mathbb{N}$, the
\emph{non-tangential grand maximal function} $M_N(f)$
of $f\in\cs'(\rn)$ is defined as
\begin{equation}\label{se8}
M_N(f)(x):=\sup_{\varphi\in\cs_N(\rn)}
M_\varphi(f)(x),\ \ \ \ \ \forall\ x\in\rn
\end{equation}
and the
\emph{radial grand maximal function} $M_N^0(f)$
of $f\in\cs'(\rn)$ is defined as
\begin{equation*}
M_N^0(f)(x):=\sup_{\varphi\in\cs_N(\rn)}
M_\varphi^0(f)(x),\ \ \ \ \ \forall\ x\in\rn.
\end{equation*}
\end{definition}

The following proposition is just
\cite[p.\,17, Proposition 3.10]{mb03}.

\begin{proposition}\label{sp1}
For every given $N\in\mathbb{N}$,
there exists a positive constant $C_{(N)}$,
depending only on $N$, such that,
for all $f\in\cs'(\rn)$ and $x\in\rn$,
\begin{equation*}
M_N^0(f)(x)\leq M_N(f)(x)
\leq C_{(N)}M_N^0(f)(x).
\end{equation*}
\end{proposition}

We now introduce the notion of
anisotropic Hardy-Lorentz spaces.

\begin{definition}\label{d-ahls}
Suppose $p\in(0,\fz),\,q\in(0,\fz]$ and
\begin{eqnarray*}
N_{(p)}:=\left\{
\begin{array}{rl}
&\lf\lfloor\lf(\dfrac1p-1\r)\dfrac{\ln b}{\ln
\lambda_-}\r\rfloor+2\hspace{1cm} {\rm when}\ p\in(0,1],\\
&2\hspace{4.44cm} {\rm when}\ p\in(1,\fz).
\end{array}\r.
\end{eqnarray*}
 For every $N\in\nn\cap(N_{(p)},\fz)$, the
\emph{anisotropic Hardy-Lorentz space} $H^{p,q}_A(\rn)$ is defined by
\begin{equation*}
H_A^{p,q}(\rn)
:=\lf\{f\in\cs'(\rn):\ M_N(f)\in\lpq\r\}
\end{equation*}
and, for any $f\in H^{p,q}_A(\rn)$, let
$\|f\|_{H^{p,q}_A(\rn)}
:=\| M_N(f)\|_{L^{p,q}(\rn)}$.
\end{definition}

\begin{remark}\label{sr1}
Even though the quasi-norm of $H^{p,q}_A(\rn)$ in Definition \ref{d-ahls}
depends on $N$, it follows from
Theorem \ref{tt1} below that the space $H^{p,q}_A(\rn)$
is independent of the choice of $N$ as long as
$N\in\nn\cap(N_{(p)},\fz)$.
\end{remark}

Obviously, when $p=q,\ H^{p,q}_A(\rn)$
becomes the anisotropic
Hardy space $H^p_A(\rn)$ introduced by Bownik in \cite{mb03}
and, when
$q=\infty,\ H^{p,q}_A(\rn)$
is the anisotropic weak Hardy space $H^{p,\fz}_A(\rn)$ investigated
by Ding and Lan in \cite{dl08}.

Now let us give some basic properties of $H^{p,q}_A(\rn)$.

\begin{proposition}\label{sp2}
Let $p\in(0,\fz),\,q\in(0,\fz]$ and $N\in\nn\cap(N_{(p)},\fz)$. Then
$H^{p,q}_A(\rn)\subset\cs'(\rn)$ and the inclusion is continuous.
\end{proposition}

\begin{proof}
Let $f\in H^{p,q}_A(\rn)$. Then, for any $\varphi\in\cs(\rn)$
and $x\in B_0$, we have
\begin{equation}\label{se12}
\beta:=|\langle f,\varphi\rangle|=
|f\ast\wz\varphi(0)|\leq M_{\wz\varphi}(f)(x),
\end{equation}
where $\wz\varphi(\cdot):=\varphi(-\cdot)$ and $M_{\wz\varphi}$
is as in \eqref{se25} with $\varphi$ replaced by $\widetilde{\varphi}$.
Notice that, for $q\in(0,\fz]$, by the definitions of $M_{\wz\varphi}$
and $M_N$,
$$\|M_{\wz\varphi}(f)\|_{L^{p,q}(\rn)}
\leq\|\wz\varphi\|_{\cs_N(\rn)}\|M_N(f)\|
_{L^{p,q}(\rn)}=\|\wz\varphi\|
_{\cs_N(\rn)}\|f\|_{H^{p,q}_A(\rn)}.$$
Thus, to show Proposition \ref{sp2}, it suffices to prove that
$\beta\ls\|M_{\wz\varphi}(f)\|_{L^{p,q}(\rn)}$.

To this end, by \eqref{se12} and \eqref{se6}, for $q\in(0,\fz)$, we have
\begin{eqnarray}\label{se13}
\ \ \ \beta
&&\ls\lf\{\sum_{k\in\mathbb{Z},\,k
<\log_2\bz}2^{kq}\r\}^{\frac1q}
\sim\lf\{\sum_{k\in\mathbb{Z},\,k
<\log_2\bz}2^{kq}\lf|\lf\{x\in B_0:\
\beta\chi_{B_0}(x)>2^k\r\}\r|^
{\frac qp}\r\}^{\frac1q}\\
&&\ls\lf\{\sum_{k\in\mathbb{Z}}2^
{kq}\lf|\lf\{x\in B_0:\ M_{\widetilde
{\varphi}}(f)(x)>2^k\r\}\r|^{\frac qp}
\r\}^{\frac1q}\noz
\ls\|M_{\widetilde{\varphi}}(f)\|
_{L^{p,q}(\rn)}.\noz
\end{eqnarray}
Similar to \eqref{se13}, we conclude that
$\beta\ls\|M_{\widetilde{\varphi}}(f)\|
_{L^{p,\fz}(\rn)}$. This finishes the proof of Proposition \ref{sp2}.
\end{proof}

\begin{proposition}\label{sp3}
For all $p\in(0,\fz),\,q\in(0,\fz]$ and
$N\in\nn\cap(N_{(p)},\fz)$, $H^{p,q}_A(\rn)$ is complete.
\end{proposition}

\begin{proof}
To prove that $H^{p,q}_A(\rn)$ is complete,
it suffices to show that, for any sequence
$\{f_k\}_{k\in\mathbb{N}}\subset H^{p,q}_A(\rn)$ such that
$\|f_k\|_{H^{p,q}_A(\rn)}\leq 2^{-k}$ for
$k\in\mathbb{N}$, the series
$\{\sum_{k=1}^mf_k\}_{m\in\nn}$ converges in
$H^{p,q}_A(\rn)$. Since
$\{\sum_{k=1}^mf_k\}_{m\in\mathbb{N}}$
is a Cauchy sequence in $H^{p,q}_A(\rn)$, from
Proposition \ref{sp2},
it follows that $\{\sum_{k=1}^mf_k\}_{m\in\mathbb{N}}$
is also a Cauchy sequence in $\cs'(\rn)$ which, together with
the completeness of $\cs'(\rn)$, implies that there exists
some $f\in\cs'(\rn)$ such that
$\sum_{k=1}^mf_k$ converges to
$f$ in $\cs'(\rn)$, as $m\to\fz$. Thus,
for any $\varphi\in\cs(\rn)$, the series
$\sum_{k=1}^mf_k\ast\varphi(x)$
converges pointwise to $f\ast\varphi(x)$ for all $x\in\rn$,
as $m\to\fz$.
Therefore, for all $x\in\rn$, we have
$$M_N(f)(x)\leq\sum_{k\in\nn}M_N(f_k)(x).$$
By this and the Aoki-Rolewicz theorem
(see \cite{ta42, sr57}), there exists
$\upsilon\in(0,1]$ such that
\begin{eqnarray*}
\lf\|M_N(f)\r\|_{L^{p,q}(\rn)}^\upsilon
\leq\lf\|\sum_{k\in\mathbb{N}}M_N(f_k)\r\|
_{L^{p,q}(\rn)}^\upsilon
\ls\sum_{k\in\nn}\lf\|M_N(f_k)\r\|_{L^{p,q}(\rn)}^\upsilon.
\end{eqnarray*}
From this, it follows that, for all $m\in\nn$,
\begin{eqnarray*}
\lf\|f-\sum_{k=1}^mf_k\r\|_{H^{p,q}_A(\rn)}
&&=\lf\|\sum_{k=m+1}^\infty f_k\r\|_{H^
{p,q}_A(\rn)}=\lf\|M_N\lf(\sum_{k=m+1}^\infty
f_k\r)\r\|_{L^{p,q}(\rn)}\\
&&\ls\lf[\sum_{k=m+1}^\infty\|M_N(f_k)\|_
{L^{p,q}(\rn)}^\upsilon\r]^{\frac 1\upsilon}\\
&&\ls\lf(\sum_{k=m+1}^\infty2^{-k\upsilon}\r)^
{\frac 1\upsilon}\sim2^{-m}\to0,\ \ \ {\rm as}\ m\to\fz.
\end{eqnarray*}
Thus, $\sum_{k=1}^mf_k\rightarrow f$ in $H^{p,q}_A(\rn)$,
as $m\rightarrow\infty$. This finishes the proof of Proposition \ref{sp3}.
\end{proof}

The following Proposition \ref{tl4}
is just \cite[p.\,13, Theorem 3.6]{mb03}.

\begin{proposition}\label{tl4}
For any given $s\in(1,\fz)$, let
$$\mathcal{F}:=
\lf\{\varphi\in L^\infty(\rn):\ |\varphi(x)|
\leq[1+\rho(x)]^{-s},\ x\in\rn\r\}.$$
For $p\in[1,\fz]$ and $f\in L^p(\rn)$,
define the maximal function associated with $\mathcal{F}$,
$M_{\mathcal{F}}$, by
\begin{eqnarray}\label{te59}
M_{\mathcal{F}}(f)(x):=
\sup_{\varphi\in\mathcal{F}}M_\varphi(f)(x),
\ \ \ \ \ \ \forall\ x\in\rn.   \end{eqnarray}
Then there exists a positive constant
$C_{(s)}$, depending on $s$, such that,
for all $\lz\in(0,\fz)$ and $f\in L^1(\rn)$,
\begin{eqnarray}\label{te56}
\lf|\lf\{x\in\rn:\ M_{\mathcal{F}}(f)(x)>\lambda\r\}\r|
\leq C_{(s)}\|f\|_{L^1(\rn)}/\lambda
\end{eqnarray}
and, for all $p\in(1,\fz]$ and $f\in L^p(\rn)$,
\begin{eqnarray}\label{te57}
\|M_{\mathcal{F}}(f)\|_{L^p(\rn)}
\leq\frac{C_{(s)}}{1-1/p}\|f\|_{L^p(\rn)}.
\end{eqnarray}
\end{proposition}

\begin{remark}\label{tr1}
Clearly, by Proposition \ref{tl4}, we know that the non-tangential grand
maximal function $M_N(f)$, defined in \eqref{se8}, and the
Hardy-Littlewood maximal function $M_{{\rm HL}}(f)$, defined by setting,
for all $f\in L^1_{{\rm loc}}(\rn)$ and $x\in\rn$,
\begin{eqnarray}\label{te58}
M_{{\rm HL}}(f)(x):=\sup_{k\in\mathbb{Z}}
\sup_{y\in x+B_k}\frac1{|B_k|}
\int_{y+B_k}|f(z)|\,dz=\sup_{x\in B\in\mathfrak{B}}
\frac1{|B|}\int_B|f(z)|\,dz,
\end{eqnarray}
where $\mathfrak{B}$ is as in \eqref{se14}, satisfy \eqref{te56} and \eqref{te57}.
\end{remark}

\section{Atomic and molecular characterizations of $H^{p,q}_A(\rn)$\label{s3}}

\hskip\parindent
In this section, we establish the atomic and
the molecular characterizations of $H^{p,q}_A(\rn)$.

\subsection{Atomic characterizations of $H^{p,q}_A(\rn)$}\label{s3.1}

\hskip\parindent
In this subsection, by using the Calder\'{o}n-Zygmund
decomposition associated with the non-tangential grand maximal
function on anisotropic $\rn$ established in \cite{mb03},
we obtain the atomic characterizations of $H^{p,q}_A(\rn)$.

We begin with the following notion of anisotropic $(p,r,s)$-atoms
from \cite[p.\,19, Definition 4.1]{mb03}.

\begin{definition}\label{d-at}
An anisotropic triplet $(p,r,s)$ is said to be \emph{admissible}
if $p\in(0,1],\,r\in(1,\fz]$ and $s\in\mathbb{N}$ with
$s\geq\lfloor(1/p-1)\ln b/\ln\lambda_-\rfloor$.
For an admissible anisotropic triplet $(p,r,s)$, a measurable
function $a$ on $\rn$ is called an \emph{anisotropic $(p,r,s)$-atom}
if

(i) ${\rm \supp}a \subset B,\ {\rm where}\
B\in\mathfrak{B}$ and $\mathfrak{B}$ is as in \eqref{se14};

(ii) $\|a\|_{L^r(\rn)}\leq |B|^{1/r-1/p}$;

(iii) $\int_\rn a(x)x^\alpha\,dx =0$ for any
$\alpha\in\zz_+^n$ with $|\alpha|\leq s$.
\end{definition}

Throughout this article, we call anisotropic
$(p,r,s)$-atom simply by $(p,r,s)$-atom.
Now, via $(p,r,s)$-atoms, we give the definition
of the anisotropic atomic Hardy-Lorentz space
$H_A^{p,r,s,q}(\rn)$ as follows.

\begin{definition}\label{d-aahls}
For an anisotropic triplet $(p,r,s)$ as in Definition \ref{d-at},
$q\in(0,\fz]$ and a dilation
$A$, the \emph{anisotropic atomic Hardy-Lorentz space}
$H_A^{p,r,s,q}(\rn)$ is defined to be the set of all distributions
$f\in\cs'(\rn)$ satisfying that there exist a sequence of $(p,r,s)$-atoms,
$\{a_i^k\}_{i\in\mathbb{N},\,k\in\mathbb{Z}}$,
supported on
$\{x_i^k+B_i^k\}_{i\in\mathbb{N},\,k\in\mathbb{Z}}\subset\mathfrak{B}$,
respectively,
and a positive constant $\widetilde{C}$ such that
$\sum_{i\in\mathbb{N}}\chi_{x_i^k+B_i^k}(x)\leq \widetilde{C}$
for all $x\in\rn$ and $k\in\mathbb{Z}$, and
$f=\sum_{k\in\mathbb{Z}}\sum_{i\in\mathbb{N}}\lambda_i^ka_i^k$
in $\cs'(\rn)$, where $\lambda_i^k\sim2^k|B_i^k|^{1/p}$ for all
$k\in\mathbb{Z}$ and $i\in\mathbb{N}$ with the implicit equivalent
positive constants independent of $k$ and $i$.

Moreover, for all $f\in H_A^{p,r,s,q}(\rn)$, define
\begin{eqnarray*}
\|f\|_{H_A^{p,r,s,q}(\rn)}:=
{\rm inf}\lf\{\lf[\sum_{k\in\zz}\lf(\sum_{i\in\nn}|\lambda_i^k|^p\r)^
{\frac qp}\r]^{\frac 1q}:\ f=\sum_{k\in\mathbb{Z}}
\sum_{i\in\mathbb{N}}\lambda_i^ka_i^k\r\}
\end{eqnarray*}
with the usual interpretation for $q=\fz$, where the
infimum is taken over all decompositions of $f$ as above.
\end{definition}

In order to establish the atomic decomposition of
$H^{p,q}_A(\rn)$, we need the following several technical lemmas,
which are
\cite[p.\,9, Lemma 2.7, p.\,19, Theorem 4.2]{mb03}
and \cite[Lemma 1.2]{wa07}, respectively.

\begin{lemma}\label{tl1}
Suppose that $\Omega\subset\rn$ is open and $|\Omega|<\infty$.
For any $d\in\zz_+$, there exist a sequence of points,
$\{x_j\}_{j\in\mathbb{N}}\subset\Omega$, and a sequence of
integers, $\{\ell_j\}_{j\in\mathbb{N}}\subset\mathbb{Z}$,
such that
\vspace{-0.25cm}
\begin{enumerate}
\item[{\rm(i)}] $\Omega=\bigcup_{j\in\mathbb{N}}
(x_j+B_{\ell_j})$;
\vspace{-0.25cm}
\item[{\rm(ii)}] $\lf\{x_j+B_{\ell_j-\tau}\r\}_{j\in\nn}$ are pairwise
disjoint, where $\tau$ is as in \eqref{se1} and \eqref{se2};
\vspace{-0.25cm}
\item[{\rm(iii)}] for every $j\in\mathbb{N},\
(x_j+B_{\ell_j+d})\cap\Omega^\complement
=\emptyset$, but $(x_j+B_{\ell_j+d+1})\cap
\Omega^\complement\neq\emptyset$;
\vspace{-0.25cm}
\item[{\rm(iv)}] $if\ (x_i+B_{\ell_i+d-2\tau})\cap
(x_j+B_{\ell_j+d-2\tau})\neq\emptyset$, then
$|\ell_i-\ell_j|\leq\tau$;
\vspace{-0.25cm}
\item[{\rm(v)}] for all $i\in\mathbb{N}$,
$\sharp\{j\in\mathbb{N}:\ (x_i+B_
 {\ell_i+d-2\tau})\cap(x_j+B_{\ell_j+d-2\tau})
 \neq\emptyset\}\leq L$,
where $L$ is a positive constant
independent of $\Omega$, $f$ and $i$.
\end{enumerate}
\end{lemma}

\begin{lemma}\label{tl2}
Let $(p,r,s)$ be an admissible anisotropic triplet as
in Definition \ref{d-at} and $N\in\nn\cap(N_{(p)},\fz)$. Then there exists
a positive constant $C$, depending only on $p$ and $r$,
such that, for all $(p,r,s)$-atoms $a$,
$\|M_N(a)\|_{L^p(\rn)}\leq C.$
\end{lemma}

\begin{lemma}\label{tl3}
Suppose that $p\in(0,\infty)$, $q\in(0,\fz]$,
$\{\mu_k\}_{k\in\mathbb{Z}}$ is a non-negative sequence
of complex numbers such that
$\{2^k\mu_k\}_{k\in\mathbb{Z}}\in\ell^q$ and
$\varphi$ is a non-negative function having the following
property: there exists $\delta\in(0,\min\{1, q/p\})$ such that,
for any $k_0\in\mathbb{N}$,
$\varphi\leq\psi_{k_0}+\eta_{k_0}$, where $\psi_{k_0}$
and $\eta_{k_0}$ are functions, depending on $k_0$ and satisfying
\begin{eqnarray*}
2^{k_0p}\lf[d_{\psi_{k_0}}(2^{k_0})\r]^\delta
\leq\widetilde{C}\sum_{k=-\infty}^{k_0-1}
\lf[2^k(\mu_k)^\delta\r]^p,\ \ \ 2^{k_0\delta p}
d_{\eta_{k_0}}(2^{k_0})\leq\widetilde{C}
\sum_{k=k_0}^\infty\lf[2^{k\delta}\mu_k\r]^p
\end{eqnarray*}
for some positive constant $\widetilde{C}$
independent of $k_0$. Then
$\varphi\in L^{p,q}(\rn)$ and
$$\|\varphi\|_{L^{p,q}(\rn)}\leq
C\|\{2^k\mu_k\}_{k\in\mathbb{Z}}\|_{\ell^q},$$
where $C$ is a positive constant independent of
$\varphi$ and $\{\mu_k\}_{k\in\mathbb{Z}}$.
\end{lemma}

Now, it is a position to state the main result of this subsection.

\begin{theorem}\label{tt1}
Let $(p,r,s)$ be an admissible anisotropic triplet as in
Definition \ref{d-at}, $q\in(0,\fz]$ and $N\in\nn\cap(N_{(p)},\fz)$. Then
$H^{p,q}_A(\rn)=H_A^{p,r,s,q}(\rn)$ with equivalent quasi-norms.
\end{theorem}

\begin{proof}
First, we show $H^{p,q}_A(\rn)\subset H_A^{p,r,s,q}(\rn)$. Observe that, by
Definition \ref{d-at}, for any $r\in(1,\fz)$, a $(p,\infty,s)$-atom
is also a $(p,r,s)$-atom and hence
$H_A^{p,\infty,s,q}(\rn)\subset H_A^{p,r,s,q}(\rn)$.
Thus, to prove Theorem \ref{tt1},
we only need to show that
\begin{eqnarray}\label{3.5x}
H^{p,q}_A(\rn)\subset H_A^{p,\infty,s,q}(\rn).
\end{eqnarray}

Now we prove \eqref{3.5x} by three steps.

\emph{Step 1.}
To show \eqref{3.5x}, for any
$f\in H^{p,q}_A(\rn),$ $\varphi\in\cs(\rn)$ with $\int_\rn\varphi(x)\,dx=1$,
and $m\in\mathbb{N}$,
let $f^{(m)}:=f\ast\varphi_{-m}$. Then, by \cite[p.\,15, Lemma 3.8]{mb03},
we have $f^{(m)}\rightarrow f$ \ in $\cs'(\rn)$ as $m\rightarrow\infty$.
Moreover, by \cite[p.\,39, Lemma 6.6]{mb03}, we know that, for all $m\in\mathbb{N}$
and $x\in\rn$,
\begin{eqnarray}\label{te11}
M_{N+2}(f^{(m)})(x)\leq C_{(N,\varphi)} M_N(f)(x),
\end{eqnarray}
where $ C_{(N,\varphi)}$ is a positive constant depending
on $N$ and $\varphi$, but independent of $f$.
Therefore, $f^{(m)}\in H^{p,q}_A(\rn)$
and $\|f^{(m)}\|_{H^{p,q}_A(\rn)}\ls\|f\|_{H^{p,q}_A(\rn)}$
with the implicit positive constant independent of $m$ and $f$.

In what follows of this step, we show that, for any $m\in\nn$,
\begin{eqnarray}\label{te1}
f^{(m)}=\sum_{k\in\zz}\sum_{i\in\nn}h_i^{m,k}
\ \ \ \ \ {\rm in}\ \ \cs'(\rn),
\end{eqnarray}
where, for all $m,\,i\in\nn$ and $k\in\zz$, $h_i^{m,k}$
is a $(p,\fz,s)$-atom multiplied by a constant
depending on $k$ and $i$ but, independent of $f$ and $m$.

To show \eqref{te1}, we borrow some ideas from the
proof of \cite[p.\,38 Theorem 6.4]{mb03}.
For $k\in\mathbb{Z}$ and $N\in\nn\cap(N_{(p)},\fz)$,
let $\Omega_k:=\{x\in\rn:\ M_N(f)(x)>2^k\}$.
Then $\Omega_k$ is open. Applying Lemma \ref{tl1}
to $\Omega_k$ with $d=6\tau$, we obtain a sequence
$\{x_i^k\}_{i\in\mathbb{N}}\subset\Omega_k$ and
a sequence of integers, $\{\ell_i^k\}_{i\in\mathbb{N}}$,
satisfying, with $\tau$ and $L$ same as in Lemma \ref{tl1},
\begin{eqnarray}\label{te22}
\Omega_k=\bigcup_{i\in\mathbb{N}}(x_i^k+B_{\ell_i^k});
\end{eqnarray}
\begin{eqnarray}\label{te2}
(x_i^k+B_{\ell_i^k-\tau})\cap(x_j^k+B_{\ell_j^k-\tau})
=\emptyset\ \ \ {\rm for\ all}\  i,\ j\in\nn\ {\rm with}\ i\neq j;
\end{eqnarray}
\begin{eqnarray*}
(x_i^k+B_{\ell_i^k+6\tau})\cap\Omega_k^\complement
=\emptyset,\ \ \ (x_i^k+B_{\ell_i^k+6\tau+1})\cap
\Omega_k^\complement\neq\emptyset\ \ \ {\rm for\ all}\
i\in\mathbb{N};
\end{eqnarray*}
\begin{eqnarray*}
{\rm if}\ (x_i^k+B_{\ell_i^k+4\tau})\cap(x_j^k+B_{\ell_j^k+
4\tau})\neq\emptyset,\ \ {\rm then}\ |\ell_i^k-\ell_j^
k|\leq\tau;
\end{eqnarray*}
\begin{eqnarray}\label{te5}
\sharp\lf\{j\in\mathbb{N}:\
(x_i^k+B_{\ell_i^k+4\tau})\cap(x_j^k+B_{\ell_j^k+
4\tau})\neq\emptyset\r\}\leq L\ \ {\rm for\ all}\  i\in\mathbb{N}.
\end{eqnarray}

Fix $\theta\in\cs(\rn)$ such that $\supp\theta\subset B_\tau$,
$0\leq\theta\leq1$, and $\theta\equiv1$ on $B_0$. For
each $i\in\mathbb{N}$, $k\in\mathbb{Z}$ and all $x\in\rn$,
define
$\theta_i^k(x):=\theta(A^{-\ell_i^k}(x-x_i^k))$
and
$$\zeta_i^k(x):=
\frac {\theta_i^k(x)}{\Sigma_{j\in\nn}\theta_j^k(x)}.$$
Then $\zeta_i^k\in\cs(\rn)$,
$\supp \zeta_i^k\subset x_i^k+B_{\ell_i^k+\tau}$,
$0\leq\zeta_i^k\leq1$, $\zeta_i^k\equiv1$
on $x_i^k+B_{\ell_i^k-\tau}$ by \eqref{te2}, and
$\sum_{i\in\mathbb{N}}\zeta_i^k=\chi_{\Omega_k}$.
Therefore, the family $\{\zeta_i^k\}_{i\in\mathbb{N}}$
forms a smooth partition of unity on $\Omega_k$.

For $\ell\in[0,\fz)$, let $\mathcal{P}_{\ell}(\rn)$ denote
the linear space of all polynomials on $\rn$ with
degree not more than $\ell$. For each $i$ and
$P\in\mathcal{P}_{\ell}(\rn)$, define
\begin{eqnarray}\label{te14}
\|P\|_{i,k}:=
\lf[\frac1{\int_\rn\zeta_i^k(x)\,dx}\int_\rn
|P(x)|^2\zeta_i^k(x)\,dx\r]^{1/2},
\end{eqnarray}
which induces a finite dimensional Hilbert space
$(\mathcal{P}_{\ell}(\rn), \|\cdot\|_{i,k})$. For each $i$,
via
$$Q\mapsto\frac1{\int_\rn\zeta_i^k(x)\,dx}\lf\langle
f^{(m)},Q\zeta_i^k\r\rangle,\ \ \ \ Q\in\mathcal{P}_{\ell}(\rn),$$
the function $f^{(m)}$ induces a linear bounded functional on
$\mathcal{P}_{\ell}(\rn)$. By the Riesz lemma, there exists
a unique polynomial  $P_i^{m,k}\in\mathcal{P}_{\ell}(\rn)$
such that, for all $Q\in\mathcal{P}_{\ell}(\rn)$,
\begin{eqnarray*}
\frac1{\int_\rn\zeta_i^k(x)\,dx}
\lf\langle f^{(m)},Q\zeta_i^k\r\rangle
&&=\frac1{\int_\rn\zeta_i^k(x)\,dx}
\lf\langle P_i^{m,k},Q\zeta_i^k\r\rangle\\
&&=\frac 1{\int_\rn\zeta_i^k(x)\,dx}
\int_\rn P_i^{m,k}(x)Q(x)\zeta_i^k(x)\,dx.
\end{eqnarray*}
For every $i\in\mathbb{N}$ , $k\in\mathbb{Z}$ and
$m\in\mathbb{N}$, define a distribution
$b_i^{m,k}:=[f^{(m)}-P_i^{m,k}]\zeta_i^k$.
From the fact that, for all $k\in\mathbb{Z}$ and $x\in\rn$,
$\sum_{i\in\mathbb{N}}
\chi_{x_i^k+B_{\ell_i^k+4\tau}}(x)\leq L$
and
$\supp b_i^{m,k}\subset x_i^k+B_{\ell_i^k+4\tau}$,
it follows that $\{\sum_{i=1}^Ib_i^{m,k}\}_{I\in\nn}$
converges in $\cs'(\rn)$. Let
$$g^{m,k}:=f^{(m)}-\sum_{i\in\mathbb{N}}b_i^{m,k}
=f^{(m)}-\sum_{i\in\mathbb{N}}\lf[f^{(m)}-P_i^{m,k}\r]
\zeta_i^k =f^{(m)}\chi_{\Omega_k^\complement}+
\sum_{i\in\mathbb{N}}P_i^{m,k}\zeta_i^k.$$
Notice that, for any $k\in\zz$ and $x\in\rn$, the number $i$
satisfying $\zeta_i^k(x)\neq0$ is less than $L$, where $L$
is the same as in \eqref{te5}. Therefore, by a proof similar to that of
\cite[p.\,25, Lemma 5.3]{mb03}, we easily conclude that,
for all $x\in\rn$,
$|\sum_{i\in\mathbb{N}}P_i^{m,k}(x)\zeta_i^k(x)|\ls2^k$.
Clearly, by \eqref{te11}, for all $x\in\rn$, we have
$$\lf|f^{(m)}(x)\chi_{\Omega_k^\complement}(x)\r|\ls
M_N(f)(x)\chi_{\Omega_k^\complement}(x)\ls2^k.$$
Thus,
$\|g^{m,k}\|_{L^{\infty}(\rn)}\ls2^k$ and
$\|g^{m,k}\|_{L^{\infty}(\rn)}\rightarrow0$
as $k\rightarrow-\infty$.

Following the proof of \cite[p.\,31, Lemma 5.7]{mb03}, for any
$k\in\mathbb{Z}$ and $p_0\in(0,p)$ satisfying
$\lfloor(1/p_0-1)\ln b/\ln\lambda_-\rfloor\leq s$, we obtain
\begin{eqnarray}\label{te12}
\int_\rn \lf[M_N\lf(\sum_{i\in\mathbb{N}}
b_i^{m,k}\r)(x)\r]^{p_0}\,dx
\ls\int_{\widetilde{\Omega}_k}
\lf[M_N(f^{(m)})(x)\r]^{p_0}\,dx,
\end{eqnarray}
where
$\widetilde{\Omega}_k:=\{x\in\rn:\ M_N(f^{(m)})(x)>2^k\}$. Since
$f^{(m)}\in H^{p,q}_A(\rn)$, it follows that there exists an integer $k_0$ such that,
for any $k\in[k_0,\fz)\cap\mathbb{Z}$, $|\widetilde{\Omega}_k|<\infty$. Noticing that,
for all $\alpha\in(0,\fz)$,
\begin{eqnarray*}\lf|\widetilde{\Omega}_k
\cap\lf\{x\in\rn:\ M_N(f^{(m)})(x)>\alpha\r\}\r|\leq\min
\lf\{|\widetilde{\Omega}_k|, \alpha^{-p}
\lf\|M_N(f^{(m)})\r\|_{L^{p,\infty}(\rn)}^p\r\},
\end{eqnarray*}
we have
\begin{eqnarray}\label{te13}
&&\int_{\widetilde{\Omega}_k}\lf[M_N(f^{(m)})(x)\r]^{p_0}\,dx\\
&&\quad=\int_0^\infty
p_0\alpha^{p_0-1}\lf|\{x\in \widetilde{\Omega}_k:\
M_N(f^{(m)})(x)>\alpha\}\r|\,d\alpha\noz\\
&&\quad\leq\int_0^\gamma p_0|\widetilde{\Omega}_k
|\alpha^{p_0-1}\,d\alpha+\int_\gamma^\infty p_0\alpha^
{p_0-p-1}\lf\|M_N(f^{(m)})\r\|_{L^{p,\infty}(\rn)}^p\,d\alpha\noz\\
&&\quad=\frac p{p-p_0}|\widetilde{\Omega}_k|^{1-p_0/p}
\lf\|M_N(f^{(m)})\r\|_{L^{p,\infty}(\rn)}^{p_0}
\ls |\widetilde{\Omega}_k|^{1-p_0/p}
\lf\|M_N(f^{(m)})\r\|_{L^{p,q}(\rn)}^{p_0},\noz
\end{eqnarray}
where
$\gamma:=\frac {\|M_N(f^{(m)})\|_{L^{p,\infty}
(\rn)}}{|\widetilde{\Omega}_k|^{1/p}}$.
By \eqref{te12} and \eqref{te13}, we find that
\begin{eqnarray*}
\lf\|\sum_{i\in\mathbb{N}}b_i^{m,k}\r\|_{H^{p_0}_A(\rn)}
:=&&\lf\|M_N\lf(\sum_{i\in\mathbb{N}}b_i^{m,k}\r)\r\|_
{L^{p_0}(\rn)}\\
\ls&&|\widetilde{\Omega}_k|^{\frac1{p_0}-\frac1p}
\|M_N(f^{(m)})\|_{L^{p,q}(\rn)}\rightarrow0,\ \ {\rm as}
\ \ k\rightarrow\infty,
\end{eqnarray*}
where $H^{p_0}_A(\rn)$ denotes the anisotropic Hardy space
introduced by Bownik in \cite{mb03}. From the above
estimates, we further deduce that
\begin{eqnarray*}
&&\lf\|f^{(m)}-\sum_{k=-N}^N\lf(g^{m,k+1}-g^{m,k}\r)\r\|_
{H^{p_0}_A(\rn)+L^\infty(\rn)}\\
&&\hs\ls\lf\|\sum_{i\in\nn}
b_i^{m,N+1}\r\|_{H^{p_0}_A(\rn)}+\lf\|g^{m,-N}\r\|_
{L^\infty(\rn)}\rightarrow0,
\end{eqnarray*}
as $N\rightarrow\infty$.
Here, for any $f\in H^{p_0}_A(\rn)+L^\infty(\rn)$, let
\begin{eqnarray*}
\lf\|f\r\|_{H^{p_0}_A(\rn)+L^\infty(\rn)}:=\inf
\lf\{\r.&&\|f_1\|_{H^{p_0}_A(\rn)}+\|f_2\|_{L^\infty(\rn)}:\\\
&&\lf.f=f_1+f_2,\ f_1\in H^{p_0}_A(\rn),\ f_2\in L^\infty(\rn)\r\},
\end{eqnarray*}
where the infimum is taken over all decompositions of $f$
as above. Therefore,
\begin{eqnarray}\label{te3}
f^{(m)}=\sum_{k=-\infty}^\infty\lf(g^
{m,k+1}-g^{m,k}\r)\ \ \ \ {\rm in}\ \ \cs'(\rn).
\end{eqnarray}

Moreover,
for $i\in\mathbb{N}, k\in\mathbb{Z}$ and $j\in\mathbb{N}$,
define a polynomial $P_{i,j}^{m,k+1}$ as the orthogonal
projection of
$[f^{(m)}-P_j^{m,k+1}]\zeta_i^k$ on $\mathcal{P}_{\ell}(\rn)$ with
respect to the norm defined by \eqref{te14}, namely,
$P_{i,j}^{m,k+1}$ is the unique element of $\mathcal{P}_{\ell}(\rn)$
satisfying, for all $Q\in\mathcal{P}_{\ell}(\rn)$,
$$\int_\rn \lf[f^{(m)}(x)-P_j^{m,k+1}(x)\r]\zeta_i^k(x)Q(x)
\zeta_j^{k+1}(x)\,dx=\int_\rn P_{i,j}^{m,k+1}(x)Q(x)
\zeta_j^{k+1}(x)\,dx.$$
By an argument parallel to the  proof of \cite[p.\,37, Lemma 6.3]{mb03},
we find that
$$\sum_{j\in\mathbb{N}}\sum_{i\in\mathbb{N}}
P_{j,i}^{m,k+1}\zeta_i^{k+1}=0.$$
Then, for $k\in\mathbb{Z}$, by the facts that
$\sum_{j\in\mathbb{N}}\zeta_j^k=\chi_{\Omega_k}$,
$$\sum_{i\in\mathbb{N}}b_i^{m,k+1}
=f^{(m)}\chi_{\Omega_{k+1}}-\sum_{i\in\nn}
P_i^{m,k+1}\zeta_i^{k+1},\ \ \ \supp\lf(
\sum_{i\in\mathbb{N}}P_i^{m,k+1}
\zeta_i^{k+1}\r)\subset\Omega_{k+1}$$
and $\Omega_{k+1}\subset\Omega_{k}$, we have
\begin{eqnarray}\label{te4}
g^{m,k+1}-g^{m,k}
&&=\lf[f^{(m)}-\sum_{i\in\mathbb{N}}b_i^{m,k+1}\r]-
\lf[f^{(m)}-\sum_{j\in\mathbb{N}}b_j^{m,k}\r]\\
&&=\sum_{j\in\mathbb{N}}b_j^{m,k}-\sum_{j\in\nn}
\sum_{i\in\mathbb{N}}b_i^{m,k+1}\zeta_j^k+\sum_
{j\in\mathbb{N}}\sum_{i\in\mathbb{N}}P_{j,i}^
{m,k+1}\zeta_i^{k+1}\noz\\
&&=\sum_{i\in\mathbb{N}}\lf[b_i^{m,k}-\sum_{j\in\nn}
\lf(b_j^{m,k+1}\zeta_i^k-P_{i,j}^{m,k+1}\zeta_j^
{k+1}\r)\r]=:\sum_{i\in\mathbb{N}}h_i^{m,k},\noz
\end{eqnarray}
where all the series converge in $\cs'(\rn)$. Furthermore,
for all $i\in\mathbb{N}$ and $k\in\mathbb{Z}$,
$$h_i^{m,k}=\lf[f^{(m)}-P_i^{m,k}\r]\zeta_i^k-
\sum_{j\in\mathbb{N}}\lf\{\lf[f^{(m)}-P_j^{m,k+1}\r]
\zeta_i^k-P_{i,j}^{m,k+1}\r\}\zeta_j^{k+1}.$$
By the definitions of $P_i^{m,k}$ and $P_{i,j}^{m,k+1}$,
we know that
\begin{eqnarray}\label{te15}
\int_\rn h_i^{m,k}(x)Q(x)\,dx=0\ \ \ \ \ \
{\rm for\ all}\ Q\in\mathcal{P}_{\ell}(\rn).
\end{eqnarray}
In addition, recall that $P_{i,j}^{m,k+1}\neq0$ implies
$(x_j^{k+1}+B_{\ell_j^{k+1}+\tau})
\cap (x_i^k+B_{\ell_i^k+\tau})\neq\emptyset$.
Then, by a proof similar to that of \cite[p.\,35, Lemma 6.1(i)]{mb03},
we find that
$$\supp\zeta_j^{k+1}\subset (x_j^{k+1}+B_{\ell_j^
{k+1}+\tau})\subset(x_i^k+B_{\ell_i^k+4\tau}).$$
Therefore,
\begin{eqnarray}\label{te16}
\supp h_i^{m,k}\subset (x_i^k+B_{\ell_i^k+4\tau}).
\end{eqnarray}
Since $\sum_{j\in\mathbb{N}}\zeta_j^{k+1}
=\chi_{\Omega_{k+1}}$, it follows that
\begin{eqnarray}\label{te17}
h_i^{m,k}=\zeta_i^kf^{(m)}\chi_{\Omega_{k+1}^
\complement}-P_i^{m,k}\zeta_i^k+\zeta_i^k\sum
_{j\in\nn}P_j^{m,k+1}\zeta_j^{k+1}+\sum_
{j\in\mathbb{N}}P_{i,j}^{m,k+1}\zeta_j^{k+1}.
\end{eqnarray}
By a proof similar to that of \cite[p.\,35, Lemma 6.1(ii) and p.\,36, Lemma 6.2]{mb03},
we find that, for $j\in\mathbb{N}$,
$$\sharp\lf\{i\in\mathbb{N}:\ (x_j^{k+1}+B_{\ell_j^{k+1}
+\tau})\cap (x_i^k+B_{\ell_i^k+\tau})\neq\emptyset\r\}\ls1$$
and, for $i,j,m\in\mathbb{N},\ k\in\mathbb{Z}$,
$$\sup_{x\in\rn}\lf|P_{i,j}^{m,k+1}(x)
\zeta_j^{k+1}(x)\r|\ls2^{k+1},$$
which, combined with
$\sup_{x\in\rn}|P_i^{m,k}(x)\zeta_i^k(x)|\ls2^k$,
$\|f^{(m)}\chi_{\Omega_{k+1}^\complement}\|
_{L^{\infty}(\rn)}\ls2^k$ and \eqref{te17},
further implies that, for all $i,\,m\in\mathbb{N}$
and $k\in\mathbb{Z}$,
\begin{eqnarray}\label{te18}
\lf\|h_i^{m,k}\r\|_{L^{\infty}(\rn)}\ls2^k.
\end{eqnarray}

By \eqref{te15}, \eqref{te16} and \eqref{te18}, we know that,
for all $k\in\zz$ and $m$, $i\in\nn$, $h_i^{m,k}$
is a multiple of a $(p,\fz,s)$-atom, which, together with
\eqref{te3} and \eqref{te4}, implies that \eqref{te1} holds true.

\emph{Step 2.} By \eqref{te18} and the Alaoglu theorem (see, for example,
\cite[Theorem 3.17]{wr91}), there exists a subsequence
$\{m_\iota\}_{\iota=1}^{\fz}\subset\mathbb{N}$
such that, for every $i\in\mathbb{N}$ and $k\in\mathbb{Z}$,
$h_i^{m_\iota,k}\rightarrow h_i^k$ weak-$\ast$ in
$L^{\fz}(\rn)$ as $\iota\rightarrow\fz$. It is easy to see that
$\supp h_i^k\subset (x_i^k+B_{\ell_i^k+4\tau})$,
$\|h_i^k\|_{L^{\infty}(\rn)}\ls2^k$ and
$\int_\rn h_i^k(x)Q(x)dx=0$ for all $Q\in\mathcal{P}_{\ell}(\rn)$.
Thus, $h_i^k$ is a multiple of a $(p,\infty,s)$-atom $a_i^k$.
Let $h_i^k:=\lambda_i^ka_i^k$, where
$\lambda_i^k\sim 2^k|B_{\ell_i^k+4\tau}|^{1/p}$. Then,
by \eqref{te5}, \eqref{te22} and \eqref{se6}, for $q\in(0,\fz)$, we have
\begin{eqnarray}\label{te65}
\sum_{k=-\infty}^\infty\lf(\sum_{i\in\mathbb{N}}
|\lambda_i^k|^p\r)^{\frac qp}
&&\sim\sum_{k=-\infty}^\infty\lf(\sum_{i\in\nn}2^
{kp}\lf|B_{\ell_i^k+4\tau}\r|\r)^{\frac qp}\\
&&\sim\sum_{k=-\infty}^\infty2^{kq}|\Omega_k|^
{\frac qp}\sim\|M_N(f)\|_{L^{p,q}(\rn)}^q\sim\|f\|_
{H^{p,q}_A(\rn)}^q.\noz
\end{eqnarray}
Similar to \eqref{te65}, we conclude that
$\sup_{k\in\zz}(\sum_{i\in\mathbb{N}}
|\lambda_i^k|^p)^{1/p}\ls\|f\|_{H^{p,\fz}_A(\rn)}$.

\emph{Step 3.} By Step 2, we easily know that, to
prove \eqref{3.5x}, it suffices to show that
$f=\sum_{k\in\mathbb{Z}}\sum_{i\in\nn}h_i^k$ in $\cs'(\rn)$.

To show this, let $f_k:=\sum_{i\in\mathbb{N}}h_i^k$ for all $k\in\zz$.
Then
$f_k^{(m_\iota)}\rightarrow f_k$ in $\cs'(\rn)$ as
$\iota\rightarrow\fz$, where
$f^{(m)}_k:=g^{m,k+1}-g^{m,k}$
for all $m\in\nn$ and $k\in\mathbb{Z}$. Indeed, by the
finite intersection property of
$\{x_i^k+B_{\ell_i^k+4\tau}\}_{i\in\nn}$
for each $k\in\zz$ (see \eqref{te5}), and the support
conditions of $h_i^k$ and $h_i^{m,k}$, we know that,
for any $\phi\in\cs(\rn)$,
\begin{eqnarray*}
\lf\langle f_k^{(m_\iota)},\phi\r\rangle
&&=\lf\langle \sum_{i\in\mathbb{N}}h_i^
{m_\iota,k},\phi\r\rangle=\sum_{i\in\nn}
\lf\langle h_i^{m_\iota,k},\phi\r\rangle\\
&&\rightarrow\sum_{i\in\mathbb{N}}
\lf\langle h_i^k,\phi\r\rangle=\lf\langle \sum_
{i\in\nn}h_i^k,\phi\r\rangle=\langle f_k,
\phi\rangle,\ \ \ \ \iota\rightarrow\fz.
\end{eqnarray*}

We next show that, for all $m\in\mathbb{N}$,
\begin{eqnarray}\label{te60}
\sum_{|k|\geq K_1}f^{(m)}_k\rightarrow0\ \ \ \ \
{\rm in}\ \cs'(\rn),\ \ {\rm as}\ K_1\rightarrow\fz.
\end{eqnarray}
To show \eqref{te60}, we consider two cases.

For $p\in(0,1)$, it suffices to show that, for all
$m\in\mathbb{N}$
\begin{eqnarray}\label{te61}
\lim_{K_2\rightarrow-\fz}\lf\|\sum_{k\leq K_2}f^
{(m)}_k\r\|_{H^1_A(\rn)}=0\ \ \ {\rm and}\
\ \ \lim_{K_3\rightarrow\fz}\lf\|\sum_{k\geq K_3}
f^{(m)}_k\r\|_{H^{p_0}_A(\rn)}=0,
\end{eqnarray}
where  $p_0\in(0,p)$ satisfies that
$\lfloor(1/p_0-1)\ln b/\ln\lambda_-\rfloor\leq s$.
Indeed, by \eqref{te15}, \eqref{te16} and \eqref{te18},
we find that
$(2^k|B_{\ell_i^k+4\tau}|)^{-1}h_i^{m,k}$ is a
$(1,\fz,s)$-atom multiplied by a constant. Therefore,
by Lemma \ref{tl2}, \eqref{te5}, \eqref{te22} and \eqref{se7}, we have
\begin{eqnarray*}
\lf\|\sum_{k\leq K_2}f^{(m)}_k\r\|_{H^1_A(\rn)}
&&\leq\sum_{k\leq K_2}\sum_{i\in\mathbb{N}}
\lf\|h_i^{m,k}\r\|_{H^1_A(\rn)}\ls\sum_{k\leq K_2}
\sum_{i\in\mathbb{N}}2^k\lf|B_{\ell_i^k+4\tau}\r|\\
&&\ls\sum_{k\leq K_2}2^k|\Omega_k|\ls\sum_
{k\leq K_2}2^{k(1-p)}\|f\|_{H^{p,\fz}_A(\rn)}^p
\ls\sum_{k\leq K_2}2^{k(1-p)}\|f\|_{H^{p,q}_A(\rn)}^p,
\end{eqnarray*}
which converges to 0 as $K_2\rightarrow-\fz$. Similarly,
$(2^k|B_{\ell_i^k+4\tau}|^{1/p_0})^{-1}h_i^{m,k}$
is a $(p_0,\fz,s)$-atom multiplied by a constant. Since
$\lfloor(1/p_0-1)\ln b/\ln\lambda_-\rfloor\leq s$,
by Lemma\ \ref{tl2}, \eqref{te5}, \eqref{te22} and \eqref{se7} again, we find that
\begin{eqnarray*}
\lf\|\sum_{k\geq K_3}f^{(m)}_k\r\|_{H^{p_0}_A(\rn)}^{p_0}
&&\leq\sum_{k\geq K_3}\sum_{i\in\mathbb{N}}\lf\|h_i^
{m,k}\r\|_{H^{p_0}_A(\rn)}^{p_0}\ls\sum_{k\geq K_3}
\sum_{i\in\mathbb{N}}2^{kp_0}\lf|B_{\ell_i^k+4\tau}\r|
\ls\sum_{k\geq K_3}2^{kp_0}|\Omega_k|\\
&&\ls\sum_{k
\geq K_3}2^{k(p_0-p)}\|f\|_{H^{p,\fz}_A(\rn)}^p
\ls\sum_{k\geq K_3}2^{k(p_0-p)}\|f\|_{H^{p,q}_A(\rn)}^p,
\end{eqnarray*}
which converges to 0 as $K_3\rightarrow\fz$.
These prove \eqref{te61} and hence \eqref{te60}.

For $p=1$, we replace $H^1_A(\rn)$ by $L^2(\rn)$. Notice that
\begin{eqnarray*}
\lf\|\sum_{k\leq K_2}f^{(m)}_k\r\|_{L^2(\rn)}
&&\leq\sum_{k\leq K_2}\lf\|\sum_{i\in\mathbb{N}}h_i^
{m,k}\r\|_{L^2(\rn)}\ls\sum_{k\leq K_2}2^k\lf|\bigcup_
{i\in\nn}\lf(x_i^k+B_{\ell_i^k+4\tau}\r)\r|^{1/2}\\
&&\ls\sum_{k\leq K_2}2^k|\Omega_k|^{1/2}\ls\sum_
{k\leq K_2}2^{k/2}\|f\|_{H^{1,\fz}_A(\rn)}^{1/2}
\ls\sum_{k\leq K_2}2^{k/2}\|f\|_{H^{1,q}_A(\rn)}^{1/2},
\end{eqnarray*}
which converges to 0 as $K_3\rightarrow\fz$. This implies
that \eqref{te60} also holds true in this case.

An argument similar to that used in the proof of \eqref{te60} also shows that
$\sum_{|k|\geq K_1}f_k\rightarrow0$ in
$\cs'(\rn)$ as $K_1\rightarrow\fz$.
From this and \eqref{te60}, it follows that,
for any $\phi\in\cs(\rn)$ and $\varepsilon\in(0,\fz)$,
there exists some $\widetilde{K}_1\in\mathbb{N}$,
independent of $m_\iota$, such that
\begin{eqnarray}\label{te63}
\lf|\lf\langle\sum_{|k|\geq \widetilde{K}_1}f_k^
{(m_\iota)},\phi\r\rangle\r|<\varepsilon/3
\ \ {\rm and}\ \ \lf|\lf\langle\sum_{|k|\geq
\widetilde{K}_1}f_k,\phi\r\rangle\r|<\varepsilon/3.
\end{eqnarray}
Fixing this $\widetilde{K}_1$, by the fact that
$f^{(m_\iota)}_k\rightarrow f_k$ in $\cs'(\rn)$
as $\iota\rightarrow\fz$ for all $k\in\zz$,
we know that there exists an integer
$\widetilde{\iota}$ such that, if $\iota>\widetilde{\iota}$, then,
for all integers $k$ with $|k|\le\widetilde{K}_1$, we have
$$\lf|\lf\langle f_k^{(m_\iota)}-f_k,\phi\r\rangle\r|<
\frac\varepsilon{6\widetilde{K}_1+3},$$
which, combined with \eqref{te63}, implies that,  if $\iota>\widetilde{\iota}$,
\begin{eqnarray*}
&&\lf|\lf\langle\sum_{k\in\mathbb{Z}}f_k^{(m_\iota)},\phi\r
\rangle-\lf\langle\sum_{k\in\mathbb{Z}}f_k,\phi\r\rangle\r|\\
&&\hs\leq\lf|\lf\langle\sum_{|k|\geq \widetilde{K}_1}f_k^{(m_
\iota)},\phi\r\rangle\r|+\lf|\lf\langle\sum_{|k|\geq \widetilde
{K}_1}f_k,\phi\r\rangle\r|+\lf|\sum_{|k|\leq \widetilde{K}_1}
\lf\langle f_k^{(m_\iota)}-f_k,\phi\r\rangle\r|\noz\\
&&\hs<\frac\varepsilon3+\frac\varepsilon3+\frac\varepsilon3
=\varepsilon\noz.
\end{eqnarray*}
Thus,
$\lim_{\iota\rightarrow\fz}\langle\sum_{k\in\zz}f_k^
{(m_\iota)},\phi\rangle=\langle\sum_{k\in\zz}f_k,\phi\rangle$,
which, together with the fact that
$f^{(m)}\rightarrow f$ in $\cs'(\rn)$
as $m\rightarrow\fz$, further implies that
$$\langle f,\phi\rangle=\lim_{\iota\rightarrow\fz}\lf\langle f^
{(m_\iota)},\phi\r\rangle=\lim_{\iota\rightarrow\fz}\lf\langle
\sum_{k\in\mathbb{Z}}f_k^{(m_\iota)},\phi\r\rangle
=\lf\langle\sum_{k\in\mathbb{Z}}f_k,\phi\r\rangle.$$
This shows
$f=\sum_{k\in\mathbb{Z}}f_k=\sum_{k\in\zz}\sum_{i\in\nn}h_i^k$
in $\cs'(\rn)$, which completes the proof of \eqref{3.5x} and hence
the proof of $H^{p,q}_A(\rn)\subset H_A^{p,r,s,q}(\rn)$.

We now prove $H_A^{p,r,s,q}(\rn)\subset H^{p,q}_A(\rn)$. To this end,
for any $f\in H_A^{p,r,s,q}(\rn)$, by Definition \ref{d-aahls},
we know that there exists a sequence of $(p,r,s)$-atoms,
$\{a_i^k\}_{i\in\mathbb{N},\,k\in\mathbb{Z}}$,
supported on
$\{x_i^k+B_i^k\}_{i\in\mathbb{N},\,k\in\mathbb{Z}}\subset\mathfrak{B}$,
respectively, such that
$f=\sum_{k\in\mathbb{Z}}
\sum_{i\in\mathbb{N}}\lambda_i^ka_i^k$ in $\cs'(\rn)$,
where $\lambda_i^k\sim2^k|B_i^k|^{1/p}$ for all
$k\in\mathbb{Z}$ and
$i\in\mathbb{N}$, $\sum_{i\in\mathbb{N}}
\chi_{x_i^k+B_i^k}(x)\ls1$ for all $k\in\mathbb{Z}$ and $x\in\rn$, and
\begin{eqnarray}\label{te66}
\|f\|_{H_A^{p,r,s,q}(\rn)}
\sim\lf[\sum_{k\in\zz}\lf(\sum_{i\in\nn}
|\lambda_i^k|^p\r)^{\frac qp}\r]^{\frac 1q}.
\end{eqnarray}
Clearly, there exists a sequence,
$\{\ell_i^k\}_{i\in\mathbb{N},\,k\in\mathbb{Z}}$,
of integers
such that $x_i^k+B_{\ell_i^k}=x_i^k+B_i^k$ for
$i\in\mathbb{N}$ and $k\in\mathbb{Z}$. It suffices
to consider only
$N=N_{(p)}:=\lfloor(\frac1p-1)\frac
{\ln b}{\ln \lambda_-}\rfloor+2$.
Let
$\mu_k:=(\sum_{i\in\mathbb{N}}|B_{\ell_i^k}|)^{1/p}$
and
$\beta:=(\frac{\ln b}{\ln \lambda_-}+N-1)
\frac{\ln \lambda_-}{\ln b}>\frac 1p$.
Then, for $r\in(1,\fz]$, there exists $\frac1r<\delta<1$
such that $\frac 1\beta<\delta p<1$.
Notice that, for any fixed $k_0\in\zz$ and all $x\in\rn$,
$$M_N(f)(x)\leq M_N\lf(\sum_{k=-\infty}^{k_0-1}\sum_
{i\in\mathbb{N}}\lambda_i^ka_i^k\r)(x)+\sum_{k=
{k_0}}^\infty \sum_{i\in\mathbb{N}}|\lambda_i^k
|M_N(a_i^k)(x)=:\psi_{k_0}(x)+\eta_{k_0}(x).$$
To prove $H_A^{p,r,s,q}(\rn)\subset H^{p,q}_A(\rn)$,
we now consider two cases: $q/p\in[1,\fz]$ and $q/p\in(0,1)$.

\emph{Case 1:} $q/p\in[1,\fz]$. In this case, to show the desired conclusion,
we claim that
\begin{eqnarray}\label{te20}
2^{k_0p}\lf[d_{\psi_{k_0}}(2^{k_0})\r]^\delta\ls
\sum_{k=-\infty}^{k_0-1}\lf[2^k\mu_k^\delta\r]^p
\ \ \ {\rm and}\ \ \
2^{k_0\delta p}d_{\eta_{k_0}}(2^{k_0})\ls
\sum_{k=k_0}^\infty\lf[2^{k\delta}\mu_k\r]^p.
\end{eqnarray}
Assume that \eqref{te20} holds true for the moment.
Notice that $\delta\in(0,q/p)$. Then, by Lemma \ref{tl3},
the fact that $|B_{\ell_i^k}|\sim\frac{|\lambda_i^k|^p}{2^{kp}}$
and \eqref{te66}, we have
\begin{eqnarray*}
\|f\|_{H^{p,q}_A(\rn)}&&=\|M_N(f)\|_{L^{p,q}(\rn)}\\
&&\ls\lf\|\lf\{2^k\mu_k\r\}_{k\in\mathbb{Z}}\r\|_{\ell^q}
\ls\lf[\sum_{k\in\zz}\lf(\sum_{i\in\nn}|\lambda_i^k|^p\r)^{\frac qp}\r]
^{\frac 1q}\sim\|f\|_{H_A^{p,r,s,q}(\rn)}
\end{eqnarray*}
with the usual interpretation for $q=\fz$, which implies that
$\|f\|_{H^{p,q}_A(\rn)}\ls\|f\|_{H_A^{p,r,s,q}(\rn)}$
and hence $H_A^{p,r,s,q}(\rn)\subset H^{p,q}_A(\rn)$.

Now, let us give the proof of the claim \eqref{te20}. To this end,
we first estimate $\psi_{k_0}$.
Notice that $a_i^k$ is a $(p,r,s)$-atom,
$\supp a_i^k\subset x_i^k+B_{\ell_i^k}$,
$\sum_{i\in\mathbb{N}}\chi_{x_i^k+B_{\ell_i^k}}\ls1$
and $\lambda_i^k\sim2^k|B_{\ell_i^k}|^{1/p}$.
For $r\in(1,\fz)$, by the H\"{o}lder inequality, we find that, for
$\sigma:=1-\frac p{r\delta}>0$ and all $x\in\rn$,
\begin{eqnarray*}
\psi_{k_0}(x)
&&\leq\sum_{k=-\infty}^{k_0-1}M_N\lf(\sum_{i\in\nn}
\lambda_i^ka_i^k\r)(x)\\
&&\leq\lf(\sum_{k=-\infty}^{k_0-1}2^{k\sigma r'}\r)^
{1/r'}\lf\{\sum_{k=-\infty}^{k_0-1}2^{-k\sigma r}
\lf[M_N\lf(\sum_{i\in\mathbb{N}}\lambda_i^ka_i^k\r)(x)
\r]^r\r\}^{1/r}\\
&&=\widetilde{C}2^{k_0\sigma}\lf\{\sum_{k=-\infty}^
{k_0-1}2^{-k\sigma r}\lf[M_N\lf(\sum_{i\in\mathbb{N}}
\lambda_i^ka_i^k\r)(x)\r]^r\r\}^{1/r},
\end{eqnarray*}
where $\widetilde{C}=:(\frac1{2^{\sigma r'}-1})^{1/r'}$,
which, combined with Proposition \ref{tl4}, Remark \ref{tr1}
and the finite intersection property of
$\{x_i^k+B_{\ell_i^k}\}_{i\in\nn}$ for each $k\in\zz$,
further implies that
\begin{eqnarray}\label{te21}
&&2^{k_0p}\lf[d_{\psi_{k_0}}(2^{k_0})\r]^\delta\\
&&\quad=2^{k_0p}\lf|\lf\{x\in\rn:\ \psi_{k_0}(x)
>2^{k_0}\r\}\r|^\delta\noz\\
&&\quad\leq 2^{k_0p}\lf|\lf\{x\in\rn:\ \widetilde{C}
^r\sum_{k=-\infty}^{k_0-1}2^{-k\sigma r}\lf[M_N
\lf(\sum_{i\in\mathbb{N}}\lambda_i^ka_i^k\r)(x)\r]
^r>2^{k_0r(1-\sigma)}\r\}\r|^\delta\noz\\
&&\quad=2^{k_0p}\lf\{\int_{\{x\in\rn:\ \widetilde{C}
^r\sum_{k=-\infty}^{k_0-1}2^{-k\sigma r}\lf[M_N
\lf(\sum_{i\in\mathbb{N}}\lambda_i^ka_i^k\r)(x)\r]^
r>2^{k_0r(1-\sigma)}\}}\,dx\r\}^\delta\noz\\
&&\quad\leq \widetilde{C}^{r\delta}2^{k_0p}2^
{-k_0r\delta(1-\sigma)}\lf\{\int_{\rn}\sum_{k=-\infty}
^{k_0-1}2^{-k\sigma r}\lf[M_N\lf(\sum_{i\in\nn}
\lambda_i^ka_i^k\r)(x)\r]^r\,dx\r\}^\delta\noz\\
&&\quad\ls\lf(\sum_{k=-\infty}^{k_0-1}2^{-k\sigma r}
\int_{\rn}\lf|\sum_{i\in\nn}\lambda_i^ka_i^k(x)\r|^
r\,dx\r)^\delta\noz\\
&&\quad\ls\lf(\sum_{k=-\infty}^{k_0-1}2^{-k\sigma r}
\sum_{i\in\mathbb{N}}|\lambda_i^k|^r\int_{x_i^k+B_
{\ell_i^k}}|a_i^k(x)|^r\,dx\r)^\delta\noz\\
&&\quad\ls\lf[\sum_{k=-\infty}^{k_0-1}2^{-k\sigma r}
\sum_{i\in\mathbb{N}}2^{kr}\lf|B_{\ell_i^k}\r|^{\frac rp}
\lf|B_{\ell_i^k}\r|^{(\frac 1r-\frac 1p)r}\r]^\delta\noz\\
&&\quad\ls\sum_{k=-\infty}^{k_0-1}2^{kp}\lf(\sum_
{i\in\mathbb{N}}\lf|B_{\ell_i^k}\r|\r)^\delta\sim\sum_
{k=-\infty}^{k_0-1}\lf[2^k\mu_k^\delta\r]^p,\noz
\end{eqnarray}
which is the desired estimate of $\psi_{k_0}$ for $r\in(1,\fz)$ in \eqref{te20}.

For $r=\fz$, by Proposition \ref{tl4}, Remark \ref{tr1}
and the finite intersection property of
$\{x_i^k+B_{\ell_i^k}\}_{i\in\nn}$ for each $k\in\zz$ again, we have
\begin{eqnarray}\label{te8}
2^{k_0p}\lf[d_{\psi_{k_0}}(2^{k_0})\r]^\delta
&&=2^{k_0p}\lf|\lf\{x\in\rn:\ \psi_{k_0}(x)
>2^{k_0}\r\}\r|^\delta\\
&&\le2^{k_0(p-\delta\widetilde{r})}\lf\{\sum_{k=-\fz}^{k_0-1}\int_{\rn}
\lf[M_N\lf(\sum_{i\in\nn}\lambda_i^ka_i^k\r)\r]^{\widetilde{r}}(x)\,dx\r\}^\delta\noz\\
&&\ls2^{k_0(p-\delta\widetilde{r})}
\lf\{\sum_{k=-\fz}^{k_0-1}\sum_{i\in\nn}\int_{x_i^k+B_{\ell_i^k}}
\lf|\lambda_i^ka_i^k\r|^{\widetilde{r}}(x)\,dx\r\}^\delta\noz\\
&&\ls\sum_{k=-\infty}^{k_0-1}2^{kp}\lf(\sum_
{i\in\mathbb{N}}\lf|B_{\ell_i^k}\r|\r)^\delta\sim\sum_
{k=-\infty}^{k_0-1}\lf[2^k\mu_k^\delta\r]^p,\noz
\end{eqnarray}
where $\widetilde{r}\in(1,\fz)$ such that $\delta\widetilde{r}>p$, which,
together with \eqref{te21}, implies the desired estimate of
$\psi_{k_0}$ in \eqref{te20}.

In order to estimate $\eta_{k_0}$, it suffices to prove that,
for all $i\in\mathbb{N}$ and $k\in\mathbb{Z}$,
\begin{eqnarray}\label{te23}
\int_\rn \lf[M_N(a_i^k)(x)\r]^{\delta p}dx
\ls\lf|B_{\ell_i^k}\r|^{1-\delta}.
\end{eqnarray}
Indeed, by \eqref{te23}, we have
\begin{eqnarray}\label{te46}
2^{k_0\delta p}d_{\eta_{k_0}}(2^{k_0})
&&=2^{k_0\delta p}\lf|\lf\{x\in\rn:\ \lf[\sum_
{k={k_0}}^\infty\sum_{i\in\mathbb{N}}|
\lambda_i^k|M_N(a_i^k)(x)\r]^{\delta p}>2^
{k_0\delta p}\r\}\r|\\
&&\leq\int_\rn\lf[\sum_{k={k_0}}^\infty\sum_
{i\in\mathbb{N}}|\lambda_i^k|M_N(a_i^k)(x)\r]^
{\delta p}\,dx\noz\\
&&\leq\sum_{k={k_0}}^\infty\sum_{i\in\nn}
|\lambda_i^k|^{\delta p}\int_\rn\lf[M_N(a_i^k)
(x)\r]^{\delta p}\,dx\noz\\
&&\ls\sum_{k={k_0}}^\infty\sum_{i\in\nn}
|\lambda_i^k|^{\delta p}\lf|B_{\ell_i^k}\r|^{1-\delta}
\ls\sum_{k={k_0}}^\infty2^{k\delta p}\sum_
{i\in\mathbb{N}}\lf|B_{\ell_i^k}\r|\sim\sum_{k=k_0}
^\infty\lf[2^{k\delta}\mu_k\r]^p\noz,
\end{eqnarray}
which is the desired estimate of $\eta_{k_0}$ in
\eqref{te20}.

To show \eqref{te23}, we write
\begin{eqnarray*}
\int_\rn \lf[M_N(a_i^k)(x)\r]^{\delta p}\,dx
&&=\int_{x_i^k+B_{\ell_i^k+\tau}} \lf[M_N(a_i^k)(x)\r]^
{\delta p}\,dx+\int_{(x_i^k+B_{\ell_i^k+\tau})^
\complement}\cdots\\
&&=:{\rm I_1}+{\rm I_2}.
\end{eqnarray*}
For $r\in(1,\fz]$, by the H\"{o}lder inequality, Proposition \ref{tl4}
and Remark \ref{tr1}, we find that
\begin{eqnarray*}
{\rm I_1}
&&=\int_{x_i^k+B_{\ell_i^k+\tau}}\lf[M_N(a_i^k)(x)\r]
^{\delta p}\,dx\\
&&\leq\lf\{\int_{x_i^k+B_{\ell_i^k+\tau}}\lf[M_N(a_i^
k)(x)\r]^r\,dx\r\}^{\frac {\delta p} r}\lf|B_{\ell_i^k+
\tau}\r|^{1-\frac {\delta p} r}\noz\\
&&\ls\lf\|a_i^k\r\|_{L^r(\rn)}^{\delta p}\lf|B_{\ell_i^
k+\tau}\r|^{1-\frac {\delta p} r}\noz\\
&&\ls\lf|B_{\ell_i^k+\tau}\r|^{\delta p(\frac 1r-\frac
1p)}\lf|B_{\ell_i^k+\tau}\r|^{1-\frac {\delta p} r}
\sim\lf|B_{\ell_i^k+\tau}\r|^{1-\delta}
\sim\lf|B_{\ell_i^k}\r|^{1-\delta}.\noz
\end{eqnarray*}

To estimate ${\rm I_2}$, it suffices to show that, for all
$i\in\mathbb{N}$, $k\in\mathbb{Z}$
and
$x\in (x_i^k+B_{\ell_i^k+\tau})^\complement$,
\begin{eqnarray}\label{te26}
M_N^0(a_i^k)(x)\ls\lf|B_{\ell_i^k}\r|^{-\frac 1p}\frac
{\lf|B_{\ell_i^k}\r|^\beta}{\lf[\rho(x-x_i^k)\r]^\beta},
\end{eqnarray}
where $M_N^0(f)$ denotes the radial grand maximal function of
$f$ as in Definition \ref{d-mf},
$\beta:=(\frac{\ln b}{\ln \lambda_-}
+N-1)\frac{\ln \lambda_-}{\ln b}$
and $\rho$ denotes the homogeneous quasi-norm associated
with the dilation $A$ in Definition \ref{sd2}. Indeed,
assuming that \eqref{te26} holds true for the moment,
noticing that $\beta\delta p>1$, then,
by Proposition \ref{sp1} and \eqref{te26}, we have
\begin{eqnarray}\label{te27}
{\rm I_2}&&\ls\int_{(x_i^k+B_{\ell_i^k+\tau})^
\complement} \lf[M_N^0(a_i^k)(x)\r]^{\delta p}\,dx\\
&&\ls\int_{\rho(x-x_i^k)\geq|B_{\ell_i^k+\tau}|}
\lf|B_{\ell_i^k}\r|^{-\delta}\frac{\lf|B_{\ell_i^k}\r|
^{\beta\delta p}}{\lf[\rho(x-x_i^k)\r]^{\beta\delta p
}}\,dx\noz\\
&&\sim\sum_{j=0}^\infty\int_{2^j|B_{\ell_i^k+\tau}
|\leq\rho(x-x_i^k)<2^{j+1}|B_{\ell_i^k+\tau}|}\lf|B_
{\ell_i^k}\r|^{-\delta}\frac{\lf|B_{\ell_i^k}\r|^{\beta
\delta p}}{\lf[\rho(x-x_i^k)\r]^{\beta\delta p}}\,dx\noz\\
&&\sim\sum_{j=0}^\infty\int_{\rho(x-x_i^k)\sim2^j|
B_{\ell_i^k+\tau}|}\lf|B_{\ell_i^k}\r|^{-\delta}\frac
{\lf|B_{\ell_i^k}\r|^{\beta\delta p}}{\lf(2^j|B_
{\ell_i^k+\tau}|\r)^{\beta\delta p}}\,dx\noz\\
&&\sim\lf|B_{\ell_i^k}\r|^{-\delta}\lf|B_{\ell_i^k+
\tau}\r|\sum_{j=0}^\infty2^j2^{-j\beta\delta p}
\sim\lf|B_{\ell_i^k}\r|^{1-\delta}\noz,
\end{eqnarray}
which completes the proof of \eqref{te23}.

Thus, to obtain the desired conclusion of Case 1,
we still need to prove \eqref{te26}. To this end,
take $x\in(x_i^k+B_{\ell_i^k+\tau})^\complement$,
$\varphi\in\cs_N(\rn)$ and $t\in\mathbb{Z}$.
Suppose that $P$ is a polynomial with degree not more than
$s$, which will be determined later. Then, for all
$i\in{\mathbb{N}}$ and $k\in\mathbb{Z}$, by the
H\"{o}lder inequality, we have
\begin{eqnarray}\label{te28}
&&\lf|(a_i^k\ast\varphi_t)(x)\r|\\
&&\hs=b^{-t}\lf|\int_\rn a_i^k(y)\varphi\lf
(A^{-t}(x-y)\r)\,dy\r|\noz\\
&&\hs=b^{-t}\lf|\int_{x_i^k+B_{\ell_i^k}}
a_i^k(y)\lf[\varphi\lf(A^{-t}(x-y)\r)-P\lf
(A^{-t}(x-y)\r)\r]\,dy\r|\noz\\
&&\hs\leq b^{-t}\lf\|a_i^k\r\|_{L^r(\rn)}\lf[\int_{x_i^k
+B_{\ell_i^k}}\lf|\varphi\lf(A^{-t}(x-y)\r)-
P\lf(A^{-t}(x-y)\r)\r|^{r'}\,dy\r]^{1/r'}\noz\\
&&\hs\leq b^{-t}\lf|B_{\ell_i^k}\r|^{1/r-1/p}b
^{t/r'}\lf\{\int_{A^{-t}(x-x_i^k)+B_{\ell_i^k
-t}}|\varphi(y)-P(y)|^{r'}\,dy\r\}^{1/r'}\noz\\
&&\hs\leq b^{-t}\lf|B_{\ell_i^k}\r|^{1/r-1/p}b^
{t/r'}b^{-t/r'}\lf|B_{\ell_i^k}\r|^{1/r'}\sup_
{y\in A^{-t}(x-x_i^k)+B_{\ell_i^k-t}}
|\varphi(y)-P(y)|\noz\\
&&\hs=\lf|B_{\ell_i^k}\r|^{-1/p}b^{\ell_i^k-t}
\sup_{y\in A^{-t}(x-x_i^k)+B_{\ell_i^k-t}}|
\varphi(y)-P(y)|\noz.
\end{eqnarray}
Suppose that
$x\in x_i^k+B_{\ell_i^k+\tau+m+1}
\setminus B_{\ell_i^k+\tau+m}$
for some integer $m\in\mathbb{Z}_+$. Then, by \eqref{se2}, we obtain
\begin{eqnarray}\label{3.33x}
A^{-t}(x-x_i^k)+B_{\ell_i^k-t}
&&\subset A^{-t}(B_{\ell_i^k+\tau+m+1}
\setminus B_{\ell_i^k+\tau+m})+B_{\ell_i^k-t}\\
&&=A^{\ell_i^k-t}((B_{\tau+m+1}
\setminus B_{\tau+m})+B_0)\subset A^
{\ell_i^k-t}(B_m)^\complement.\noz
\end{eqnarray}

If $\ell_i^k\geq t$,
we choose $P\equiv0$. Then, by \eqref{3.33x}, we know that
\begin{eqnarray}\label{te29}
\ \ \sup_{y\in A^{-t}(x-x_i^k)+B_{\ell_i^k-t}}|\varphi(y)|
\leq\sup_{y\in A^{-t}(x-x_i^k)+B_{\ell_i^k-t}}\
\min\lf\{1,\ \rho(y)^{-N}\r\}\leq b^{-N(\ell_i^k-t+m)}.
\end{eqnarray}
If $\ell_i^k<t$, then we let $P$ be the Taylor expansion
of $\varphi$ at the point $A^{-t}(x-x_i^k)$ of order $s$.
By the Taylor remainder theorem, \eqref{se19} and
\eqref{3.33x}, we have
\begin{eqnarray}\label{te30}
\ \ \ \ \sup_{y\in A^{-t}(x-x_i^k)+
B_{\ell_i^k-t}}|\varphi(y)-P(y)|
&&\ls\sup_{z\in B_{\ell_i^k-t}}\
\sup_{|\alpha|=s+1}\lf|\partial^
\alpha\varphi(A^{-t}(x-x_i^k)+z)
\r||z|^{s+1}\\
&&\ls\lambda_-^{(s+1)(\ell_i^k-t)}
\sup_{y\in A^{-t}(x-x_i^k)+B_{
\ell_i^k-t}}\min\lf\{1,\rho(y)^{-N}
\r\}\noz\\
&&\ls\lambda_-^{(s+1)(\ell_i^k-t)}
\min\lf\{1,b^{-N(\ell_i^k-t+m)}\r\}.\noz
\end{eqnarray}
Combining \eqref{te28}, \eqref{te29} and
\eqref{te30}, for all
$x\in x_i^k+B_{\ell_i^k+\tau+m+1}
\setminus B_{\ell_i^k+\tau+m}$ with $m\in\zz_+$,
we further conclude that
\begin{eqnarray*}
M_N^0(a_i^k)(x)
&&=\sup_{\varphi\in\cs_N(\rn)}\sup_{t\in\zz}
\lf|(a_i^k\ast\varphi_t)(x)\r|\\
&&\ls\lf|B_{\ell_i^k}\r|^{-1/p}\max\lf\{
\sup_{t\in\mathbb{Z},\,t\leq\ell_i^k}b^{
\ell_i^k-t}b^{-N(\ell_i^k-t+m)},\r.\\
&&\quad\lf.\sup_{t\in\mathbb{Z},\,t>\ell_
i^k}b^{\ell_i^k-t}\lambda_-^{(s+1)(
\ell_i^k-t)}\min\lf\{1,b^{-N(\ell_i^k-t+m)}
\r\}\r\}.\noz
\end{eqnarray*}
Notice that the supremum over $t\leq\ell_i^k$
has the largest value when $t=\ell_i^k$. Without
loss of generally, we may assume that
$s:=\lfloor(\frac1p-1)\frac{\ln b}{\ln \lambda_-}\rfloor$.
Since $N=s+2$ implies $b\lambda_-^{s+1}\leq b^N$
and the above supremum over $t>\ell_i^k$ is attained when
$\ell_i^k-t+m=0$, it follows that
\begin{eqnarray}\label{te7}
M_N^0(a_i^k)(x)
&&\ls\lf|B_{\ell_i^k}\r|^{-1/p}\max\lf\{b^
{-Nm},(b\lambda_-^{s+1})^{-m}\r\}\ls
\lf|B_{\ell_i^k}\r|^{-1/p}(b\lambda_-^{s+1})
^{-m}\\
&&\sim\lf|B_{\ell_i^k}\r|^{-1/p}b^{-m}
b^{-(s+1)m\frac{\ln\lambda
_-}{\ln b}}\noz\\
&&\ls\lf|B_{\ell_i^k}\r|^{-1/p}b^{\ell_i^k
[(s+1)\frac{\ln\lambda_-}{\ln b}+1]}b^{-(\ell
_i^k+\tau+m)[(s+1)\frac{\ln\lambda_-}{\ln b}
+1]}\noz\\
&&\ls\lf|B_{\ell_i^k}\r|^{-1/p}\lf|B_{\ell_i^k}
\r|^{(s+1)\frac{\ln\lambda_-}{\ln b}+1}\lf[\rho
(x-x_i^k)\r]^{-[(s+1)\frac{\ln\lambda_-}{\ln b}
+1]}\noz\\
&&\sim\lf|B_{\ell_i^k}\r|^{-1/p}\frac{\lf|B_{\ell
_i^k}\r|^\beta}{\lf[\rho(x-x_i^k)\r]^\beta},\noz
\end{eqnarray}
which is \eqref{te26}. This finishes the proof of Case 1.

\emph{Case 2:} $q/p\in(0,1)$. In this case, when $r\in(1,\fz)$,
similar to \eqref{te21}, we have
\begin{eqnarray}\label{te9}
\ \ \ \ \ \ \ \ 2^{k_0p}\lf[d_{\psi_{k_0}}(2^{k_0})\r]^\delta
\ls\lf[\sum_{k=-\infty}^{k_0-1}2^{-k\sigma r}
\sum_{i\in\mathbb{N}}2^{kr}\lf|B_{\ell_i^k}\r|^
{\frac rp}\lf|B_{\ell_i^k}\r|^{(\frac 1r-\frac 1p)r}\r]
^\delta\sim\lf(\sum_{k=-\infty}^{k_0-1}2^
{\frac{kp}\delta}\mu_k^p\r)^\delta.
\end{eqnarray}
By a similar calculation to \eqref{te8}, we easily know that
\eqref{te9} also holds true for $r=\fz$. This further implies that
\begin{eqnarray}\label{te68}
\ \ \ \sum_{k_0\in\zz}2^{k_0q}
\lf|\lf\{x\in\rn:\ \psi_{k_0}(x)>2^{k_0}\r\}\r|^
{\frac qp}
&&\ls\sum_{k_0\in\zz}2^{k_0(q-\frac q\delta)}
\sum_{k=-\infty}^{k_0-1}2^{\frac{kq}\delta}
\mu_k^q\\
&&\sim\sum_{k\in\zz}\sum_{k_0=k+1}^{\fz}
2^{k_0(q-\frac q\delta)}2^{\frac{kq}\delta}
\mu_k^q\ls\sum_{k\in\zz}2^{kq}\mu_k^q.\noz
\end{eqnarray}
On the other hand, similar to \eqref{te46}, we obtain
\begin{eqnarray*}
2^{k_0\delta p}d_{\eta_{k_0}}(2^{k_0})
\ls\sum_{k=k_0}^\infty\lf[2^{k\delta}\mu_k\r]^p,
\end{eqnarray*}
which, together with $q<p$, implies that
\begin{eqnarray*}
&&2^{k_0\delta p}\lf|\lf\{x\in\rn:\ \eta_{k_0}(x)
>2^{k_0}\r\}\r|\\
&&\hs\ls\sum_{k=k_0}^\fz2^{-k
\widetilde{\delta}p}\lf[2^{k(1-\widetilde{\delta}
)}\mu_k\r]^p\ls2^{-k_0\widetilde{\delta}p}
\lf\{\sum_{k=k_0}^\fz\lf[2^{k(1-\widetilde
{\delta})}\mu_k\r]^q\r\}^{\frac pq},
\end{eqnarray*}
where $\widetilde{\delta}:=\frac{1-\delta}2$.
Therefore,
\begin{eqnarray}\label{te69}
&&\sum_{k_0\in\zz}2^{k_0q}\lf|\lf\{x\in\rn:\
\eta_{k_0}(x)>2^{k_0}\r\}\r|^{\frac qp}\\
&&\hs\ls\sum_{k_0\in\zz}
2^{k_0\widetilde{\delta}q}\sum_{k=k_0}^\fz
\lf[2^{k(1-\widetilde{\delta})}\mu_k\r]^q
\sim\sum_{k\in\zz}
\lf[2^{k(1-\widetilde{\delta})}\mu_k\r]^q
\sum_{k_0=-\fz}^k2^{k_0\widetilde{\delta}q}
\ls\sum_{k\in\zz}2^{kq}\mu_k^q.\noz
\end{eqnarray}
Notice that
$\mu_k:=(\sum_{i\in\mathbb{N}}|B_{\ell_i^k}|)^{1/p}$
and
$\lambda_i^k\sim2^k|B_{\ell_i^k}|^{1/p}$.
Combining \eqref{se6}, \eqref{te68}, \eqref{te69} and
\eqref{te66}, we further conclude that
\begin{eqnarray*}
\|M_N(f)\|_{L^{p,q}(\rn)}^q\sim
&&\sum_{k_0\in\zz}2^{k_0q}\lf|\lf\{x\in
\rn:\ M_N(f)(x)>2^{k_0}\r\}\r|^{\frac qp}\\
\ls&&\sum_{k_0\in\zz}2^{k_0q}\lf|\lf\{x\in
\rn:\ \psi_{k_0}(x)>2^{k_0}\r\}\r|^{\frac qp}\\
&&+\sum_{k_0\in\zz}2^{k_0q}\lf|\lf\{x\in\rn:\
\eta_{k_0}(x)>2^{k_0}\r\}\r|^{\frac qp}\noz\\
\ls&&\sum_{k\in\zz}2^{kq}\mu_k^q\sim\sum
_{k\in\zz}\lf[\sum_{i\in\nn}|\lambda_i^k|^p\r]
^{\frac qp}\sim\|f\|_{H_A^{p,r,s,q}(\rn)}^q,
\end{eqnarray*}
which implies that
$\|f\|_{H^{p,q}_A(\rn)}\ls\|f\|_{H_A^{p,r,s,q}(\rn)}$
and
$H_A^{p,r,s,q}(\rn)\subset H^{p,q}_A(\rn)$.
This finishes the proof of Case 2 and hence Theorem \ref{tt1}.
\end{proof}

\subsection{Molecular characterizations of $H^{p,q}_A(\rn)$}\label{s3.2}

\hskip\parindent
In this subsection, we establish the molecular characterizations
of $H^{p,q}_A(\rn)$. We begin with the following notion of anisotropic
$(p,r,s,\varepsilon)$-molecules associated with dilated balls.

\begin{definition}\label{d-mol}
An anisotropic quadruple $(p,r,s,\varepsilon)$ is said to be
\emph{admissible} if
$p\in(0,1]$, $r\in(1,\fz]$, $s\in\mathbb{N}$
with
$s\geq\lfloor(1/p-1)\ln b/\ln\lambda_-\rfloor$
and
$\varepsilon\in(0,\infty)$.
For an admissible anisotropic quadruple $(p,r,s,\varepsilon)$,
a measurable function $m$
is called an \emph{anisotropic $(p,r,s,\varepsilon)$-molecule}
associated with a dilated ball $B\in\mathfrak{B}$ if

(i) for each $j\in\zz_+$,
$\|m\|_{L^r(U_j(B))}\leq b^{-j\varepsilon}|B|^{1/r-1/p}$,
where $U_0(B):=B$ and, for
$j\in\mathbb{N}$, $U_j(B):=(A^jB)\setminus(A^{j-1}B)$;

(ii) $\int_\rn m(x)x^\alpha\,dx =0$ for any $\alpha\in\zz_+^n$
with $|\alpha|\leq s$.
\end{definition}

Throughout this article, we call anisotropic
$(p,r,s,\varepsilon)$-molecule by
$(p,r,s,\varepsilon)$-molecule for convenience. Via
$(p,r,s,\varepsilon)$-molecules, we introduce
the following anisotropic molecular Hardy-Lorentz
space $H_A^{p,r,s,\varepsilon,q}(\rn)$.

\begin{definition}\label{d-amhls}
For an admissible anisotropic quadruple
$(p,r,s,\varepsilon)$, $q\in(0,\fz]$ and a dilation $A$,
the \emph{anisotropic molecular Hardy-Lorentz space}
$H_A^{p,r,s,\varepsilon,q}(\rn)$ is defined to be the
set of all distributions $f\in\cs'(\rn)$ satisfying that there exist
a sequence of $(p,r,s,\varepsilon)$-molecules,
$\{m_i^k\}_{i\in\mathbb{N},\,k\in\mathbb{Z}}$,
associated with dilated balls
$\{x_i^k+B_i^k\}_{i\in\mathbb{N},\,k\in\mathbb{Z}}\subset\mathfrak{B}$,
respectively, and a positive constant $\widetilde{C}$
such that
$\sum_{i\in\mathbb{N}}\chi_{x_i^k+B_i^k}(x)
\leq \widetilde{C}$
for all $k\in\mathbb{Z}$ and $x\in\rn$,
and
$f=\sum_{k\in\mathbb{Z}}
\sum_{i\in\mathbb{N}}\lambda_i^km_i^k$ in $\cs'(\rn)$,
where $\lambda_i^k\sim2^k|B_i^k|^{1/p}$ for all
$k\in\mathbb{Z}$ and $i\in\mathbb{N}$.

Moreover, define
\begin{eqnarray*}
\|f\|_{H_A^{p,r,s,\varepsilon,q}(\rn)}
:={\rm inf}\lf\{\lf[\sum_{k\in\zz}\lf(
\sum_{i\in\nn}|\lambda_i^k|^p\r)^
{\frac qp}\r]^{\frac 1q}:\ f=\sum_{
k\in\mathbb{Z}}\sum_{i\in\nn}
\lambda_i^km_i^k\r\}
\end{eqnarray*}
with the usual modification made when $q=\fz$, where the
infimum is taken over all decompositions of $f$ as above.
\end{definition}

Now we state the main theorem of this subsection as follows.
\begin{theorem}\label{tt2}
Let $(p,r,s,\varepsilon)$ be an admissible anisotropic
quadruple defined as Definition \ref{d-mol} with
$\varepsilon\in(\max\{1,(s+1)\log_b(\lz_+)\},\fz)$, $q\in(0,\fz]$ and $N\in\nn\cap(N_{(p)},\fz)$.
Then $H^{p,q}_A(\rn)=H_A^{p,r,s,\varepsilon,q}(\rn)$
with equivalent quasi-norms.
\end{theorem}
\begin{proof}
By the definitions of $(p,\infty,s)$-atoms and
$(p,r,s,\varepsilon)$-molecules, we know that each
$(p,\infty,s)$-atom is also a $(p,r,s,\varepsilon)$-molecule,
which implies that
$$H_A^{p,\infty,s,q}(\rn)
\subset H_A^{p,r,s,\varepsilon,q}(\rn).$$
This, combined with Theorem \ref{tt1}, further implies that,
to prove Theorem \ref{tt2}, it suffices to show
$H_A^{p,r,s,\varepsilon,q}(\rn)\subset H^{p,q}_A(\rn)$.

To show this, for any $f\in H_A^{p,r,s,\varepsilon,q}(\rn)$, by Definition \ref{d-amhls},
we know that there exists
a sequence of $(p,r,s,\varepsilon)$-molecules,
$\{m_i^k\}_{i\in\mathbb{N},\,k\in\mathbb{Z}}$,
associated with
$\{x_i^k+B_i^k\}_{i\in\mathbb{N},\,k\in\mathbb{Z}}\subset\mathfrak{B}$,
respectively, such that
$f=\sum_{k\in\mathbb{Z}}
\sum_{i\in\mathbb{N}}\lambda_i^km_i^k$ in $\cs'(\rn)$,
$\lambda_i^k\sim2^k|B_i^k|^{1/p}$ for all $k\in\mathbb{Z}$
and
$i\in\mathbb{N}$,
$\sum_{i\in\mathbb{N}}\chi_{x_i^k+B_i^k}(x)\ls1$ for all
$k\in\mathbb{Z}$ and $x\in\rn$,
and
\begin{eqnarray}\label{te67}
\|f\|_{H_A^{p,r,s,\varepsilon,q}(\rn)}\sim
\lf[\sum_{k\in\zz}\lf(\sum_{i\in\nn}|\lambda
_i^k|^p\r)^{\frac qp}\r]^{\frac 1q}.
\end{eqnarray}

Take a sequence
$\{\ell_i^k\}_{i\in\mathbb{N},\,k\in\mathbb{Z}}$
of integers satisfying that, for all
$i\in\mathbb{N}$ and $k\in\mathbb{Z}$,
$x_i^k+B_{\ell_i^k}:=x_i^k+B_i^k$.
It suffices to consider only
$N=N_{(p)}:=\lfloor(\frac1p-1)\frac{\ln b}{\ln \lambda_-}\rfloor+2$.
Let
$$\beta:=\lf(\frac{\ln b}{\ln \lambda_-}+N-1\r)
\frac{\ln \lambda_-}{\ln b}>\frac 1p$$
and
$\mu_k:=(\sum_{i\in\mathbb{N}}
|B_{\ell_i^k}|)^{1/p}$ for all $k\in\mathbb{Z}$.
Then, for $r\in(1,\fz]$, there exist $\widetilde{r}\in(1,r)$
and $\delta\in(0,1)$ such that
$\frac1{\widetilde{r}}<\delta<1$
and $\frac 1\beta<\delta p<1$. Notice that, for any fixed
$k_0\in\zz$ and all $x\in\rn$,
\begin{eqnarray}\label{3.41x}
M_N(f)(x)&&\leq M_N\lf(\sum_{k=-\infty}^{k_0-1}
\sum_{i\in\mathbb{N}}\lambda_i^km_i^k\r)(x)
+\sum_{k={k_0}}^\infty \sum_{i\in\nn}
|\lambda_i^k|M_N(m_i^k)(x)\\
&&=:\widetilde{\psi}
_{k_0}(x)+\widetilde{\eta}_{k_0}(x).\noz
\end{eqnarray}
To finish the proof that
$H_A^{p,r,s,\varepsilon,q}(\rn)\subset H^{p,q}_A(\rn)$,
it suffices to show
$\|M_N(f)\|_{L^{p,q}(\rn)}\ls
\|f\|_{H_A^{p,r,s,\varepsilon,q}(\rn)}$.
To this end,
we now consider two cases: $q/p\in[1,\fz]$ and $q/p\in(0,1)$.

\emph{Case 1:} $q/p\in[1,\fz]$. In this case,
if we can prove that
\begin{eqnarray}\label{te40}
2^{k_0p}\lf[d_{\widetilde{\psi}_{k_0}}
(2^{k_0})\r]^\delta
\ls\sum_{k=-\infty}^{k_0-1}\lf[
2^k\mu_k^\delta\r]^p\ \ \ {\rm and}\ \ \
2^{k_0\delta p}d_{\widetilde{\eta}_{k_0}}
(2^{k_0})\ls\sum_{k=k_0}^\infty\lf[2^
{k\delta}\mu_k\r]^p,
\end{eqnarray}
then, noticing that $\delta\in(0,q/p)$, by Lemma \ref{tl3},
$|B_{\ell_i^k}|\sim\frac{|\lambda_i^k|^p}{2^{kp}}$
and \eqref{te67},
we have
\begin{eqnarray*}
\|f\|_{H^{p,q}_A(\rn)}
&&=\|M_N(f)\|_{L^{p,q}(\rn)}\\
&&\ls\lf\|\lf\{2^k\mu_k\r\}_{k\in\zz}
\r\|_{\ell^q}\ls\lf[\sum_{k\in\zz}
\lf(\sum_{i\in\nn}|\lambda_i^k|^p
\r)^{\frac qp}\r]^{\frac 1q}\sim
\|f\|_{H_A^{p,r,s,\varepsilon,q}(\rn)},
\end{eqnarray*}
which is the desired conclusion.

Now, we show \eqref{te40}.
Notice that, for all $k\in\mathbb{Z}$ and
$i\in\mathbb{N}$, $m_i^k$ is a
$(p,r,s,\varepsilon)$-molecule associated with
$x_i^k+B_{\ell_i^k}$,
$\sum_{i\in\mathbb{N}}\chi_{x_i^k+B_{\ell_i^k}}\ls1$,
$\lambda_i^k\sim2^k|B_{\ell_i^k}|^{1/p}$ and
$\widetilde{r}\in(1,r)$. By the H\"{o}lder inequality,
we find that, for $\sigma:=1-\frac p{\widetilde{r}\delta}>0$
and all $x\in\rn$,
\begin{eqnarray*}
\widetilde{\psi}_{k_0}(x)
&&\leq\sum_{k=-\infty}^{k_0-1}M_N\lf(\sum_
{i\in\mathbb{N}}\lambda_i^km_i^k\r)(x)\\
&&\leq\lf(\sum_{k=-\infty}^{k_0-1}2^{k\sigma
\wz{r}'}\r)^{1/\wz{r}'}\lf\{\sum_{k=-\infty}^
{k_0-1}2^{-k\sigma\wz{r}}\lf[M_N\lf(\sum_{i\in
\mathbb{N}}\lambda_i^km_i^k\r)(x)\r]^{\wz{r}}
\r\}^{1/\wz{r}}\\
&&\sim2^{k_0\sigma}\lf\{\sum_
{k=-\infty}^{k_0-1}2^{-k\sigma\wz{r}}\lf[M_N
\lf(\sum_{i\in\mathbb{N}}\lambda_i^km_i^k\r)(x)
\r]^{\wz{r}}\r\}^{1/\wz{r}}.
\end{eqnarray*}
This further implies that
\begin{eqnarray}\label{te44}
&&2^{k_0p}\lf[d_{\widetilde{\psi}_{k_0}}(2^{k_0})\r]^\delta\\
&&\hs\ls2^{k_0p}
2^{-k_0\widetilde{r}\delta(1-\sigma)}\lf\{\int_{\rn}
\sum_{k=-\infty}^{k_0-1}2^{-k\sigma\wz{r}}\lf[M_N
\lf(\sum_{i\in\mathbb{N}}\lambda_i^km_i^k\r)(x)\r]
^{\wz{r}}\,dx\r\}^\delta\noz\\
&&\hs\ls\lf[\sum_{k=-\infty}^{k_0-1}2^{-k\sigma\wz{r}}
\int_{\rn}\lf|\sum_{i\in\mathbb{N}}\lambda_i^km_i^
k(x)\r|^{\wz{r}}\,dx\r]^\delta=:\lf(\sum_{k=-\infty}^
{k_0-1}2^{-k\sigma\wz{r}}F_k\r)^\delta.\noz
\end{eqnarray}
Moreover, by the H\"{o}lder inequality,
we know that, for all
$k\in\mathbb{Z}\cap(-\infty,k_0-1]$,
\begin{eqnarray}\label{te41}
(F_k)^{1/\wz{r}}
&&\sim\sup_{\|g\|_{L^{\wz{r}'}(\rn)}
=1}\lf|\int_{\rn}\lf[\sum_{i\in\mathbb{N}}\lambda_i^k
m_i^k(x)\r]g(x)\,dx\r|\\
&&\ls\sup_{\|g\|_{L^{\wz{r}'}(\rn)}=1}\lf\{\sum_{i\in\nn}
|\lambda_i^k|\sum_{j\in\zz_+}\int_{U_j(x_i^k+B_{\ell_
i^k})}\lf|m_i^k(x)\r|\lf|g(x)\r|\,dx\r\}\noz\\
&&\ls\sup_{\|g\|_{L^{\wz{r}'}(\rn)}=1}\lf\{\sum_{i\in\nn}
|\lambda_i^k|\sum_{j\in\zz_+}\lf[\int_{U_j(x_i^k+B_{
\ell_i^k})}\lf|m_i^k(x)\r|^r\,dx\r]^{1/r}\r.\noz\\
&&\hs\lf.\times\lf[\int_{U_j(x_i^k+B_{\ell_i^k})}\lf|g(x)\r|
^{r'}\,dx\r]^{1/r'}\r\}\noz\\
&&\sim\sup_{\|g\|_{L^{\wz{r}'}(\rn)}=1}\lf\{\sum_{i\in\nn}
|\lambda_i^k|\sum_{j\in\zz_+}\lf\|m_i^k\r\|_{L^r(U_j
(x_i^k+B_{\ell_i^k}))}F_{i,j}^k\r\},\noz
\end{eqnarray}
where, for all
$k\in\mathbb{Z}\cap(-\infty,k_0-1]$, $i\in\mathbb{N}$ and
$j\in\zz_+$,
\begin{eqnarray}\label{te6}
F_{i,j}^k:=\lf[\int_{U_j(x_i^k+B_{\ell_i^k})}
\lf|g(x)\r|^{r'}\,dx\r]^{1/r'},
\end{eqnarray}
$U_0(x_i^k+B_{\ell_i^k}):=x_i^k+B_{\ell_i^k}$
and, for any $j\in\mathbb{N}$,
$$\ U_j(x_i^k+B_{\ell_i^k})
:=\lf[A^j(x_i^k+B_{\ell_i^k})\r]
\setminus\lf[A^{j-1}(x_i^k+B_
{\ell_i^k})\r].$$
By this, \eqref{te6} and \eqref{te58}, we find that, for all
$k\in\mathbb{Z}\cap(-\infty,k_0-1]$, $i\in\mathbb{N}$
and $j\in\zz_+$,
\begin{eqnarray}\label{te42}
F_{i,j}^k&&\leq\lf|A^jB_{\ell_i^k}\r|^{1/r'}
\lf[\frac1{\lf|A^jB_{\ell_i^k}\r|}\int_{A^j(x_i
^k+B_{\ell_i^k})}|g(x)|^{r'}\,dx\r]^{1/r'}\\
&&\ls\lf|A^jB_{\ell_i^k}\r|^{1/r'}\inf_{x\in x_i
^k+B_{\ell_i^k}}\lf\{M_{{\rm HL}}(|g|^{r'})(x)\r\}
^{1/r'}\noz\\
&&\ls\lf|A^jB_{\ell_i^k}\r|^{1/r'}\lf\{\frac1{\lf
|B_{\ell_i^k}\r|}\int_{x_i^k+B_{\ell_i^k}}\lf[M
_{HL}(|g|^{r'})(x)\r]^{\wz{r}'/r'}\,dx\r\}^
{1/\wz{r}'},\noz
\end{eqnarray}
where $M_{{\rm HL}}(f)$ is the Hardy-Littlewood maximal
function as in \eqref{te58}. Notice that $M_{{\rm HL}}$ is
bounded on $L^r(\rn)$ for all $r\in(1,\infty)$ (see Lemma
\ref{tl4} and Remark \ref{tr1}),
$\{x_i^k+B_{\ell_i^k}\}_{i\in\mathbb{N}}$ is finitely
overlapped, $|\lambda_i^k|\sim2^k|B_{\ell_i^k}|^{1/p}$,
$\|m_i^k\|_{L^r(U_j(x_i^k+B_{\ell_i^k}))}
\leq b^{-j\varepsilon}|B_{\ell_i^k}|^{1/r-1/p}$,
$\varepsilon>1>1/r'$
and $r>\wz{r}$. Then, by \eqref{te41}, \eqref{te42} and
the H\"{o}lder inequality, for all
$k\in\mathbb{Z}\cap(-\infty,k_0-1]$, we have
\begin{eqnarray*}
F_k&&\ls\sup_{\|g\|_{L^{\wz{r}'}(\rn)}=1}\lf\{
\sum_{i\in\mathbb{N}}2^k\lf|B_{\ell_i^k}\r|^{1/p}
\sum_{j\in\zz_+}b^{-j\varepsilon}\lf|B_{\ell_i^k}\r|^
{1/r-1/p}b^{j/r'}\lf|B_{\ell_i^k}\r|^{1/r'}\r.\\
&&\ \ \lf.\times\lf[\frac1{\lf|B_{\ell_i^k}\r|}\int_
{x_i^k+B_{\ell_i^k}}\lf\{M_{{\rm HL}}(|g|^{r'})(x)\r\}
^{\wz{r}'/r'}\,dx\r]^{1/\wz{r}'}\r\}^{\wz{r}}\noz\\
&&\ls\sup_{\|g\|_{L^{\wz{r}'}(\rn)}=1}\lf\{\lf[\sum
_{i\in\mathbb{N}}2^{k\wz{r}}\lf|B_{\ell_i^k}\r|\r]
^{1/\wz{r}}\lf[\sum_{i\in\mathbb{N}}\int_{x_i^k+
B_{\ell_i^k}}\lf\{M_{{\rm HL}}(|g|^{r'})(x)\r\}^{\wz{r}'/r'}
\,dx\r]^{1/\wz{r}'}\r\}^{\wz{r}}\noz\\
&&\ls\lf[\sum_{i\in\mathbb{N}}2^{k\wz{r}}\lf|B_{\ell
_i^k}\r|\r]\sup_{\|g\|_{L^{\wz{r}'}(\rn)}=1}\lf[\int_
{\rn}\lf\{M_{{\rm HL}}(|g|^{r'})(x)\r\}^{\wz{r}'/r'}\,dx\r]^
{\wz{r}/\wz{r}'}\noz\\
&&\ls\lf[\sum_{i\in\mathbb{N}}2^{k\wz{r}}\lf|B_{\ell_
i^k}\r|\r]\sup_{\|g\|_{L^{\wz{r}'}(\rn)}=1}\lf[\int_{\rn}
|g(x)|^{\wz{r}'}\,dx\r]^{\wz{r}/\wz{r}'}\ls\sum_{i\in\nn}
2^{k\wz{r}}\lf|B_{\ell_i^k}\r|.\noz
\end{eqnarray*}
By this and \eqref{te44}, we know that
\begin{eqnarray}\label{te45}
2^{k_0p}\lf[d_{\widetilde{\psi}_{k_0}}(2^{k_0})\r]^\delta
&&\ls\lf[\sum_{k=-\infty}^{k_0-1}2^{-k\sigma\wz{r}}
\sum_{i\in\mathbb{N}}2^{k\wz{r}}\lf|B_{\ell_i^k}\r|
\r]^\delta\\
&&\ls\sum_{k=-\infty}^{k_0-1}2^{k\widetilde{r}\delta
(1-\sigma)}\lf[\sum_{i\in\mathbb{N}}\lf|B_{\ell_i^k}
\r|\r]^\delta\sim\sum_{k=-\infty}^{k_0-1}\lf[2^k\mu
_k^\delta\r]^p,\noz
\end{eqnarray}
which is the first desired estimate in \eqref{te40}.

Now we establish the desired estimate of $\wz{\eta}_{k_0}$
in \eqref{te40}. It suffices to show that, for all
$i\in\mathbb{N}$ and $k\in\mathbb{Z}$,
\begin{eqnarray}\label{te47}
\int_\rn \lf[M_N(m_i^k)(x)\r]^{\delta p}dx
\ls\lf|B_{\ell_i^k}\r|^{1-\delta}.
\end{eqnarray}
Indeed, as in \eqref{te46}, by \eqref{te47}, we find that
\begin{eqnarray}\label{te48}
\ \ 2^{k_0\delta p}d_{\wz{\eta}_{k_0}}(2^{k_0})
&&\leq\sum_{k={k_0}}^\infty\sum_{i\in\nn}
|\lambda_i^k|^{\delta p}\int_\rn\lf[M_N(m_i^k)(x)
\r]^{\delta p}dx\\&&\ls\sum_{k={k_0}}^\infty
\sum_{i\in\mathbb{N}}|\lambda_i^k|^{\delta p}
\lf|B_{\ell_i^k}\r|^{1-\delta}\ls\sum_{k={k_0}}^\infty
2^{k\delta p}\sum_{i\in\mathbb{N}}\lf|B_{\ell_i^k}\r|
\sim\sum_{k=k_0}^\infty\lf[2^{k\delta}\mu_k\r]^p,\noz
\end{eqnarray}
which is desired.

In order to show \eqref{te47}, we write
\begin{eqnarray*}
&&\int_\rn \lf[M_N(m_i^k)(x)\r]^{\delta p}\,dx\\
&&\hs=\int_{x_i^k+B_{\ell_i^k+\tau}} \lf[M_N(m_i^
k)(x)\r]^{\delta p}\,dx+\int_{(x_i^k+B_{\ell_i^
k+\tau})^\complement}\cdots
=:\wz{{\rm I}}_1+\wz{{\rm I}}_2.
\end{eqnarray*}
For $r\in(1,\fz]$, by the H\"{o}lder inequality,
Proposition \ref{tl4} and Remark \ref{tr1}, we have
\begin{eqnarray}\label{te49}
\wz{{\rm I}}_1&&=\int_{x_i^k+B_{\ell_i^k+\tau}}
\lf[M_N(m_i^k)(x)\r]^{\delta p}\,dx\\
&&\leq\lf\{\int_{x_i^
k+B_{\ell_i^k+\tau}}\lf[M_N(m_i^k)(x)\r]^r\,dx\r\}
^{\frac {p\delta} r}\lf|B_{\ell_i^k+\tau}\r|^{1-
\frac {p\delta} r}\noz\\
&&\ls\sum_{j\in\zz_+}\lf\|m_i^k\r\|_{L^r(U_j({x_i
^k+B_{\ell_i^k}}))}^{{p\delta}}\lf|B_{\ell_i^k+
\tau}\r|^{1-\frac {p\delta} r}\noz\\
&&\ls\sum_{j\in\zz_+}b^{-j\varepsilon p\delta}\lf
|B_{\ell_i^k}\r|^{(\frac 1r-\frac 1p) p\delta}
\lf|B_{\ell_i^k+\tau}\r|^{1-\frac {p\delta} r}
\sim\lf|B_{\ell_i^k}\r|^{1-\delta}.\noz
\end{eqnarray}
To estimate $\wz{{\rm I}}_2$, we only need to prove that,
for all $i\in\mathbb{N}$, $k\in\mathbb{Z}$ and
$x\in (x_i^k+B_{\ell_i^k+\tau})^\complement$,
\begin{eqnarray}\label{te52}
M_N^0(m_i^k)(x)\ls\lf|B_{\ell_i^k}\r|^{-\frac 1p}\frac
{\lf|B_{\ell_i^k}\r|^\beta}{\lf[\rho(x-x_i^k)\r]^\beta},
\end{eqnarray}
where $M_N^0(f)$ denotes the radial grand maximal function of
$f$ as in Definition \ref{d-mf}, $\rho$ is the homogeneous quasi-norm associated
with dilation $A$ and $\beta:=(\frac{\ln b}{\ln \lambda_-}+N-1)
\frac{\ln \lambda_-}{\ln b}$. Indeed, noticing that $\beta\delta p>1$,
as in \eqref{te27}, by Proposition \ref{sp1} and \eqref{te52},
we have
\begin{eqnarray*}
\wz{{\rm I}}_2&&=\int_{(x_i^k+B_{\ell_i^k+\tau})^
\complement} \lf[M_N(m_i^k)(x)\r]^{\delta p}\,dx\\
&&\ls\int_{\rho(x-x_i^k)\geq|B_{\ell_i^k+\tau}|}
\lf|B_{\ell_i^k}\r|^{-\delta}\frac{\lf|B_{\ell_i^k}
\r|^{\beta\delta p}}{\lf[\rho(x-x_i^k)\r]^{\beta\delta
p}}\,dx\sim\lf|B_{\ell_i^k}\r|^{1-\delta},\noz
\end{eqnarray*}
which, combined with \eqref{te49}, finishes
the proof of \eqref{te47}.

Thus, to obtain the desired conclusion of Case 1,
we only need to prove \eqref{te52}. To this end,
for any $i\in\mathbb{N}$ and $k\in\mathbb{Z}$,
take
$x\in(x_i^k+B_{\ell_i^k+\tau})^
\complement$, $\varphi\in\cs_N(\rn)$
and $t\in\mathbb{Z}$.
Suppose that $P$ is a polynomial of degree
 no more than $s$ which will be determined later.
Then, for all $i\in{\mathbb{N}}$ and
$k\in\mathbb{Z}$,
by the H\"{o}lder inequality, we have
\begin{eqnarray}\label{te36}
&&\lf|(m_i^k\ast\varphi_t)(x)\r|\\
&&\hs=b^{-t}\lf|\int_\rn m_i^k(y)\varphi\lf(A^{-t}
(x-y)\r)\,dy\r|\noz\\
&&\hs\le b^{-t}\sum_{j\in\zz_+}\lf|\int_{U_j(x_i^k+
B_{\ell_i^k})} m_i^k(y)\lf[\varphi\lf(A^{-t}(x-y)
\r)-P\lf(A^{-t}(x-y)\r)\r]\,dy\r|\noz\\
&&\hs\leq b^{-t}\sum_{j\in\zz_+}\lf\|m_i^k\r\|_
{L^r(U_j(x_i^k+B_{\ell_i^k}))}\noz\\
&&\hs\hs\times\lf[\int_{U_j(x_i^k+B_{\ell_i^k})}
\lf|\varphi\lf(A^{-t}(x-y)\r)-P\lf(A^{-t}(x-y)\r)\r|
^{r'}\,dy\r]^{1/r'}\noz\\
&&\hs\leq b^{-t}\sum_{j\in\zz_+}b^{-j\varepsilon}
\lf|B_{\ell_i^k}\r|^{1/r-1/p}b^{t/r'}\noz\\
&&\hs\hs\times\lf\{\int_{A^{-t+j}(x-x_i^k)+A^jB_{\ell
_i^k-t}}|\varphi(y)-P(y)|^{r'}\,dy\r\}^{1/r'}\noz\\
&&\hs\leq b^{-t}\lf|B_{\ell_i^k}\r|^{1/r-1/p}b^{t/r'}
b^{-t/r'}\lf|B_{\ell_i^k}\r|^{1/r'}\noz\\
&&\hs\hs\times\lf\{\sum_{j\in\zz_+}b^{-j\varepsilon}
b^{j/r'}\sup_{y\in A^{-t+j}(x-x_i^k)+A^jB_{\ell_
i^k-t}}|\varphi(y)-P(y)|\r\}\noz\\
&&\hs=\lf|B_{\ell_i^k}\r|^{-1/p}b^{\ell_i^k-t}\sum_
{j\in\zz_+}b^{-j\varepsilon}b^{j/r'}\sup_{y\in A^
{-t+j}(x-x_i^k)+A^jB_{\ell_i^k-t}}|\varphi(y)-P(y)|\noz.
\end{eqnarray}
Suppose that $x\in x_i^k+B_{\ell_i^k+\tau+m+1}
\setminus B_{\ell_i^k+\tau+m}$
for some $m\in\zz_+$. Then, by \eqref{se2}, we obtain
\begin{eqnarray}\label{te10}
\ \ \ \ \ A^{-t+j}(x-x_i^k)+A^jB_{\ell_i^k-t}
&&\subset A^{-t+j}(B_{\ell_i^k+\tau
+m+1}\setminus B_{\ell_i^k+\tau+m})
+A^jB_{\ell_i^k-t}\\
&&=A^{\ell_i^k-t+j}((B_{\tau+m+1}
\setminus B_{\tau+m})+B_0)\subset
A^{\ell_i^k-t+j}(B_m)^\complement.\noz
\end{eqnarray}

If $\ell_i^k\geq t$,
we choose $P\equiv0$. Then
\begin{eqnarray}\label{te37}
\ \ \ \ \ \ \sup_{y\in A^{-t+j}(x-x_i^k)
+A^jB_{\ell_i^k-t}}|\varphi(y)|
\leq\sup_{y\in A^{\ell_i^k-t+j}
(B_m)^\complement}\min\lf\{1,\
\rho(y)^{-N}\r\}\leq b^{-N(\ell_i^k-t+j+m)}.
\end{eqnarray}
If $\ell_i^k<t$, then we let $P$ be the Taylor expansion
of $\varphi$ at the point $A^{-t+j}(x-x_i^k)$ of order $s$.
By the Taylor remainder theorem,
\eqref{se5}, \eqref{se19} and \eqref{te10}, we have
\begin{eqnarray}\label{te38}
&&\sup_{y\in A^{-t+j}(x-x_i^k)+
A^jB_{\ell_i^k-t}}|\varphi(y)-P(y)|\\
&&\quad\ls\sup_{z\in A^jB_{\ell_i^
k-t}}\ \sup_{|\alpha|=s+1}\lf|\partial
^\alpha\varphi(A^{-t+j}(x-x_i^k)+z)\r|
|z|^{s+1}\noz\\
&&\quad\ls b^{j(s+1)\log_b(\lz_+)}\lambda_-^{(s+1)
(\ell_i^k-t)}\sup_{y\in A^{-t+j}(x-x_i^k)
+A^jB_{\ell_i^k-t}}\min\lf\{1,\rho(y)^{-N}
\r\}\noz\\
&&\quad\ls b^{j(s+1)\log_b(\lz_+)}\lambda_-^{(s+1)
(\ell_i^k-t)}\sup_{y\in A^{\ell_i^k-t+j}
(B_m)^\complement}\min\lf\{1,\rho(y)^{-N}
\r\}\noz\\
&&\quad\ls b^{j(s+1)\log_b(\lz_+)}\lambda_-^{(s+1)
(\ell_i^k-t)}\min\lf\{1,b^{-N(\ell_i^k-t+j+m
)}\r\}.\noz
\end{eqnarray}
Take $s:=\lfloor(\frac1p-1)\frac{\ln b}{\ln \lambda_-}\rfloor$.
Since $N=s+2$, it follows that $b\lambda_-^{s+1}\leq b^N$. By this,
\eqref{te36}, \eqref{te37}, \eqref{te38} and
$\varepsilon>(s+1)\log_b(\lz_+)$, for all
$x\in x_i^k+B_{\ell_i^k+\tau+m+1}
\setminus B_{\ell_i^k+\tau+m}$ with $m\in\zz_+$,
we find that
\begin{eqnarray*}
\lf[M_N^0(m_i^k)(x)\r]^p=&&\sup_{\varphi\in\cs_N(\rn)}
\sup_{t\in\mathbb{Z}}\lf|(m_i^k\ast\varphi_t)
(x)\r|^p\\
\ls&&\lf|B_{\ell_i^k}\r|^{-1}\sum_{j\in\zz_+}
b^{-jp(\varepsilon-1/r')}\max\lf\{\sup_{t\in\zz,\,
t\leq\ell_i^k}b^{p(\ell_i^k-t)}b^{-Np(\ell_
i^k-t+j+m)},\r.\noz\\
&&\lf.\sup_{t\in\mathbb{Z},\,t>\ell_i^k}b^
{p(\ell_i^k-t)}b^{jp(s+1)\log_b(\lz_+)}\lambda_-^{p(s+1)
(\ell_i^k-t)}
\min\lf\{1,b^{-Np(\ell_i^k-t+j+m)}\r\}\r\}\noz\\
\ls&&\lf|B_{\ell_i^k}\r|^{-1}
\sum_{j\in\zz_+}b^{-jp[\varepsilon-(s+1)\log_b(\lz_+)+1-1/r']}
\max\lf\{b^{-Npm},(b\lambda_-^{s+1})^{-pm}\r\}\\
\ls&&\lf|B_{\ell_
i^k}\r|^{-1}(b\lambda_-^{s+1})^{-pm}\noz.
\end{eqnarray*}
Form this, as in \eqref{te7}, we easily deduce that \eqref{te52}
holds true for $q/p\in[1,\fz]$. This finishes the proof of Case 1.

\emph{Case 2:} $q/p\in(0,1)$.
In this case, let
$\widetilde{\psi}_{k_0}$ and $\widetilde{\eta}_{k_0}$
be as in \eqref{3.41x}. Similar to \eqref{te45}, we have
\begin{eqnarray*}
2^{k_0p}\lf[d_{\widetilde{\psi}_{k_0}}(2^{k_0})\r]^\delta
&&\ls\lf[\sum_{k=-\infty}^{k_0-1}2^{-k\sigma\wz{r}}
\sum_{i\in\mathbb{N}}2^{k\wz{r}}\lf|B_{\ell_i^k}\r|
\r]^\delta\sim\lf(\sum_{k=-\infty}^{k_0-1}2^{\frac{kp}
\delta}\mu_k^p\r)^\delta,
\end{eqnarray*}
which further implies that
\begin{eqnarray}\label{te70}
\ \ \sum_{k_0\in\zz}2^{k_0q}\lf|\lf\{x\in\rn:\
\widetilde{\psi}_{k_0}(x)>2^{k_0}\r\}\r|^{\frac qp}
&&\ls\sum_{k_0\in\zz}2^{k_0(q-\frac q\delta)}\sum_
{k=-\infty}^{k_0-1}2^{\frac{kq}\delta}\mu_k^q\\
&&\sim\sum_{k\in\zz}\sum_{k_0=k+1}^{\fz}2^
{k_0(q-\frac q\delta)}2^{\frac{kq}\delta}\mu_k^q
\ls\sum_{k\in\zz}2^{kq}\mu_k^q.\noz
\end{eqnarray}
On the other hand, similar to \eqref{te48}, we find that
\begin{eqnarray*}
2^{k_0\delta p}d_{\widetilde{\eta}_{k_0}}(2^{k_0})
\ls\sum_{k=k_0}^\infty\lf[2^{k\delta}\mu_k\r]^p,
\end{eqnarray*}
which implies that
\begin{eqnarray*}
&&2^{k_0\delta p}\lf|\lf\{x\in\rn:\ \widetilde{\eta}_{k_0}(x)
>2^{k_0}\r\}\r|\\
&&\hs\ls\sum_{k=k_0}^\fz2^{-k\widetilde{\delta}
p}\lf[2^{k(1-\widetilde{\delta})}\mu_k\r]^p\ls2^{-k_0
\widetilde{\delta}p}\lf\{\sum_{k=k_0}^\fz\lf[2^{k(1-
\widetilde{\delta})}\mu_k\r]^q\r\}^{\frac pq},
\end{eqnarray*}
where $\widetilde{\delta}:=\frac{1-\delta}2$. Therefore, we have
\begin{eqnarray}\label{te71}
&&\sum_{k_0\in\zz}2^{k_0q}\lf|
\lf\{x\in\rn:\ \widetilde{\eta}_{k_0}(x)>2^{k_0}\r\}\r|
^{\frac qp}\\
&&\hs\ls\sum_{k_0\in\zz}2^{k_0\widetilde{\delta}q}\sum_
{k=k_0}^\fz\lf[2^{k(1-\widetilde{\delta})}\mu_k\r]^q\noz\\
&&\hs\sim\sum_{k\in\zz}\lf[2^{k(1-\widetilde{\delta})}\mu_
k\r]^q\sum_{k_0=-\fz}^k2^{k_0\widetilde{\delta}q}
\ls\sum_{k\in\zz}2^{kq}\mu_k^q.\noz
\end{eqnarray}
Notice that
$\mu_k:=(\sum_{i\in\mathbb{N}}|B_{\ell_i^k}|)^{1/p}$
and $\lambda_i^k\sim2^k|B_{\ell_i^k}|^{1/p}$.
Combining \eqref{se6}, \eqref{te70}, \eqref{te71} and
\eqref{te67}, we further conclude that
\begin{eqnarray*}
\|M_N(f)\|_{L^{p,q}(\rn)}^q\sim&&\sum_{k_0\in\zz}2^{k_0
q}\lf|\lf\{x\in\rn:\ M_N(f)(x)>2^{k_0}\r\}\r|^{\frac qp}\\
\ls&&\sum_{k_0\in\zz}2^{k_0q}\lf|
\lf\{x\in\rn:\ \widetilde{\psi}_{k_0}(x)>2^{k_0}\r\}\r|^
{\frac qp}\\
&&+\sum_{k_0\in\zz}2^{k_0q}\lf|\lf\{x\in\rn:\
\widetilde{\eta}_{k_0}(x)>2^{k_0}\r\}\r|^{\frac qp}\noz\\
\ls&&\sum_{k\in\zz}2^{kq}\mu_k^q\sim\sum_{k\in\zz}
\lf[\sum_{i\in\nn}|\lambda_i^k|^p\r]^{\frac qp}
\sim\|f\|_{H_A^{p,r,s,\varepsilon,q}(\rn)}^q,
\end{eqnarray*}
which implies that
$\|f\|_{H^{p,q}_A(\rn)}\ls\|f\|_{H_A^{p,r,s,\varepsilon,q}(\rn)}$.
This finishes the proof of Case 2 and hence Theorem \ref{tt2}.
\end{proof}

\section{Maximal function characterizations of $H^{p,q}_A(\rn)$\label{s4}}

\hskip\parindent
In this section, we  characterize $H^{p,q}_A(\rn)$
in terms of the radial maximal function $M_\varphi^0$ (see \eqref{se10})
or the non-tangential maximal function $M_\varphi$ (see \eqref{se25}).
We begin with the following Definitions \ref{d-4.1} and \ref{d-4.2}
from \cite{mb03}.

\begin{definition}\label{d-4.1}
For any function
$F:\ \rn\times\mathbb{Z}\rightarrow [0,\infty)$,
$K\in\mathbb{Z}\cup\{\infty\}$ and $\ell\in\mathbb{Z}$,
the maximal function of $F$ with \emph{aperture}
$\ell$ is defined by
\begin{eqnarray*}
F^{* K}_\ell(x):=
\sup_{k\in\mathbb{Z},\,k\leq K}
\sup_{y\in x+B_{k+\ell}}F(y,k),
\ \ \ \forall\ x\in\rn.
\end{eqnarray*}
\end{definition}

\begin{definition}\label{d-4.2}
Let $K\in\zz$ and $L\in[0,\fz)$. For $\varphi\in\cs$,
the \emph{radial maximal function} $M_\varphi^{0(K,L)}(f)$,
the \emph{non-tangential maximal function} $M_\varphi^{(K,L)}(f)$ and
the \emph{tangential maximal function} $T_\varphi^{N(K,L)}(f)$ of
$f\in\cs'(\rn)$ are, respectively, defined by
setting, for all $x\in\rn$,
\begin{eqnarray*}
M_\varphi^{0(K,L)}(f)(x):=\sup_{k\in\mathbb{Z},\,k\leq K}
|(f\ast\varphi_k)(x)|\lf[\max\lf\{1,\rho\lf(A^{-K}x\r)\r\}\r]^
{-L}\lf(1+b^{-k-K}\r)^{-L},
\end{eqnarray*}
\begin{eqnarray*}
M_\varphi^{(K,L)}(f)(x):=\sup_{k\in\mathbb{Z},\,k\leq K}
\sup_{y\in x+B_k}|(f\ast\varphi_k)(y)|\lf[\max\lf\{1,\rho
\lf(A^{-K}y\r)\r\}\r]^{-L}\lf(1+b^{-k-K}\r)^{-L}
\end{eqnarray*}
and
\begin{eqnarray*}
T_\varphi^{N(K,L)}(f)(x):=\sup_{k\in\mathbb{Z},\,k\leq K}
\sup_{y\in\rn}\frac{|(f\ast\varphi_k)(y)|}{[\max\lf\{1,\rho
\lf(A^{-k}(x-y)\r)\r\}]^{N}}\frac{\lf(1+b^{-k-K}\r)^{-L}}
{[\max\lf\{1,\rho\lf(A^{-K}y\r)\r\}]^{L}}.
\end{eqnarray*}
Furthermore, the \emph{radial grand maximal function} $M_N^{0(K,L)}(f)$
and the \emph{non-tangential grand maximal function}
$M_N^{(K,L)}(f)$ of
$f\in\cs'(\rn)$ are, respectively, defined by
setting, for all $x\in\rn$,
\begin{eqnarray*}
M_N^{0(K,L)}(f)(x):=
\sup_{\varphi\in\cs_N(\rn)}M_\varphi^{0(K,L)}(f)(x)
\end{eqnarray*} and
\begin{eqnarray*}
M_N^{(K,L)}(f)(x):=
\sup_{\varphi\in\cs_N(\rn)}M_\varphi^{(K,L)}(f)(x).
\end{eqnarray*}
\end{definition}

The following Lemmas \ref{fl1} through \ref{fl5} are just
\cite[p.\,42, Lemma 7.2, p.\,45, Lemma 7.5 and p.\,46, Lemma 7.6]{mb03}, respectively.

\begin{lemma}\label{fl1}
There exists a positive constant $C$ such that,
for all functions
$F :\ \rn\times\mathbb{Z}\rightarrow [0,\infty)$,
$\ell\in[\ell',\fz)\cap\mathbb{Z}$,
$K\in\mathbb{Z}\cup\{\infty\}$ and $\lambda\in(0,\fz)$,
$$\lf|\lf\{x\in\rn:\ F^{* K}_\ell(x)>
\lambda\r\}\r|\leq Cb^{\ell-\ell'}\lf|\lf
\{x\in\rn:\ F^{* K}_{\ell'}(x)>\lambda\r\}\r|.$$
\end{lemma}
\begin{lemma}\label{fl2}
Suppose that $\varphi\in\cs(\rn)$ with
$\int_{\rn}\varphi(x)\,dx\neq0$.
For any given $N\in\mathbb{N}$ and $L\in[0,\fz)$, there
exist an $I\in\mathbb{N}$ and a positive constant $C_{(L)}$,
depending on $L$,
such that, for all $K\in\zz_+$ and
$f\in\cs'(\rn)$,
$$M_I^{0(K,L)}(f)(x)\leq C_{(L)}T_\varphi^{N(K,L)}(f)(x),
\ \ \ \ \ \forall\ x\in\rn.$$
\end{lemma}

\begin{lemma}\label{fl5}
Suppose that $p\in(0,\fz),\,\varphi\in\cs(\rn)$
and $K\in\zz_+$. Then, for any given $M\in(0,\fz)$,
there exist $L\in(0,\fz)$ and a positive constant $C_{(K,\,M)}$, such that,
for all $f\in\cs'(\rn)$ and $x\in\rn$,
\begin{eqnarray}\label{fe13}
M_\varphi^{(K,L)}(f)(x)
\leq C_{(K,\,M)}\lf[\max\lf\{1,\rho(x)\r\}\r]^{-M}.
\end{eqnarray}
\end{lemma}

\begin{lemma}\label{fl3}
Let $p\in(0,\fz),\,N\in(1/p,\fz)\cap\mathbb{N},\,q\in(0,\fz]$
and $\varphi\in\cs(\rn)$.
Then there exists a positive constant $C$ such that, for all
$K\in\mathbb{Z}$, $L\in[0,\fz)$ and $f\in\cs'(\rn)$,
\begin{eqnarray}\label{fe2}
\lf\|T_\varphi^{N(K,L)}(f)\r\|_{L^{p,q}(\rn)}
\leq C\lf\|M_\varphi^{(K,L)}(f)\r\|_{L^{p,q}(\rn)}.
\end{eqnarray}
\end{lemma}

\begin{proof}
We first prove that, for all $p\in(0,\fz)$, $q\in(0,\fz]$,
$K\in\mathbb{Z}$ and $\ell\in[\ell',\fz)\cap\mathbb{Z}$,
\begin{eqnarray}\label{fe8}
\lf\|F_\ell^{* K}\r\|_{L^{p,q}(\rn)}
\ls b^{(\ell-\ell')/p}\lf\|F_{\ell'}^{* K}\r\|
_{L^{p,q}(\rn)},
\end{eqnarray}
where $F_\ell^{* K}$ is as in Definition \ref{d-4.1} and,
for all $\varphi\in\cs(\rn)$, $f\in\cs'(\rn)$,
$k\in\zz$ and $y\in\rn$,
\begin{eqnarray}\label{fe1}
F(y,k):=|(f\ast\varphi_k)(y)|\max\lf[\lf\{1,\rho
\lf(A^{-K}y\r)\r\}\r]^{-L}\lf(1+b^{-k-K}\r)^{-L}.
\end{eqnarray}
To this end, fix $x\in\rn$. For $k\in(-\fz,K]\cap\zz$, if $x-y\in B_{k+1}$,
then
\begin{eqnarray}\label{fe5}
F(y,k)\lf[\max\lf\{1,\rho\lf(A^{-k}(x-y)\r)\r\}\r]^{-N}
\leq F^{* K}_1(x);
\end{eqnarray}
if $x-y\in B_{k+j+1}\setminus B_{k+j}$
for some $j\in\mathbb{N}$, then
\begin{eqnarray}\label{fe7}
F(y,k)\lf[\max\lf\{1,\rho\lf(A^{-k}(x-y)\r)\r\}\r]^{-N}
\leq F^{* K}_{j+1}(x)b^{-jN}.
\end{eqnarray}
By taking supremum over all $k\in(-\fz,K]\cap\zz$
and $y\in\rn$ on the both sides of \eqref{fe5} and \eqref{fe7},
\eqref{fe1} and the definition
of $T_\varphi^{N(K,L)}$,
we further find that
\begin{eqnarray}\label{fe3}
T_\varphi^{N(K,L)}(f)(x)
\leq\sum_{j=0}^\infty F^{* K}_{j+1}(x)b^{-jN}.
\end{eqnarray}
Notice that $F^{* K}_{\ell'}$ is a non-negative function
for any $\ell'\in\zz$. By \eqref{se6} and Lemma \ref{fl1},
for $\ell\in[\ell',\fz)\cap\mathbb{Z}$, $K\in\mathbb{Z}$ and $p,\,q\in(0,\fz)$,
we conclude that
\begin{eqnarray}\label{fe4}
\lf\|F_\ell^{* K}\r\|_{L^{p,q}(\rn)}
&&\sim\lf[\int_0^\infty\lambda^{q-1}\lf|\lf\{
x\in\rn:\ F_\ell^{* K}(x)>\lambda\r\}\r|^
{q/p}\,d\lambda\r]^{1/q}\\
&&\ls b^{(\ell-\ell')/p}\lf[\int_0^\infty\lambda
^{q-1}\lf|\lf\{x\in\rn:\ F_{\ell'}^{* K}(x)>
\lambda\r\}\r|^{q/p}\,d\lambda\r]^{1/q}\noz\\
&&\sim b^{(\ell-\ell')/p}\lf\|F_{\ell'}^{* K}\r\|
_{L^{p,q}(\rn)}.\noz
\end{eqnarray}
It is easy to see that \eqref{fe4} is also true for
$q=\infty$ by the definition of $L^{p,\infty}(\rn)$ norm.
This proves \eqref{fe8}.

Now we show \eqref{fe2}. By \eqref{fe3}, the Aoki-Rolewicz theorem
(see \cite{ta42, sr57}), \eqref{fe8} and
$N\in(1/p,\fz)\cap\mathbb{N}$, we know that
there exists $\upsilon\in(0,1]$ such that
\begin{eqnarray*}
\lf\|T_\varphi^{N(K,L)}(f)\r\|_{L^{p,q}(\rn)}^\upsilon
&&\leq \sum_{j=0}^\infty b^{-jN\upsilon}\lf\|F^{* K}
_{j+1}\r\|_{L^{p,q}(\rn)}^\upsilon\\
&&\ls \sum_{j=0}^\infty b^{-jN\upsilon}b^{(j+1)
\upsilon/p}\lf\|F^{* K}_0\r\|_{L^{p,q}(\rn)}^\upsilon
\ls\lf\|M_\varphi^{(K,L)}(f)\r\|_{L^{p,q}(\rn)}^\upsilon\noz,
\end{eqnarray*}
which implies \eqref{fe2} and
hence completes the proof of Lemma \ref{fl3}.
\end{proof}

\begin{lemma}\label{fl4}
Suppose that $p\in(1,\fz)$ and $q\in(0,\fz]$.
Then there exists a positive constant $C$
such that, for all $f\in L^{p,q}(\rn)$,
\begin{eqnarray}\label{fe6}
\lf\|M_\mathcal{F}(f)\r\|_{L^{p,q}(\rn)}
\leq C\|f\|_{L^{p,q}(\rn)},
\end{eqnarray}
where $M_\mathcal{F}(f)$ is defined by \eqref{te59}.
\end{lemma}
\begin{proof} Let $E\subset\rn$ be an arbitrary
measurable set and $|E|<\infty$.
By \eqref{te56} and \eqref{te57}, we have
\begin{eqnarray*}
\lf\|M_\mathcal{F}(\chi_E)\r\|_{L^{1,\infty}(\rn)}
\ls\|\chi_E\|_{L^1(\rn)}\sim|E|
\end{eqnarray*}
and
\begin{eqnarray*}
\lf\|M_\mathcal{F}(\chi_E)\r\|_{L^\infty(\rn)}
\ls\|\chi_E\|_{L^\infty(\rn)}\ls1.
\end{eqnarray*}
Thus, applying \cite[Theorem 1.1 and
Remark 1.4]{lly11} to $M_\mathcal{F}$
and $f\in L^{p,q}(\rn)$ with
$p_0=q_0=1$ and $p_1=q_1=\infty$, we obtain \eqref{fe6}.
This finishes the proof of Lemma \ref{fl4}.
\end{proof}

\begin{remark}\label{fr1}
As a corollary of Lemma \ref{fl4}, the operators
$M_N(f)$ in \eqref{se8} and $M_{{\rm HL}}(f)$
in \eqref{te58} also satisfy \eqref{fe6}.
\end{remark}

Now we state the main result of this section.

\begin{theorem}\label{ft1}
Suppose that $p\in(0,\fz),\,q\in(0,\fz]$ and $\varphi\in\cs(\rn)$
with $\int_{\rn}\varphi(x)\,dx\neq0$. Then, for all $f\in\cs'(\rn)$,
the following statements are equivalent:
\begin{eqnarray}\label{fe9}
f\in H^{p,q}_A(\rn);   \end{eqnarray}
\begin{eqnarray}\label{fe10}
M_\varphi(f)\in L^{p,q}(\rn);   \end{eqnarray}
\begin{eqnarray}\label{fe11}
M_\varphi^0(f)\in L^{p,q}(\rn).   \end{eqnarray}
In this case,
\begin{eqnarray*}
\|f\|_{H^{p,q}_A(\rn)}\leq C_1\lf\|M_\varphi^0(f)
\r\|_{L^{p,q}(\rn)}\leq C_1\|M_\varphi(f)\|_
{L^{p,q}(\rn)}\leq C_2\|f\|_{H^{p,q}_A(\rn)},
\end{eqnarray*}
where $C_1$ and $C_2$ are positive constants independent
of $f$.
\end{theorem}

\begin{proof}
Clearly, \eqref{fe9} implies \eqref{fe10} and
\eqref{fe10} implies \eqref{fe11}.
Thus, to prove Theorem \ref{ft1}, it suffices
to show that \eqref{fe10} implies \eqref{fe9}
and \eqref{fe11} implies \eqref{fe10}.

We first prove that \eqref{fe10} implies \eqref{fe9}.
To this end, notice that, by Lemma
\ref{fl2} with $N\in(1/p,\fz)\cap\mathbb{N}$ and
$L=0$, we find that there exists an $I\in\nn$ such that
$M_I^{0(K,0)}(f)(x)\ls T_\varphi^{N(K,0)}(f)(x)$
for all $K\in\zz_+$, $f\in\cs'(\rn)$ and $x\in\rn$.
From this and Lemma \ref{fl3}, we  further deduce that,
for all
$K\in\mathbb{Z}_+$ and $f\in\cs'(\rn)$,
\begin{eqnarray}\label{fe12}
\lf\|M_I^{0(K,0)}(f)\r\|_{L^{p,q}(\rn)}\ls
\lf\|M_\varphi^{(K,0)}(f)\r\|_{L^{p,q}(\rn)}.
\end{eqnarray}
Letting $K\rightarrow\infty$ in \eqref{fe12}, by
\cite[Proposition 1.4.5(8)]{lg08} and the Fatou lemma,
we know that
$$\lf\|M_I^0(f)\r\|_{L^{p,q}(\rn)}
\ls\lf\|M_\varphi(f)\r\|_{L^{p,q}(\rn)},$$
which, together with Proposition \ref{sp1}, shows that
\eqref{fe10} implies \eqref{fe9}.

Now we show that \eqref{fe11} implies \eqref{fe10}.
Suppose now that $M_\varphi^0(f)\in L^{p,q}(\rn)$.
By Lemma \ref{fl5}, we find that there exists $L\in(0,\fz)$
such that \eqref{fe13} holds true, which further implies that
$M_\varphi^{(K,L)}(f)\in L^{p,q}(\rn)$ for all $K\in\zz_+$.
Indeed, for $q/p\in(0,1]$, by \eqref{se6} and
\eqref{fe13}, we have
\begin{eqnarray*}
&&\lf\|M_\varphi^{(K,L)}(f)\r\|_{L^{p,q}(\rn)}^q\\
&&\hs\sim\int_0^\infty\lambda^{q-1}\lf|\lf
\{x\in\rn:\ M_\varphi^{(K,L)}(f)(x)>\lambda\r\}
\r|^{q/p}\,d\lambda\\
&&\hs\ls\int_0^\infty\lambda^{q-1}\lf|\lf\{x\in B_1:\
M_\varphi^{(K,L)}(f)(x)>\lambda\r\}\r|^{q/p}
\,d\lambda\\
&&\hs\hs+\sum_{j=1}^\infty\int_0^\infty\lambda^
{q-1}\lf|\lf\{x\in B_{j+1}\setminus B_j:\ M_\varphi
^{(K,L)}(f)(x)>\lambda\r\}\r|^{q/p}\,d\lambda\\
&&\hs\ls\int_0^1
\lambda^{q-1}\lf|B_1\r|^{q/p}\,d\lambda
+\sum_{j=1}^\infty\int_0^{b^{-jM}}\lambda^
{q-1}\lf|B_{j+1}\r|^{q/p}\,d\lambda\\
&&\hs\sim\sum_{j=0}^\infty b^{-jMq}b^{(j+1)q/p}
\sim1\ \ \ \ \ \ \ {\rm as}\ \ M>1/p.
\end{eqnarray*}
For another case when $q/p\in(1,\fz)$, by \eqref{se6},
the Minkowski integral inequality and \eqref{fe13}, we find that
\begin{eqnarray}\label{fe15}
&&\lf\|M_\varphi^{(K,L)}(f)\r\|_{L^{p,q}(\rn)}^q\\
&&\hs\ls\lf\|M_\varphi^{(K,L)}(f)\r\|_{L^{p,q}(B_1)}^q
+\lf\|M_\varphi^{(K,L)}(f)\r\|_{L^{p,q}(\rn
\setminus B_1)}^q\noz\\
&&\hs\sim\int_0^\infty\lambda^{q-1}\lf|\lf\{x\in B_1:\
M_\varphi^{(K,L)}(f)(x)>\lambda\r\}\r|^{q/p}
\,d\lambda\noz\\
&&\hs\hs+\lf[\sum_{k\in\zz}2^{kq}\lf(\sum_{j=1}^\infty
\lf|\lf\{x\in B_{j+1}\setminus B_j:\ M_\varphi^{(K,L)}
f(x)>2^k\r\}\r|\r)^{\frac qp}\r]^{\frac pq\cdot\frac qp}
\noz\\
&&\hs\ls\int_0^\infty\lambda^{q-1}\lf|\lf\{x\in B_1:\ M_\varphi
^{(K,L)}(f)(x)>\lambda\r\}\r|^{q/p}\,d\lambda\noz\\
&&\hs\hs+\lf[\sum_{j=1}^\infty\lf(\sum_{k\in\zz}2^{kq}
\lf|\lf\{x\in B_{j+1}\setminus B_j:\ M_\varphi^{(K,L)}
(f)(x)>2^k\r\}\r|^{\frac qp}\r)^{\frac pq}\r]^{\frac qp}\noz\\
&&\hs\sim\int_0^\infty\lambda^{q-1}\lf|\lf\{x\in B_1:\
M_\varphi^{(K,L)}(f)(x)>\lambda\r\}\r|^{q/p}\,d\lambda\noz\\
&&\hs\hs+\lf[\sum_{j=1}^\infty\lf(\int_0^\infty\lambda^{q-1}\lf|\lf\{
x\in B_{j+1}\setminus B_j:\ M_\varphi^{(K,L)}(f)(x)>\lambda
\r\}\r|^{\frac qp}\,d\lambda\r)^{\frac pq}\r]^{\frac qp}\noz\\
&&\hs\ls\int_0^1\lambda^{q-1}\lf|B_1\r|^{q/p}\,d\lambda+
\lf[\sum_{j=1}^\infty\lf(\int_0^{b^{-jM}}\lambda^{q-1}\lf|
B_{j+1}\r|^{q/p}\,d\lambda\r)^{\frac pq}\r]^{\frac qp}\noz\\
&&\hs\ls\lf[\sum_{j=0}^\infty b^{-jMp}b^{(j+1)}\r]^{\frac qp}
\sim1\ \ \ \ \ \ \ \ {\rm as}\ \ M>1/p.\noz
\end{eqnarray}
Clearly, \eqref{fe15} also holds true for $q=\infty$,
since
$$\|f\|_{L^{p,\infty}(\rn)}=\sup_{\lambda
\in(0,\fz)}\lf\{\lambda d_f^{1/p}(\lambda)\r\}
\sim\sup_{k\in\zz} 2^k \lf[d_f(2^k)\r]^{\frac{1}{p}}.$$

On the other hand, by Lemmas \ref{fl2} and \ref{fl3},
we know that, for any $L\in(0,\fz)$, there exists  some $I\in\nn$
such that, for all $K\in\zz_+$ and $f\in\cs'(\rn)$,
$$\lf\|M_I^{0(K,L)}(f)\r\|_{L^{p,q}(\rn)}\leq C_3
\lf\|M_\varphi^{(K,L)}(f)\r\|_{L^{p,q}(\rn)},$$
where $C_3$ is a positive constant independent of $K$.
For any given $K\in\zz_+$, let
\begin{eqnarray}\label{fe22}
\Omega_K:=\lf\{x\in\rn:\ M_I^{0(K,L)}(f)(x)\leq C_4
M_\varphi^{(K,L)}(f)(x)\r\}\end{eqnarray}
with $C_4:=2C_3$. Then
\begin{eqnarray}\label{fe16}
\lf\|M_\varphi^{(K,L)}(f)\r\|_{L^{p,q}(\rn)}\ls
\lf\|M_\varphi^{(K,L)}(f)\r\|_{L^{p,q}(\Omega_K)},
\end{eqnarray}
since
$$\lf\|M_\varphi^{(K,L)}(f)\r\|_{L^{p,q}(\Omega_K
^\complement)}\leq C_4^{-1}\lf\|M_I^{0(K,L)}(f)\r\|
_{L^{p,q}(\Omega_K^\complement)}\leq C_3/C_4
\lf\|M_\varphi^{(K,L)}(f)\r\|_{L^{p,q}(\rn)}.$$

To finish the proof that \eqref{fe11} implies \eqref{fe10},
for any given $L\in[0,\fz)$,
it suffices to
show that, for all $t\in(0,p)$, $K\in\zz_+$ and $f\in\cs'(\rn)$,
\begin{eqnarray}\label{fe17}
M_\varphi^{(K,L)}(f)(x)\ls\lf\{M_{{\rm HL}}
\lf(\lf[M_\varphi^{0(K,L)}(f)\r]^t\r)(x)\r\}
^{1/t},\ \ \ \ \ \ \ \forall\ x\in\Omega_K.
\end{eqnarray}
Indeed, if \eqref{fe17} holds true, then, by \eqref{fe16},
\eqref{se22}, \eqref{fe17} and Remark \ref{fr1},
for all $K\in\zz_+$ and $f\in\cs'(\rn)$, we have
\begin{eqnarray}\label{fe18}
\lf\|M_\varphi^{(K,L)}(f)\r\|_{L^{p,q}(\rn)}^t
&&\ls\lf\|M_\varphi^{(K,L)}(f)\r\|_{L^{p,q}
(\Omega_K)}^t\\
&&\sim\lf\|\lf[M_\varphi^{(K,L)}(f)\r]^t\r\|_{L^
{\frac pt,\frac qt}(\Omega_K)}\noz\\
&&\ls\lf\|M_{{\rm HL}}\lf(\lf[M_\varphi^{0(K,L)}(f)\r]^
t\r)\r\|_{L^{\frac pt,\frac qt}(\Omega_K)}\noz\\
&&\ls\lf\|\lf[M_\varphi^{0(K,L)}(f)\r]^t\r\|_{L^{
\frac pt,\frac qt}(\rn)}\sim\lf\|M_\varphi^{0(K,L)}
(f)\r\|_{L^{p,q}(\rn)}^t.\noz
\end{eqnarray}
Noticing that
$M_\varphi^{(K,L)}(f)(x)$ and $M_\varphi^{0(K,L)}(f)(x)$
converge pointwise and monotonically to $M_\varphi (f)(x)$
and $M_\varphi^0(f)(x)$ for all $f\in\cs'(\rn)$ and $x\in\rn$, respectively, as
$K\rightarrow\infty$, by \cite[Proposition 1.4.5(8)]{lg08},
the monotone convergence theorem and \eqref{fe18},
we have
\begin{eqnarray*}
\lf\|M_\varphi (f)\r\|_{L^{p,q}(\rn)}
\ls\lf\|M_\varphi^0(f)\r\|_{L^{p,q}(\rn)},
\end{eqnarray*}
which shows that \eqref{fe11} implies \eqref{fe10}.

Thus, to complete the proof of Theorem \ref{ft1},
we only need to prove \eqref{fe17}. For this purpose,
for all $f\in\cs'(\rn)$, $k\in\zz$ and $y\in\rn$, let
\begin{eqnarray}\label{fe19}
F(y,k):=|(f\ast\varphi_k)(y)|\lf[\max\lf\{1,\rho
\lf(A^{-K}y\r)\r\}\r]^{-L}\lf[1+b^{-k-K}\r]^{-L}.
\end{eqnarray}
Let $x\in\Omega_K$. By the definitions of $F_0^{* K}$ and
$M_\varphi^{(K,L)}$ and \eqref{fe19}, it is easy to see that there exist
$k\in(-\fz,K]\cap\mathbb{N}$ and $y\in x+B_k$
 such that
\begin{eqnarray}\label{fe23}F(y,k)\geq F_0^{* K}(x)/2
=M_\varphi^{(K,L)}(f)(x)/2.
\end{eqnarray}
For this $y$, consider $\widetilde{x}\in y+B_{k-\ell}$ for some
integer $\ell\in\zz_+$ to be specified later. We write
\begin{eqnarray}\label{4.21x}
f\ast\varphi_k(\widetilde{x})-f\ast\varphi_k(z)=f\ast\Psi_k(z),
\ \ \ \ \ \forall\ z\in\rn,
\end{eqnarray}
where $\Psi(z):=\varphi(z+A^{-k}(\widetilde{x}-y))-\varphi(z)$
for all $z\in\rn$.
By \eqref{se24}, \eqref{se4}, the mean value theorem and
\eqref{se19}, we conclude that
\begin{eqnarray}\label{fe20}
\ \ \ \ \|\Psi\|_{\cs_I(\rn)}
&&\leq\sup_{h\in B_{-\ell}}\lf\|\varphi(\cdot
+h)-\varphi(\cdot)\r\|_{\cs_I(\rn)}\\
&&=\sup_{h\in B_{-\ell}}\sup_{z\in\rn}\sup_
{|\alpha|\leq I}\max\lf\{1,\lf[\rho(z)\r]^I\r\}
\lf|\partial^\alpha\varphi(z+h)-\partial^\alpha
\varphi(z)\r|\noz\\
&&\ls\lf[\sup_{h\in B_{-\ell}}\sup_{z\in\rn}
\sup_{|\alpha|\leq I+1}\max\lf\{1,\lf[\rho(z+h)
\r]^I\r\}\lf|\partial^\alpha\varphi(z+h)\r|\r]
\lf[\max_{h\in B_{-\ell}}|h|\r]\noz\\&&\ls C_5\lambda
_-^{-\ell},\noz
\end{eqnarray}
where the positive constant $C_5$ is
independent of $L$. Notice that,
by a proof similar to that of \cite[p.\,17, Proposition 3.10]{mb03},
we have, for all $x\in\rn$,
\begin{eqnarray}\label{fe21}
M_I^{(K,L)}(f)(x)\leq b^{\tau I}M_I^{0(K,L)}(f)(x).
\end{eqnarray}
Moreover, by \eqref{se4}, we know that, for all $\widetilde{x}\in y+B_{k-\ell}$,
$$\max\lf\{1,\rho\lf(A^{-K}\widetilde{x}\r)\r\}
\leq b^\tau\max\lf\{1,\rho\lf(A^{-K}y\r)\r\},$$
which, combined with
\eqref{4.21x}, \eqref{fe23}, \eqref{fe20}, \eqref{fe21}
and \eqref{fe22}, implies that
\begin{eqnarray}\label{fe24}
\ \ \ \ \ \ b^{\tau L}F(\widetilde{x},k)
&&\geq\lf[|f\ast\varphi_k(y)|-|f\ast\Psi_k(y)|\r]
\lf[\max\lf\{1,\rho\lf(A^{-K}y\r)\r\}\r]^{-L}\lf(1+b
^{-k-K}\r)^{-L}\\
&&\geq F(y,k)-M_I^{(K,L)}(f)(x)\|\Psi\|_{\cs_I(\rn)}\noz\\
&&\gs
M_\varphi^{(K,L)}(f)(x)/2-C_5\lambda_-^{-\ell}b^
{\tau I}M_I^{0(K,L)}(f)(x)\noz\\
&&\gs M_\varphi^{(K,L)}(f)(x)/2-C_4C_5\lambda_
-^{-\ell}b^{\tau I}M_\varphi^{(K,L)}(f)(x)\gs
M_\varphi^{(K,L)}(f)(x)/4,\noz
\end{eqnarray}
where $\ell$ is chosen to be the smallest integer such that
$C_4C_5\lambda_-^{-\ell}b^{\tau I}\leq1/4$.
Therefore, by \eqref{fe24} and \eqref{se1},
for all $t\in(0,p)$ and $x\in\Omega_K$, we have
\begin{eqnarray*}
\lf[M_\varphi^{(K,L)}(f)(x)\r]^t
&&\ls\frac{4^tb^{\tau Lt}}{\lf|B_{k-\ell}\r|}
\int_{y+B_{k-\ell}}\lf[F(z,k)\r]^t\,dz\\
&&\ls4^tb^{\tau Lt}\frac{b^{\tau+\ell}}{\lf
|B_{k+\tau}\r|}\int_{x+B_{k+\tau}}\lf[M_
\varphi^{0(K,L)}(f)\r]^t(z)\,dz\noz\\
&&\ls 4^tb^{\tau Lt}M_{{\rm HL}}\lf(\lf[M_\varphi^{0(K,L)}(f)
\r]^t\r)(x),\noz
\end{eqnarray*}
which implies \eqref{fe17}. This finishes the proof
of Theorem \ref{ft1}.
\end{proof}

\section{Finite atomic decomposition characterizations of $H^{p,q}_A(\rn)$\label{s7}}

\hskip\parindent
In this section, we obtain the finite atomic decomposition characterizations
of $H^{p,q}_A(\rn)$.
To be precise, we prove that, for any given finite linear
combination of $(p,r,s)$-atoms with $r\in(0,\fz)$ (or continuous $(p,\fz,s)$-atoms),
its quasi-norm in $H^{p,q}_A(\rn)$ can be achieved via all its finite atomic decompositions.

\begin{definition}\label{sevend1}
For an admissible anisotropic triplet $(p,r,s)$, $q\in(0,\fz]$ and a dilation $A$, denote by
$H_{A,{\rm fin}}^{p,r,s,q}(\rn)$ the set of  all distributions
$f\in\cs'(\rn)$ satisfying that there exist
$K,\,I\in\nn$, a finite sequence of $(p,r,s)$-atoms,
$\{a_i^k\}_{i\in[1,I]\cap\nn,\,k\in[1,K]\cap\mathbb{Z}}$,
supported on
$\{x_i^k+B_i^k\}_{i\in[1,I]\cap\nn,\,k\in[1,K]\cap\mathbb{Z}}
\subset\mathfrak{B}$, respectively,
and a positive constant $\widetilde{C}$, independent of $I$ and $K$, such that
$\sum_{i=1}^I\chi_{x_i^k+B_i^k}(x)\leq \widetilde{C}$
for all $x\in\rn$ and $k\in[1,K]\cap\mathbb{Z}$, and
$f=\sum_{k=1}^K\sum_{i=1}^I\lambda_i^ka_i^k$
in $\cs'(\rn)$, where $\lambda_i^k\sim2^k|B_i^k|^{1/p}$ for all
$k\in[1,K]\cap\mathbb{Z}$ and $i\in[1,I]\cap\mathbb{N}$ with the implicit equivalent
positive constants independent of $k,\,K$ and $i,\,I$.
Moreover, the
quasi-norm of $f$ in $H_{A,{\rm fin}}^{p,r,s,q}(\rn)$ is defined by
\begin{eqnarray*}
\|f\|_{H_{A,{\rm fin}}^{p,r,s,q}(\rn)}:=\inf\lf\{
\lf[\sum_{k=1}^{K}\lf(
\sum_{i=1}^I|\lz_i^k|^p\r)^{\frac qp}\r]^{\frac1q}:\ \
f=\sum_{k=1}^{K}\sum_{i=1}^I\lz_i^ka_i^k,\ \ \
 K,\,I\in\nn\r\}
\end{eqnarray*}
with the usual modification made when $q=\fz$, where the infimum is taken over all
decompositions of $f$ as above.
\end{definition}

Obviously, by Theorem \ref{tt1}, we know that,
for any admissible anisotropic triplet $(p,r,s)$ and $q\in(0,\fz)$,
the set $H_{A,{\rm fin}}^{p,r,s,q}(\rn)$ is dense in $H^{p,q}_A(\rn)$ with
respect to the quasi-norm $\|\cdot\|_{H^{p,q}_A(\rn)}$.
From this, we deduce the following density of $H^{p,q}_A(\rn)$.

\begin{lemma}\label{sevenl4}
If $p,\,q\in(0,\fz)$, then
\begin{enumerate}
\item[{\rm (i)}] for any $r\in[1,\fz]$, $H^{p,q}_A(\rn)\cap L^r(\rn)$
is dense in $H^{p,q}_A(\rn)$;
\item[{\rm (ii)}] $H^{p,q}_A(\rn)\cap C_c^\fz(\rn)$
is dense in $H^{p,q}_A(\rn)$.
\end{enumerate}
\end{lemma}

\begin{proof}
We first prove (i). If $p\in(1,\fz)$ and $q\in(0,\fz)$, then $H^{p,q}_A(\rn)=L^{p,q}(\rn)$
(see Remark \ref{sixr5}(ii) below).
By \cite[Theorem 1.4.13]{lg08}, we know that the set of simple functions is dense in
$L^{p,q}(\rn)$. Thus, $L^{p,q}(\rn)\cap L^r(\rn)$ is also dense in $L^{p,q}(\rn)$
for all $r\in[1,\fz]$.
If $p\in(0,1]$ and $q\in(0,\fz)$, by the density of the set $H_{A,{\rm fin}}^{p,\fz,s,q}(\rn)$
in $H^{p,q}_A(\rn)$ and $H_{A,{\rm fin}}^{p,\fz,s,q}(\rn)\subset L^r(\rn)$ for all $r\in[1,\fz]$,
we easily find that
$H^{p,q}_A(\rn)\cap L^r(\rn)$ is dense in $H^{p,q}_A(\rn)$. This finishes the proof of (i).

Now we prove (ii). To this end, we claim that, for any $\varphi\in\cs(\rn)$ with
$\int_{\rn}\varphi(x)\,dx\neq0$ and $f\in H^{p,q}_A(\rn)$,
\begin{eqnarray}\label{sevene21}
f\ast\varphi_k\rightarrow f \ \ \ \ \ {\rm in}\ H^{p,q}_A(\rn)\ \ {\rm as}\ k\to-\fz.
\end{eqnarray}
To show this, we first assume that $f\in H^{p,q}_A(\rn)\cap L^2(\rn)$.
For this case, to prove \eqref{sevene21}, it suffices to show that
\begin{eqnarray}\label{sevene22}
M_N\lf(f\ast\varphi_k-f\r)(x)\rightarrow0\ \ \ \ {\rm for\ a.\,e.}\
x\in\rn\ \ {\rm as}\ k\to-\fz,
\end{eqnarray}
where $N:=N_{(p)}+2$. Indeed, it is easy to see that
$f\ast\varphi_k-f\in L^2(\rn)$ for all $k\in\mathbb{Z}$,
which, together with
Proposition \ref{tl4} and Remark \ref{tr1}, implies that
$M_N(f\ast\varphi_k-f)\in L^2(\rn)$ for all $k\in\mathbb{Z}$. By this,
\cite[p.\,39, Lemma 6.6]{mb03}, \eqref{sevene22}, \eqref{se6}
and the Lebesgue dominated convergence theorem, we know that
\eqref{sevene21} holds true for all $f\in H^{p,q}_A(\rn)\cap L^2(\rn)$.

Now, we show \eqref{sevene22}.
Notice that, if $g$ is continuous and has compact support,
then $g$ is uniformly continuous on $\rn$. Thus, for any $\delta\in(0,\fz)$,
there exists $\eta\in(0,\fz)$ such that, for all $y\in\rn$ satisfying $\rho(y)<\eta$
and $x\in\rn$, $|g(x-y)-g(x)|<\frac\delta{2\|\varphi\|_{L^1(\rn)}}$.
Without loss of generality,
we may assume that $\int_{\rn}\varphi(x)\,dx=1$. Then
$\int_{\rn}\varphi_k(x)\,dx=1$ for all $k\in\zz$. From this, we deduce that,
for all $k\in\zz$ and $x\in\rn$,
\begin{eqnarray}\label{sevene24}
&&|g\ast\varphi_k(x)-g(x)|\\
&&\hs\le\int_{\rho(y)<\eta}|g(x-y)-g(x)||\varphi_k(y)|\,dy
+\int_{\rho(y)\ge\eta}|g(x-y)-g(x)||\varphi_k(y)|\,dy\noz\\
&&\hs<\frac\delta2+2\|g\|_{L^\fz(\rn)}
\int_{\rho(y)\ge b^{-k}\eta}|\varphi(y)|\,dy.\noz
\end{eqnarray}
On the other hand, by the integrability of $\varphi$, we know that
there exists $\widetilde{k}\in\zz$ such that,
for all $k\in[\widetilde{k},\fz)\cap\zz$,
$2\|g\|_{L^\fz(\rn)}
\int_{\rho(y)\ge b^{-k}\eta}|\varphi(y)|\,dy<\frac\delta2$,
which, combined with \eqref{sevene24}, implies that
$\lim_{k\to-\fz}|g\ast\varphi_k(x)-g(x)|=0$ holds true uniformly for all $x\in\rn$.
Therefore, $\|g\ast\varphi_k-g\|_{L^\fz(\rn)}\to0$ as $k\to-\fz$,
which, together with Proposition \ref{tl4} and Remark \ref{tr1},
further implies that
\begin{eqnarray}\label{sevene23}
\lf\|M_N\lf(g\ast\varphi_k-g\r)\r\|_{L^\fz(\rn)}\ls
\|g\ast\varphi_k-g\|_{L^\fz(\rn)}\to0\ \ \ {\rm as}\ k\to-\fz.
\end{eqnarray}
For any given $\epsilon\in(0,\fz)$, there exists a continuous function $g$
with compact support such that $\|f-g\|^2_{L^2(\rn)}<\epsilon$.
By \eqref{sevene23} and \cite[p.\,39, Lemma 6.6]{mb03}, there exists
a positive constant $C_6$ such that, for all $x\in\rn$,
\begin{eqnarray*}
&&\limsup_{k\to-\fz}M_N\lf(f\ast\varphi_k-f\r)(x)\\
&&\hs\le\sup_{k\in\zz}M_N\lf((f-g)\ast\varphi_k\r)(x)
+\limsup_{k\to-\fz}M_N\lf(g\ast\varphi_k-g\r)(x)
+M_N(g-f)(x)\\
&&\hs\le C_6M_{N_{(p)}}(g-f)(x).
\end{eqnarray*}
Therefore, by Proposition \ref{tl4} and Remark \ref{tr1} again,
we find that there exists a
positive constant $C_7$ such that, for any $\lz\in(0,\fz)$,
\begin{eqnarray*}
&&\lf|\lf\{x\in\rn:\ \limsup_{k\to-\fz}
M_N\lf(f\ast\varphi_k-f\r)(x)>\lz\r\}\r|\\
&&\hs\le\lf|\lf\{x\in\rn:\
M_{N_{(p)}}(g-f)(x)>\frac\lz{C_6}\r\}\r|
\le C_7\frac{\|f-g\|^2_{L^2(\rn)}}{\lz^2}\le C_7\frac{\epsilon}{\lz^2},
\end{eqnarray*}
which implies that \eqref{sevene22} holds true for all
$f\in H^{p,q}_A(\rn)\cap L^2(\rn)$.

Assume now $f\in H^{p,q}_A(\rn)$. By (i), $H^{p,q}_A(\rn)\cap L^2(\rn)$
is dense in $H^{p,q}_A(\rn)$.
Thus, for any given $\epsilon\in(0,\fz)$, there exists a function
$g\in H^{p,q}_A(\rn)\cap L^2(\rn)$ such that
$$\|f-g\|_{H^{p,q}_A(\rn)}^q<\epsilon.$$
Moreover, by \cite[p.\,39, Lemma 6.6]{mb03} again and $f\in H^{p,q}_A(\rn)$,
we know that $\{f\ast\varphi_k\}_{k\in\zz}$ are uniformly
bounded in $H^{p,q}_A(\rn)$ and
$$\sup_{k\in\zz}\lf\|(f-g)\ast\varphi_k\r\|_{H^{p,q}_A(\rn)}
\ls\|f-g\|_{H^{p,q}_A(\rn)}.$$
Therefore,
by \eqref{sevene21} being true for all $f\in H^{p,q}_A(\rn)\cap L^2(\rn)$,
we further conclude that
\begin{eqnarray*}
&&\limsup_{k\to-\fz}\|f\ast\varphi_k-f\|_{H^{p,q}_A(\rn)}^q\\
&&\hs\le\sup_{k\in\zz}\|(f-g)\ast\varphi_k\|_{H^{p,q}_A(\rn)}^q
+\limsup_{k\to-\fz}\|g\ast\varphi_k-g\|_{H^{p,q}_A(\rn)}^q
+\|g-f\|_{H^{p,q}_A(\rn)}^q\\
&&\hs\ls\|g-f\|_{H^{p,q}_A(\rn)}^q\ls\epsilon.
\end{eqnarray*}
This implies that the claim \eqref{sevene21} holds true.

Notice that, if $f\in H_{A,{\rm fin}}^{p,r,s,q}(\rn)$
and $\varphi\in C_c^\fz(\rn)$
with $\int_{\rn}\varphi(x)\,dx\neq0$, then
$f\ast\varphi_k\in C_c^\fz(\rn)\cap H^{p,q}_A(\rn)$
for all $k\in\zz$ and, by \eqref{sevene21},
\begin{eqnarray*}
f\ast\varphi_k\rightarrow f \ \ \ \ \ {\rm in}\ H^{p,q}_A(\rn)\ \ {\rm as}\ k\to-\fz.
\end{eqnarray*}
From this and the density of the set $H_{A,{\rm fin}}^{p,r,s,q}(\rn)$ in
$H^{p,q}_A(\rn)$, we further deduce that
$C_c^\fz(\rn)\cap H^{p,q}_A(\rn)$ is dense in $H^{p,q}_A(\rn)$. This finishes the proof
of (ii) and hence Lemma \ref{sevenl4}.
\end{proof}

The following conclusion is from Theorem \ref{tt1} and its proof.
We state it here for the later application.
\begin{lemma}\label{sevenl1}
If $p\in(0,1],\,q\in(0,\fz],\,r\in(1,\fz]$ and $s\in\mathbb{N}$ with
$s\geq\lfloor(1/p-1)\ln b/\ln\lambda_-\rfloor$, then,
for any $f\in H^{p,q}_A(\rn)\cap L^r(\rn)$, there exist
$\{\lz_i^k\}_{k\in\zz,\,i\in\nn}\subset\mathbb{C}$,
$\{x_i^k\}_{k\in\zz,\,i\in\nn}\subset\rn$, balls
$\{B_{\ell_i^k}\}_{k\in\zz,\,i\in\nn}$ and
$(p,\fz,s)$-atoms $\{a_i^k\}_{k\in\zz,\,i\in\nn}$ such that
$$f=\sum_{k\in\zz}\sum_{i\in\nn}\lz_i^ka_i^k,$$
where the series converges almost everywhere and
also converges in $\cs'(\rn)$,
\begin{eqnarray}\label{sevene1}
\supp a_i^k\subset B_{\ell_i^k+4\tau},\ \
\Omega_k=\bigcup_{i\in\mathbb{N}}(x_i^k+B_{\ell_i^k+4\tau})\ \
for\ all\ k\in\zz\ and\ i\in\nn,\ here
\end{eqnarray}
$$\Omega_k:=\lf\{x\in\rn:\ M_N(f)(x)>2^k\r\};$$
\begin{eqnarray}\label{sevene2}
(x_i^k+B_{\ell_i^k-\tau})\cap(x_j^k+B_{\ell_j^k-\tau})
=\emptyset\ \ \ for\ all\ k\in\zz\ and\ i,\,j\in\nn\ with\ i\neq j;
\end{eqnarray}
\begin{eqnarray}\label{sevene3}
\sharp\lf\{j\in\mathbb{N}:\
(x_i^k+B_{\ell_i^k+4\tau})\cap(x_j^k+B_{\ell_j^k+
4\tau})\neq\emptyset\r\}\leq L\ \ for\ all\ \ i\in\mathbb{N},
\end{eqnarray}
here $L$ is a
positive constant independent of $\Omega_k$ and $f$.
Moreover, there exists a positive constant $C$, independent of $f$, such that
\begin{eqnarray}\label{sevene4}
|\lz_i^ka_i^k|\le C2^k\ \ \ \ \ \ \ \ \ \
for\ all\ \ k\in\zz\ and\ i\in\mathbb{N},
\end{eqnarray}
\begin{eqnarray}\label{sevene5}
\sum_{k\in\zz}\lf(\sum_{i\in\nn}|\lz_i^k|^p\r)^{\frac qp}
\le C\|f\|_{H^{p,q}_A(\rn)}^q.
\end{eqnarray}
\end{lemma}

\begin{remark}\label{sevenr1}
For all $i\in\nn,\,k\in\zz$ and $\ell\in[0,\fz)$,
let $\zeta_i^k$ and $\mathcal{P}_{\ell}(\rn)$ be the same as in the proof of
Theorem \ref{tt1}. For any $f\in H^{p,q}_A(\rn)\cap L^r(\rn)$,
by an argument similar to that used in the proof of Theorem \ref{tt1},
we also find
that there exists a unique polynomial $P_i^k\in\mathcal{P}_{\ell}(\rn)$
such that, for all $Q\in\mathcal{P}_{\ell}(\rn)$,
\begin{eqnarray}\label{sevene12}
\lf\langle f,Q\zeta_i^k\r\rangle=
\lf\langle P_i^k,Q\zeta_i^k\r\rangle=
\int_\rn P_i^k(x)Q(x)\zeta_i^k(x)\,dx.
\end{eqnarray}
Moreover,
for $i,\,j\in\mathbb{N}$ and $k\in\mathbb{Z}$,
we define a polynomial $P_{i,j}^{k+1}$ as the orthogonal
projection of
$(f-P_j^{k+1})\zeta_i^k$ on $\mathcal{P}_{\ell}(\rn)$ with
respect to the norm defined by \eqref{te14}, namely,
$P_{i,j}^{k+1}$ is the unique element of $\mathcal{P}_{\ell}(\rn)$
such that, for all $Q\in\mathcal{P}_{\ell}(\rn)$,
\begin{eqnarray}\label{sevene14}
\int_\rn \lf[f(x)-P_j^{k+1}(x)\r]\zeta_i^k(x)Q(x)
\zeta_j^{k+1}(x)\,dx=\int_\rn P_{i,j}^{k+1}(x)Q(x)
\zeta_j^{k+1}(x)\,dx
\end{eqnarray}
 and, for all $i\in\mathbb{N}$ and $k\in\mathbb{Z}$,
\begin{eqnarray}\label{sevene17}
\lz_i^ka_i^k=(f-P_i^k)\zeta_i^k-
\sum_{j\in\mathbb{N}}\lf[(f-P_j^{k+1})
\zeta_i^k-P_{i,j}^{k+1}\r]\zeta_j^{k+1}.
\end{eqnarray}
\end{remark}

The following Lemmas \ref{sevenl2} and \ref{sevenl3} are just
\cite[Lemmas 4.4 and 5.2]{blyz08}, respectively;
see also \cite[p.\,25, Lemma 5.3 and p.\,36, Lemma 6.2]{mb03}.

\begin{lemma}\label{sevenl2}
There exists a positive constant $C$, independent of $f$, such that,
for all $i\in\nn$ and $k\in\zz$,
$$\sup_{y\in\rn}\lf|P_i^k(y)\zeta_i^k(y)\r|\le C\sup_{y\in U_i^k}
M_N(f)(y)\le C2^k,$$
where $U_i^k:=(x_i^k+B_{\ell_i^k+4\tau+1})\cap(\Omega_k)^\com$.
\end{lemma}

\begin{lemma}\label{sevenl3}
There exists a positive constant $C$, independent of $f$, such that,
for all $i,\,j\in\nn$ and $k\in\zz$,
$$\sup_{y\in\rn}\lf|P_{i,j}^{k+1}(y)\zeta_j^{k+1}(y)\r|\le C\sup_
{y\in\widetilde{U}_i^k}M_N(f)(y)\le C2^{k+1},$$
where
$\widetilde{U}_i^k:=
(x_j^{k+1}+B_{\ell_j^{k+1}+4\tau+1})\cap(\Omega_{k+1})^\com$.
\end{lemma}

The following Theorem \ref{sevent1}
extends \cite[Theorem 3.1 and Remark 3.3]{msv08} to the setting
of anisotropic Hardy-Lorentz spaces.

\begin{theorem}\label{sevent1}
Let $q\in(0,\fz]$ and $(p,r,s)$ be an admissible anisotropic triplet.
\begin{enumerate}
\item[{\rm (i)}]
If $r\in(1,\fz)$, then $\|\cdot\|_{H_{A,{\rm fin}}^{p,r,s,q}(\rn)}$
and $\|\cdot\|_{H^{p,q}_A(\rn)}$ are equivalent quasi-norms on
$H_{A,{\rm fin}}^{p,r,s,q}(\rn)$;
\item[{\rm (ii)}]
$\|\cdot\|_{H_{A,{\rm fin}}^{p,\fz,s,q}(\rn)}$
and $\|\cdot\|_{H^{p,q}_A(\rn)}$ are equivalent quasi-norms on
$H_{A,{\rm fin}}^{p,\fz,s,q}(\rn)\cap C(\rn)$.
\end{enumerate}
\end{theorem}

\begin{proof}
Obviously, by Theorem \ref{tt1},
$H_{A,{\rm fin}}^{p,r,s,q}(\rn)\subset H^{p,q}_A(\rn)$ and,
for all $f\in H_{A,{\rm fin}}^{p,r,s,q}(\rn)$,
$$\|f\|_{H^{p,q}_A(\rn)}\ls\|f\|_{H_{A,{\rm fin}}^{p,r,s,q}(\rn)}.$$
Thus, we only need to prove that,
for all $f\in H_{A,{\rm fin}}^{p,r,s,q}(\rn)$ when $r\in(1,\fz)$
and for all $f\in H_{A,{\rm fin}}^{p,r,s,q}(\rn)\cap C(\rn)$ when $r=\fz$,
$\|f\|_{H_{A,{\rm fin}}^{p,r,s,q}(\rn)}\ls\|f\|_{H^{p,q}_A(\rn)}$.
We prove this by five steps.

\emph{Step 1.} Let $r\in(1,\fz]$. Without loss of generality,
we may assume that $f\in H_{A,{\rm fin}}^{p,r,s,q}(\rn)$ and
$\|f\|_{H^{p,q}_A(\rn)}=1$. Notice that $f$ has compact support.
Then there exists some $k_0\in\zz$ such that $\supp f\subset B_{k_0}$,
where $B_{k_0}$ is as in Section \ref{s2}. For any $k\in\zz$, let
$$\Omega_k:=\lf\{x\in\rn:\ M_N(f)(x)>2^k\r\},$$
here and hereafter in this section, we let $N\equiv N_{(p)}$. Since
$f\in H^{p,q}_A(\rn)\cap L^{\widetilde{r}}(\rn)$, where
$\widetilde{r}:=r$ if $r\in(1,\fz)$ and $\widetilde{r}:=2$ if $r=\fz$,
by Lemma \ref{sevenl1}, we know that there exist
$\{\lz_i^k\}_{k\in\zz,\,i\in\nn}\subset\mathbb{C}$ and a sequence of
$(p,\fz,s)$-atoms, $\{a_i^k\}_{k\in\zz,\,i\in\nn}$, such that
$f=\sum_{k\in\zz}\sum_{i\in\nn}\lz_i^ka_i^k$ holds true almost everywhere
and also in $\cs'(\rn)$ and, moreover, \eqref{sevene1} through \eqref{sevene5}
of Lemma \ref{sevenl1} also hold true.

\emph{Step 2.} In this step, we prove that there exists a positive constant
$\widetilde{C}$ such that, for all
$x\in (B_{k_0+4\tau})^\com$,
\begin{eqnarray}\label{sevene10}
M_N(f)(x)\le\widetilde{C}|B_{k_0}|^{-1/p}.
\end{eqnarray}

To this end, for any fixed $x\in (B_{k_0+4\tau})^\com$,
by Proposition \ref{sp1}, we have
\begin{eqnarray*}
M_N(f)(x)&&\ls M_N^0(f)(x)\\
&&\ls\sup_{\phi\in\cs_N(\rn)}\sup_{k\in[k_0,\fz)\cap\zz}
\lf|f\ast\phi_k(x)\r|+\sup_{\phi\in\cs_N(\rn)}\sup_{k\in(-\fz,k_0)\cap\zz}
\lf|f\ast\phi_k(x)\r|=:{\rm I}_1+{\rm I}_2.
\end{eqnarray*}

For ${\rm I}_1$, assume that $\theta\in\cs(\rn)$ satisfying that
$\supp \theta\subset B_\tau,\,0\le\theta\le1$ and $\theta\equiv1$ on $B_0$.
For $k\in[k_0,\fz)\cap\zz$, from $\supp f\subset B_{k_0}$, we deduce that
\begin{eqnarray}\label{sevene6}
f\ast\phi_k(x)=\int_\rn\phi_k(x-y)\theta(A^{-k_0}y)f(y)\,dy
=:f\ast\varphi_{k_0}(0),
\end{eqnarray}
where $\varphi(y):=b^{k_0-k}\phi(A^{-k}x+A^{k_0-k}y)\theta(-y)$
for all $y\in\rn$. Noticing that, for any $\alpha\in\zz_+^n$ with $|\alpha|\le N$,
by \eqref{se19}, $\lz_-\in(1,\fz),\,k\in[k_0,\fz)\cap\zz$ and $\|\phi\|_{\cs_N(\rn)}\le1$,
we find that
$$\lf|\partial^\alpha\lf[\phi(A^{k_0-k}\cdot)\r](y)\r|
\ls(\lz_-)^{(k_0-k)|\alpha|}\|\phi\|_{\cs_N(\rn)}\ls1,$$
which, combined with the product rule and $\supp \theta\subset B_\tau$,
further implies that
\begin{eqnarray}\label{sevene7}
\|\varphi\|_{\cs_N(\rn)}=\sup_{|\alpha|\le N}\sup_{y\in B_\tau}
\lf|\partial^\alpha_y\lf[\phi\lf(A^{-k}x+A^{k_0-k}y\r)\theta(-y)\r]\r|
\lf[1+\rho(y)\r]^N\ls1.
\end{eqnarray}
Thus, noticing that $[\|\varphi\|_{\cs_N(\rn)}]^{-1}\varphi\in\cs_N(\rn)$
and, for all $z\in B_{k_0}$, $0\in z+B_{k_0}$, by the definition of $M_N$, we
know that, for all $z\in B_{k_0}$,
\begin{eqnarray}\label{sevene8}
M_N(f)(z)\ge\sup_{u\in z+B_{k_0}}
\lf|\lf(\frac\varphi{\|\varphi\|_{\cs_N(\rn)}}\r)_{k_0}\ast f(u)\r|
\ge\frac1{\|\varphi\|_{\cs_N(\rn)}}\lf|\varphi_{k_0}\ast f(0)\r|.
\end{eqnarray}
Combining \eqref{sevene6}, \eqref{sevene7} and \eqref{sevene8},
we further conclude that, for all $x\in (B_{k_0+4\tau})^\com$,
$$\lf|f\ast\phi_k(x)\r|\le\|\varphi\|_{\cs_N(\rn)}\inf_{z\in B_{k_0}}
M_N(f)(z)\ls\inf_{z\in B_{k_0}}M_N(f)(z)$$
and hence,
${\rm I}_1\ls\inf_{z\in B_{k_0}}M_N(f)(z)=:C_8\inf_{z\in B_{k_0}}M_N(f)(z)$. Let
$$\widetilde{{\rm I}}:=\frac{{\rm I}_1+C_8\inf_{z\in B_{k_0}}M_N(f)(z)}2.$$
Then, it is easy to see that
${\rm I}_1<\widetilde{{\rm I}}<C_8\inf_{z\in B_{k_0}}M_N(f)(z)$. Therefore,
\begin{eqnarray*}
\|f\|_{H^{p,q}_A(\rn)}&&\ge\|f\|_{H^{p,\fz}_A(\rn)}\\
&&\gs\sup_{\lz\in(0,\fz)}\lz
\lf|\lf\{z\in B_{k_0}:\ C_8M_N(f)(z)>\lz\r\}\r|^{\frac1p}\\
&&\gs{\rm I}_1
\lf|\lf\{z\in B_{k_0}:\ \widetilde{{\rm I}}>{\rm I}_1\r\}\r|^{\frac1p}
\sim{\rm I}_1|B_{k_0}|^{\frac1p},
\end{eqnarray*}
which, together with $\|f\|_{H^{p,q}_A(\rn)}=1$, further implies that
${\rm I}_1\ls|B_{k_0}|^{-\frac1p}$.

For ${\rm I}_2$, by $\supp f\subset B_{k_0}$ and $\theta\equiv1$ on
$B_0$, we find that, for all $k\in(-\fz,k_0)\cap\zz$,
$x\in (B_{k_0+4\tau})^\com$ and $z\in B_{k_0}$,
$$f\ast\phi_k(x)=\int_\rn\phi_k(x-y)\theta(A^{-k_0}y)f(y)\,dy
=:f\ast\psi_k(z),$$
where $\psi(u):=\phi(A^{-k}(x-z)+u)\theta(A^{-k_0}z-A^{k-k_0}u)$
for all $u\in\rn$. Notice that, if $u\in\supp \psi$,
then $A^{-k_0}z-A^{k-k_0}u\in B_\tau$ and hence $u\in B_{k_0-k+2\tau}$.
Therefore, by \eqref{se1} and \eqref{se2}, we have
\begin{eqnarray*}
A^{-k}(x-z)+u&&\in\lf(B_{k_0-k+4\tau}\r)^\com+B_{k_0-k}+B_{k_0-k+2\tau}\\
&&\subset\lf(B_{k_0-k+4\tau}\r)^\com+B_{k_0-k+3\tau}
\subset\lf(B_{k_0-k+3\tau}\r)^\com,
\end{eqnarray*}
which implies that $\rho(A^{-k}(x-z)+u)\ge b^{k_0-k+3\tau}$.
From this, \eqref{se19}, $\lz_-\in(1,\fz),\,k\in(-\fz,k_0)\cap\zz$ and $\phi\in\cs_N(\rn)$,
we further deduce that
$$\|\psi\|_{\cs_N(\rn)}\ls\sup_{|\alpha|\le N}\sup_{u\in\supp\psi}
(\lz_-)^{(k-k_0)|\alpha|}
\lf[\frac{1+\rho(u)}{1+\rho\lf(A^{-k}(x-z)+u\r)}\r]^N\ls1.$$
Thus, by an argument similar to that used for ${\rm I}_1$, we have
${\rm I}_2\ls\inf_{z\in B_{k_0}}M_N(f)(z)\ls|B_{k_0}|^{-\frac1p}$.
Combining the above estimates of ${\rm I}_1$ and ${\rm I}_2$, we
show that \eqref{sevene10} holds true.

\emph{Step 3.} We now denote by $\widetilde{k}$ the largest integer $k$
such that $2^k<\widetilde{C}|B_{k_0}|^{-\frac1p}$, where $\widetilde{C}$
is the same as in \eqref{sevene10}. Then, by \eqref{sevene10}, we have
\begin{eqnarray}\label{sevene11}
\Omega_k\subset B_{k_0+4\tau}\ \ \ \ {\rm for\ all}\ k\in(\widetilde{k},\fz]\cap\zz.
\end{eqnarray}
Let $h:=\sum_{k=-\fz}^{\widetilde{k}}\sum_{i\in\nn}\lz_i^ka_i^k$
and $\ell:=\sum_{k=\widetilde{k}+1}^\fz\sum_{i\in\nn}\lz_i^ka_i^k$,
where the series converges almost everywhere and also in $\cs'(\rn)$. Clearly,
$f=h+\ell$. In what follows of this step, we show that $h$ is a constant multiple
of a $(p,\fz,s)$-atom with the constant independent of $f$. To this end,
observe that
$\supp \ell\subset\bigcup_{k=\widetilde{k}+1}^\fz\Omega_k
\subset B_{k_0+4\tau}$,
which, combined with $\supp f\subset B_{k_0+4\tau}$,
further implies that
$\supp h\subset B_{k_0+4\tau}$.

Notice that, for any $r\in(1,\fz]$ and $r_1\in(1,r)$, by the H\"{o}lder inequality,
we have
$$\int_{\rn}|f(x)|^{r_1}\,dx
\le|B_{k_0}|^{1-\frac{r_1}r}\|f\|_{L^r(\rn)}^{r_1}<\fz.$$
Observing that $\supp f\subset B_{k_0}$ and $f$
has vanishing moments up to order $s$, we know that $f$ is a constant multiple of a
$(1,r_1,0)$-atom and therefore, by Lemma \ref{tl2}, $M_N(f)\in L^1(\rn)$.
Then, by \eqref{sevene3}, \eqref{sevene1}, \eqref{sevene11} and
\eqref{sevene4}, we have
$$\int_{\rn}\sum_{k=\widetilde{k}+1}^\fz\sum_{i\in\nn}
|\lz_i^ka_i^k(x)x^\alpha|\,dx
\ls\sum_{k\in\zz}2^k|\Omega_k|\ls\|M_N(f)\|_{L^1(\rn)}<\fz.$$
This, together with the vanishing moments of $a_i^k$, implies that $\ell$
has vanishing moments up to $s$ and hence, so does $h$ by $h=f-\ell$.
Moreover, by
\eqref{sevene3}, \eqref{sevene4} and the fact that
$2^{\widetilde{k}}<\widetilde{C}|B_{k_0}|^{-\frac1p}$, we find that,
for all $x\in\rn$,
$$|h(x)|\ls\sum_{k=-\fz}^{\widetilde{k}}2^k\ls|B_{k_0}|^{-\frac1p}.$$
Thus, there exists a positive constant $C_9$, independent of $f$, such that
$h/C_9$ is a $(p,\fz,s)$-atom and, by Definition \ref{d-at}, it is also a
$(p,r,s)$-atom for any admissible anisotropic triplet $(p,r,s)$.

\emph{Step 4.} In this step, we show (i). To this end,
assume that $r\in(1,\fz)$. We first show that
$$\sum_{k=\widetilde{k}+1}^\fz\sum_{i\in\nn}\lz_i^ka_i^k\in L^r(\rn).$$
For all $x\in\rn$, since $\rn=\bigcup_{j\in\zz}(\Omega_j\setminus\Omega_{j+1})$,
it follows that there exists a $j_0\in\zz$ such that
$x\in(\Omega_{j_0}\setminus\Omega_{j_0+1})$. Notice that
$\supp a_i^k\subset B_{\ell_i^k+\tau}\subset \Omega_k\subset\Omega_{j_0+1}$
for all $k\in(j_0,\fz)\cap\zz$, using \eqref{sevene3} and \eqref{sevene4}, we conclude that,
for all $x\in(\Omega_{j_0}\setminus\Omega_{j_0+1})$,
$$\sum_{k=\widetilde{k}+1}^\fz\sum_{i\in\nn}|\lz_i^ka_i^k(x)|
\ls\sum_{k\le j_0}2^k\ls2^{j_0}\ls M_N(f)(x).$$
Since $f\in L^r(\rn)$, by Proposition \ref{tl4} and Remark \ref{tr1}, we have
$M_N(f)\in L^r(\rn)$. Thus, by the Lebesgue dominated convergence theorem,
we further have $\sum_{k=\widetilde{k}+1}^K\sum_{i\in\nn}\lz_i^ka_i^k$
converges to $\ell$ in $L^r(\rn)$ as $K\ge\widetilde{k}+1$ and $K\to\fz$.

Now, for any positive integer $K>\widetilde{k}$ and
$k\in[\widetilde{k}+1,K]\cap\zz$, let
\begin{eqnarray*}
I_{(K,k)}:=\lf\{i\in\nn:\ |i|+|k|\le K\r\}
\ \ \ {\rm and}\ \ \ \ell_{(K)}:=\sum_{k=\widetilde{k}+1}^K
\sum_{i\in I_{(K,k)}}\lz_i^ka_i^k.
\end{eqnarray*}
Since $\ell\in L^r(\rn)$, it follows that,
for any given $\epsilon\in(0,1)$, there exists $K\in[\widetilde{k}+1,\fz)\cap\zz$
large enough,
depending on $\epsilon$, such that
$(\ell-\ell_{(K)})/\epsilon$ is a $(p,r,s)$-atom.
Therefore, $f=h+\ell_{(K)}+(\ell-\ell_{(K)})$
is a finite linear combination of $(p,r,s)$-atoms.
By Step 3 and \eqref{sevene5}, we further
find that
$$\|f\|_{H_{A,{\rm fin}}^{p,r,s,q}(\rn)}^q
\ls (C_0)^q+\sum_{k=\widetilde{k}+1}^K
\lf(\sum_{i\in I_{(K,k)}}|\lz_i^k|^p\r)^{\frac qp}
+\epsilon^q\ls1,$$
which completes the proof of (i).

\emph{Step 5.} In this step, we show (ii). To this end,
assume that $f$ is a continuous function in
$H_{A,{\rm fin}}^{p,\fz,s,q}(\rn)$. Then $a_i^k$ is also continuous due
to its construction (see also \eqref{te17}). Since
$$M_N(f)(x)\le C_{(n,N)}\|f\|_{L^\fz(\rn)}\ \ \ \ \ {\rm for\ all}\ x\in\rn,$$
where $C_{(n,N)}$ is a positive constant only depending on $n$ and $N$, it follows
that the level set $\Omega_k$ is empty for all $k$ satisfying that
\begin{eqnarray}\label{sevene25}
2^k\ge C_{(n,N)}\|f\|_{L^\fz(\rn)}.
\end{eqnarray}
 Let $\widehat{k}$ be the largest integer
for which \eqref{sevene25} does not hold true. Then the index $k$ in the sum defining
$\ell$ runs only over $k\in\{\widetilde{k}+1,\ldots,\widehat{k}\}$.

Let $\epsilon\in(0,\fz)$. Since $f$ is uniformly continuous, it follows that there exists
a $\delta\in(0,\fz)$ such that $|f(x)-f(y)|<\epsilon$ whenever $\rho(x-y)<\delta$.
Write $\ell=\ell_1^\epsilon+\ell_2^\epsilon$ with
$$\ell_1^\epsilon:=\sum_{k=\widetilde{k}+1}^{\widehat{k}}
\sum_{i\in F_1^{(k,\delta)}}\lz_i^ka_i^k\ \ \ \ {\rm and}\ \ \ \
\ell_2^\epsilon:=\sum_{k=\widetilde{k}+1}^{\widehat{k}}
\sum_{i\in F_2^{(k,\delta)}}\lz_i^ka_i^k,$$
where, for $k\in\{\widetilde{k}+1,\ldots,\widehat{k}\}$,
$F_1^{(k,\delta)}:=\{i\in\nn:\ b^{\ell_i^k+\tau}\ge\delta\}$
and
$F_2^{(k,\delta)}:=\{i\in\nn:\ b^{\ell_i^k+\tau}<\delta\}$.
Notice that, for any fixed $k\in\{\widetilde{k}+1,\ldots,\widehat{k}\}$,
by \eqref{sevene2} and \eqref{sevene11}, we know that
$F_1^{(k,\delta)}$ is a finite set and hence $\ell_1^\epsilon$
is continuous.

On the other hand, for any $k\in\{\widetilde{k}+1,\ldots,\widehat{k}\}$,
$i\in\nn$ such that $b^{\ell_i^k+\tau}<\delta$ and
$x\in x_i^k+B_{\ell_i^k+\tau}$,
$|f(x)-f(x_i^k)|<\epsilon$. By \eqref{sevene12} and
$\supp\zeta_i^k\subset x_i^k+B_{\ell_i^k+\tau}$,
we find that, for all $Q\in\mathcal{P}_{\ell}(\rn)$,
$$\frac 1{\int_\rn\zeta_i^k(x)\,dx}\int_\rn
\lf[\widetilde{f}(x)-\widetilde{P}_i^k(x)\r]Q(x)\zeta_i^k(x)\,dx=0,$$
where, for all $x\in\rn$,
$$\widetilde{f}(x):=\lf[f(x)-f(x_i^k)\r]
\chi_{\{x_i^k+B_{\ell_i^k+\tau}\}}(x)
\ \ {\rm and}\ \ \widetilde{P}_i^k(x):=
P_i^k(x)-f(x_i^k).$$
Since $|\widetilde{f}(x)|<\epsilon$ for all $x\in\rn$ implies
$M_N(\widetilde{f})(x)\ls\epsilon$ for all $x\in\rn$,
by Lemma \ref{sevenl2}, we have
\begin{eqnarray}\label{sevene13}
\sup_{y\in\rn}\lf|\widetilde{P}_i^k(y)\zeta_i^k(y)\r|
\ls\sup_{y\in\rn}M_N(\widetilde{f})(y)\ls\epsilon.
\end{eqnarray}
Similar to Remark \ref{sevenr1},
for all $k\in\{\widetilde{k}+1,\ldots,\widehat{k}\}$,
$i\in F_2^{(k,\delta)}$ and $j\in\nn$, let
$\widetilde{P}_{i,j}^{k+1}$
be the orthogonal projection of
$(\widetilde{f}-\widetilde{P}_j^{k+1})\zeta_i^k$ on
$\mathcal{P}_{\ell}(\rn)$ with
respect to the norm defined by \eqref{te14},
then, for all $Q\in\mathcal{P}_{\ell}(\rn)$,
\begin{eqnarray}\label{sevene15}
\int_\rn \lf[\widetilde{f}(x)-\widetilde{P}_j^{k+1}(x)\r]\zeta_i^k(x)Q(x)
\zeta_j^{k+1}(x)\,dx=\int_\rn \widetilde{P}_{i,j}^{k+1}(x)Q(x)
\zeta_j^{k+1}(x)\,dx.
\end{eqnarray}
By $\supp\zeta_i^k\subset x_i^k+B_{\ell_i^k+\tau}$, we have
$[\widetilde{f}-\widetilde{P}_j^{k+1}]\zeta_i^k
=[f-P_j^{k+1}]\zeta_i^k$.
From this, \eqref{sevene14} and \eqref{sevene15}, we further deduce that
$\widetilde{P}_{i,j}^{k+1}=P_{i,j}^{k+1}$. Then, by Lemma \ref{sevenl3},
we find that
\begin{eqnarray}\label{sevene16}
\sup_{y\in\rn}\lf|\widetilde{P}_{i,j}^{k+1}(y)\zeta_j^{k+1}(y)\r|
\ls\sup_{y\in\rn}M_N(\widetilde{f})(y)\ls\epsilon.
\end{eqnarray}
Furthermore, by \eqref{sevene17} and
$\sum_{j\in\nn}\zeta_j^{k+1}=\chi_{\Omega_{k+1}}$, we have
\begin{eqnarray*}
\lz_i^ka_i^k&&=(f-P_i^k)\zeta_i^k-
\sum_{j\in\mathbb{N}}\lf[(f-P_j^{k+1})
\zeta_i^k-P_{i,j}^{k+1}\r]\zeta_j^{k+1}\\
&&=\zeta_i^k\widetilde{f}\chi_{\Omega_{k+1}^
\com}-\widetilde{P}_i^k\zeta_i^k+\zeta_i^k\sum
_{j\in\nn}\widetilde{P}_j^{k+1}\zeta_j^{k+1}+\sum_
{j\in\mathbb{N}}\widetilde{P}_{i,j}^{k+1}\zeta_j^{k+1},
\end{eqnarray*}
which, combined with \eqref{sevene13}, \eqref{sevene16}
and \cite[p.\,35, Lemma 6.1(ii)]{mb03}, further implies that
$|\lz_i^ka_i^k(x)|\ls\epsilon$ for all
$k\in\{\widetilde{k}+1,\ldots,\widehat{k}\}$,
$i\in F_2^{(k,\delta)}$ and $x\in x_i^k+B_{\ell_i^k+\tau}$.

Moreover, using \eqref{sevene1} and \eqref{sevene3},
we conclude that there exists a positive constant $C_{10}$,
independent of $f$, such that
\begin{eqnarray}\label{sevene18}
\lf|\ell_2^\epsilon\r|\le C_{10}\sum_{k=\widetilde{k}+1}^{\widehat{k}}
\epsilon=C_{10}\lf(\widehat{k}-\widetilde{k}\r)\epsilon.
\end{eqnarray}
Since $\epsilon$ is arbitrary, we hence split $\ell$ into a continuous
part and a part which is uniformly arbitrarily small. This fact implies that
$\ell$ is continuous. Therefore, $h=f-\ell$ is a $C_9$ multiple of a continuous
$(p,\fz,s)$-atom by Step 3.

Now we give a finite decomposition of $f$. To this end, we use again the
splitting $\ell:=\ell_1^\epsilon+\ell_2^\epsilon$. Obviously, for any
$\epsilon\in(0,\fz)$, $\ell_1^\epsilon$ is a finite linear combination of
continuous $(p,\fz,s)$-atoms and, by \eqref{sevene5}, we have
\begin{eqnarray}\label{sevene19}
\sum_{k=\widetilde{k}+1}^{\widehat{k}}
\lf(\sum_{i\in F_1^{(k,\delta)}}\lf|\lz_i^k\r|^p\r)^{\frac qp}
\ls\|f\|^q_{H^{p,q}_A(\rn)}.
\end{eqnarray}
Observe that $\ell$ and $\ell_1^\epsilon$ are continuous and have vanishing
moments up to $s$ and hence, so does $\ell_2^\epsilon$. Moreover,
$\supp\ell_2^\epsilon\subset B_{k_0+4\tau}$ and
$\|\ell_2^\epsilon\|_{L^\fz(\rn)}\le C_{10}(\widehat{k}-\widetilde{k})\epsilon$
by \eqref{sevene18}. Therefore, we choose
$\epsilon$ small enough such that $\ell_2^\epsilon$ becomes an arbitrarily
small multiple of a continuous $(p,\fz,s)$-atom. Indeed,
$\ell_2^\epsilon=\lz^\epsilon a^\epsilon$, where
$\lz^\epsilon:=C_{10}(\widehat{k}-\widetilde{k})\epsilon|B_{k_0+4\tau}|^{1/p}$
and $a^\epsilon$ is a continuous $(p,\fz,s)$-atom. Thus,
$f=h+\ell_1^\epsilon+\ell_2^\epsilon$ gives the desired finite atomic
decomposition of $f$. Then, by \eqref{sevene19} and the fact that $h/C_9$ is
a $(p,\fz,s)$-atom, we have
$$\|f\|_{H_{A,{\rm fin}}^{p,\fz,s,q}(\rn)}\le
\|h\|_{H_{A,{\rm fin}}^{p,\fz,s,q}(\rn)}
+\|\ell_1^\epsilon\|_{H_{A,{\rm fin}}^{p,\fz,s,q}(\rn)}
+\|\ell_2^\epsilon\|_{H_{A,{\rm fin}}^{p,\fz,s,q}(\rn)}\ls1.$$
This finishes the proof of (ii) and hence Theorem \ref{sevent1}.
\end{proof}

\section{Some applications\label{s6}}

\hskip\parindent
In this section, we give some applications. In Subsection \ref{s6.1},
we consider the interpolation
properties of the anisotropic Hardy-Lorentz space
$H^{p,q}_A(\rn)$ via the real method. In Subsection \ref{s6.2},
we first obtain the boundedness of the $\dz$-type Calder\'on-Zygmund operators
from $H^p_A(\rn)$ to $L^{p,\fz}(\rn)$ (or $H^{p,\fz}_A(\rn)$)
in the critical case. Then we prove that some Calder\'{o}n-Zygmund
operators are bounded from $H^{p,q}_A(\rn)$
to $L^{p,\fz}(\rn)$.
In addition, as an application of the finite atomic decomposition
characterizations of $H^{p,q}_A(\rn)$ in Theorem \ref{sevent1},
we establish a criterion for the boundedness of sublinear
operators from $H_A^{p,q}(\mathbb{R}^n)$ into a quasi-Banach
space, which is of independent interest.
Moreover, using this criterion, we further obtain
the boundedness of the $\dz$-type Calder\'on-Zygmund operators
from $H_A^{p,q}(\rn)$ to $L^{p,q}(\rn)$ (or $H_A^{p,q}(\rn)$).

\subsection{ Interpolation of $H^{p,q}_A(\rn)$}\label{s6.1}

\hskip\parindent
In this subsection, as an application of the atomic
decomposition for the anisotropic Hardy-Lorentz space $H^{p,q}_A(\rn)$,
we prove the real interpolation
properties on $H^{p,q}_A(\rn)$
(see Theorem \ref{sixt2} below), whose isotropic version
includes \cite[Theorem 2.5]{wa07} as a special case (see Remark \ref{sixr4}(ii) below).

We first recall some basic notions about the theory
of real interpolation. Assume that $(X_1,\,X_2)$ is a
compatible couple of quasi-normed spaces, namely,
$X_1$ and $X_2$ are two quasi-normed linear spaces
which are continuously embedded in some larger
topological vector space. Let $X_1+X_2:=\{f_1+f_2:\
f_1\in X_1,\,f_2\in X_2\}$. For $t\in(0,\fz]$,
the Peetre $K$-functional on $X_1+X_2$ is defined as
$$K(t,f;X_1,X_2):=\inf\lf\{\|f_1\|_{X_1}+t\|f_2\|_{X_2}:\
f=f_1+f_2,\,f_1\in X_1\ {\rm and}\ f_2\in X_2\r\}.$$
Moreover, for $\theta\in(0,1)$ and $q\in(0,\fz]$,
the interpolation space $(X_1,\,X_2)_{\theta,q}$
is defined as
\begin{eqnarray*}
(X_1,\,X_2)_{\theta,q}:=\lf\{f\in X_1+X_2:\
\|f\|_{\theta,q}:=\lf(\int_0^\fz\lf[t^{-\theta}
K(t,f;X_1,X_2)\r]^q\,\frac{dt}t\r)^{1/q}<\fz\r\}.
\end{eqnarray*}

It is well known that
$$\lf(L^{q_1}(\rn),L^{q_2}(\rn)\r)_{\theta,q}
=L^q(\rn)\ \ \ {\rm and}\ \ \  \lf(L^{p,q_1}(\rn),
L^{p,q_2}(\rn)\r)_{\theta,q}=L^{p,q}(\rn),$$
where $p\in(0,\fz)$, $0<q_1\leq q_2\le\fz$, $q\in[q_1,q_2]$ and
$\theta\in(0,1)$ satisfying that
$1/q=(1-\theta)/q_1+\theta/q_2$ (see \cite{bl76}).

The main result of this subsection is the following real interpolation
properties of $H^{p,q}_A(\rn)$.

\begin{theorem}\label{sixt2}
Let $p\in(0,\fz)$ and $q_1,\,q,\,q_2\in(0,\fz]$.
\begin{enumerate}
\item[{\rm (i)}] If $p_1,\,p_2\in(0,\fz)$, $p_1\neq p_2$,
$1/p=(1-\theta)/p_1+\theta/p_2$ and $\theta\in(0,1)$,
then
\begin{eqnarray}\label{sixe2}
\lf(H_A^{p_1,q_1}(\rn),H_A^{p_2,q_2}(\rn)\r)_{\theta,q}
=H^{p,q}_A(\rn).
\end{eqnarray}
\item[{\rm (ii)}] If
$1/q=(1-\theta)/q_1+\theta/q_2$ and $\theta\in(0,1)$,
then
\begin{eqnarray}\label{sixe21}
\lf(H_A^{p,q_1}(\rn),H_A^{p,q_2}(\rn)\r)_{\theta,q}
=H^{p,q}_A(\rn).
\end{eqnarray}
\end{enumerate}
\end{theorem}

In order to prove Theorem \ref{sixt2}, we need the following
technical lemma on the decomposition of a function into
its ``good" and ``bad" parts, whose proof is similar to
\cite[Lemma 5.7 and Lemma 5.10(ii)]{mb03}, the details
being omitted.

\begin{lemma}\label{sixl2}
Let $p\in(0,1]$, $N\in[N_{(p)},\fz)\cap\mathbb{Z}$,
$f\in C_c^\fz(\rn)$, $\lambda\in(0,\fz)$ and
$$\Omega_\lz:=\lf\{x\in\rn:\ M_N(f)(x)>\lz\r\}.$$
Then there exist
two functions $g_\lz$ and $b_\lz$ such that $f=g_\lz+b_\lz$ and
\begin{eqnarray*}
\|g_\lz\|_{L^\fz(\rn)}\le C_{11}\lz,\ \ \ \ \
\|b_\lz\|_{H_A^p(\rn)}^p\le C_{12}\int_{\Omega_\lz}
\lf[M_N(f)(x)\r]^p\,dx,
\end{eqnarray*}
where $C_{11}$ and $C_{12}$ are positive constants independent of
$f$ and $\lz$, and $H_A^p(\rn)$ is the anisotropic Hardy space introduced
in \cite{mb03}.
\end{lemma}

By Lemma \ref{sixl2} and an argument parallel to the proof of
\cite[Theorem 1]{frs74}, we obtain the following real interpolation
properties, the details being omitting.
\begin{lemma}\label{sixl3}
Assume that $p_0\in(0,1]$ and $q\in(0,\fz]$ satisfying that
$1/p=(1-\theta)/p_0$ and $\theta\in(0,1)$.
Then
\begin{eqnarray}\label{sixe1}
\lf(H^{p_0}_A(\rn),L^\fz(\rn)\r)_{\theta,q}=H_A^{p,q}(\rn).
\end{eqnarray}
\end{lemma}

\begin{remark}\label{sixr3}
If $A:=d\,{\rm I}_{n\times n}$ for some $d\in\rr$ with $|d|\in(1,\fz)$,
then $H^{p_0}_A(\rn)$ and $H^{p,q}_A(\rn)$ in Lemma \ref{sixl3}
become the classical isotropic Hardy and Hardy-Lorentz spaces,
respectively. In this case, if $p_0\in(0,1]$ and $q\in(0,\fz]$ satisfying that
$1/p=(1-\theta)/p_0$ and $\theta\in(0,1)$,
then
\begin{eqnarray*}
\lf(H^{p_0}(\rn),L^\fz(\rn)\r)_{\theta,q}=H^{p,q}(\rn),
\end{eqnarray*}
which is just \cite[Theorem 1]{frs74}.
\end{remark}

Now we employ Lemma \ref{sixl3} to prove Theorem \ref{sixt2}(i).

\begin{proof}[Proof of Theorem \ref{sixt2}(i)]
Indeed,
if $p_1,\,p_2\in(0,\fz)$ and $p_1\neq p_2$, then
there exist $r\in(0,\min\{p_1,p_2,1\})$ and $\eta_1,\,\eta_2\in(0,1)$ such that
$1/p_i=(1-\eta_i)/r$, $i\in\{1,2\}$.
Let $\eta:=(1-\theta)\eta_1+\theta\eta_2$. Noticing that
$1/p=(1-\theta)/p_1+\theta/p_2=(1-\eta)/r$,
by Lemma \ref{sixl3} and the reiteration theorem
(see, for example, \cite[Theorem 2]{m84}), we know that
\begin{eqnarray*}
\lf(H_A^{p_1,q_1}(\rn),H_A^{p_2,q_2}(\rn)\r)_{\theta,q}
&&=\lf(\lf(H_A^r(\rn),L^\fz(\rn)\r)_{\eta_1,q_1},
\lf(H_A^r(\rn),L^\fz(\rn)\r)_{\eta_2,q_2}\r)_{\theta,q}\\
&&=\lf(H_A^r(\rn),L^\fz(\rn)\r)_{\eta,q}=H^{p,q}_A(\rn),
\end{eqnarray*}
which is the desired conclusion \eqref{sixe2}.
This finishes the proof of Theorem \ref{sixt2}(i).
\end{proof}

As an immediate consequence of Theorem \ref{sixt2}(i), we easily know
that the anisotropic Hardy-Lorentz space $H^{p,q}_A(\rn)$ serves as
a median space
between two anisotropic Hardy spaces via the real method,
which is the following Corollary \ref{sixc1}.

\begin{corollary}\label{sixc1}
Assume that $p,\,p_1,\,p_2\in(0,\fz)$, $p_1\neq p_2$
and $q\in(0,\fz]$ satisfying that
$1/p=(1-\theta)/p_1+\theta/p_2$ and $\theta\in(0,1)$.
Then
\begin{eqnarray*}
\lf(H_A^{p_1}(\rn),H_A^{p_2}(\rn)\r)_{\theta,q}
=H^{p,q}_A(\rn).
\end{eqnarray*}
\end{corollary}

\begin{remark}\label{sixr5}
(i) If $A:=d\,{\rm I}_{n\times n}$ for some $d\in\rr$ with $|d|\in(1,\fz)$,
then $H^{p_1}_A(\rn)$, $H^{p_2}_A(\rn)$ and $H^{p,q}_A(\rn)$
in Corollary \ref{sixc1}
become the classical isotropic Hardy and Hardy-Lorentz spaces,
respectively. In this case, if $p,\,p_1,\,p_2\in(0,\fz)$, $p_1\neq p_2$
and $q\in(0,\fz]$ satisfying that $1/p=(1-\theta)/p_1+
\theta/p_2,\ \theta\in(0,1)$, then, by Corollary \ref{sixc1}, we have
$$\lf(H^{p_1}(\rn),H^{p_2}(\rn)\r)_{\theta,q}
=H^{p,q}(\rn).$$
In particular, $\lf(H^{p_1}(\rn),H^{p_2}(\rn)\r)_{\theta,p}
=H^p(\rn)$, provided that $1/p=(1-\theta)/p_1+
\theta/p_2,\ \theta\in(0,1)$.

(ii) If $p\in(1,\fz)$ and $q\in(0,\fz]$, then $H^{p,q}_A(\rn)=L^{p,q}(\rn)$.
Indeed, for any $p\in(1,\fz)$, there exist $p_1,\,p_2\in(1,\fz)$, $p_1\neq p_2$
and $\theta\in(0,1)$ such that $1/p=(1-\theta)/p_1+
\theta/p_2$. From this, Corollary \ref{sixc1}, $H^r_A(\rn)=L^r(\rn)$
for all $r\in(1,\fz)$ (see \cite[p.\,16, Remark]{mb03}) and the corresponding
interpolation result of Lorentz spaces (see, for example, \cite[Theorem 3]{m84}),
we deduce that
$$H^{p,q}_A(\rn)=\lf(H_A^{p_1}(\rn),H_A^{p_2}(\rn)\r)_{\theta,q}
=\lf(L^{p_1}(\rn),L^{p_2}(\rn)\r)_{\theta,q}=L^{p,q}(\rn).$$
\end{remark}

Now we turn to prove Theorem \ref{sixt2}(ii) via Remark \ref{sixr5}(ii).

\begin{proof}[Proof of Theorem \ref{sixt2}(ii)]
To show Theorem \ref{sixt2}(ii),
we consider two cases. If $p\in(0,1]$, by a proof similar
to that of \cite[Theorem 2.5]{wa07}, we easily obtain the desired
conclusion \eqref{sixe21}. If $p\in(1,\fz)$, by Remark \ref{sixr5}(ii)
and the interpolation properties of Lorentz spaces
(see, for example, \cite[Theorem 5.3.1]{bl76}),
we find that \eqref{sixe21} holds true. This finishes the proof of
Theorem \ref{sixt2}(ii)
and hence Theorem \ref{sixt2}.
\end{proof}

\begin{remark}\label{sixr4}
(i) If $A:=d\,{\rm I}_{n\times n}$ for some $d\in\rr$ with $|d|\in(1,\fz)$,
then $H^{p_i,q_i}_A(\rn)$, $H^{p,q_i}_A(\rn)$,
$i\in\{1,2\}$, and $H^{p,q}_A(\rn)$ in Theorem \ref{sixt2}
become the classical isotropic Hardy-Lorentz spaces.
In this case, by Theorem \ref{sixt2}(i), we know that
\begin{eqnarray*}
\lf(H^{p_1,q_1}(\rn),H^{p_2,q_2}(\rn)\r)_{\theta,q}
=H^{p,q}(\rn),
\end{eqnarray*}
provided that $p_1,\,p,\,p_2\in(0,\fz)$, $p_1\neq p_2$
and $q_1,\,q,\,q_2\in(0,\fz]$ satisfying that
$1/p=(1-\theta)/p_1+
\theta/p_2,\ \theta\in(0,1)$,
which is a well-known interpolation result for
classical isotropic Hardy-Lorentz spaces (see \cite[p.\,75, (2)]{frs74}).
In addition, by Theorem \ref{sixt2}(ii), we have
\begin{eqnarray*}
\lf(H^{p,q_1}(\rn),H^{p,q_2}(\rn)\r)_{\theta,q}
=H^{p,q}(\rn),
\end{eqnarray*}
provided that $p\in(0,\fz)$ and
$q_1,\,q,\,q_2\in(0,\fz]$ satisfying that
$1/q=(1-\theta)/q_1+\theta/q_2,\ \theta\in(0,1)$,
which generalizes \cite[Theorem 2.5]{wa07}.

(ii) Lemma \ref{sixl3} also holds true for all $p_0\in(1,\fz)$
and $q\in(0,\fz]$. Indeed, noticing that, if $p_0\in(1,\fz)$,
then $p\in(1,\fz)$. Thus, by Remark \ref{sixr5}(ii),
we have $H_A^{p_0}(\rn)=L^{p_0}(\rn)$
and $H_A^{p,q}(\rn)=L^{p,q}(\rn)$.
From this and the fact that, for all $q\in(0,\fz]$,
$$\lf(L^{p_0}(\rn),L^\fz(\rn)\r)_{\theta,q}=L^{p,q}(\rn)\ \ \ {\rm with}\
\frac1p=\frac{1-\theta}{p_0}\ {\rm and}\ \theta\in(0,1)$$
(see \cite[Theorem 5.3.1]{bl76}), we further deduce that
\eqref{sixe1} holds true for all $p_0\in(1,\fz)$ and $q\in(0,\fz]$.
\end{remark}

\subsection{Boundedness of Calder\'on-Zygmund operators}\label{s6.2}

\hskip\parindent
As another application of the atomic
decomposition for $H^{p,q}_A(\rn)$, in this subsection,
we first obtain the boundedness of the $\dz$-type Calder\'on-Zygmund operators
from $H^p_A(\rn)$ to $L^{p,\fz}(\rn)$ (or $H^{p,\fz}_A(\rn)$)
in the critical case (see Theorem \ref{sixt3} and Remark \ref{sixr1} below).
As the third application of Theorem \ref{tt1},
we also prove that some Calder\'{o}n-Zygmund
operators are bounded from $H^{p,q}_A(\rn)$
to $L^{p,\fz}(\rn)$ (see Theorem \ref{sixt1} below).
This application is a generalization of \cite[Theorem 2.2]{wa07} in the present setting.
In addition, as an application of the finite atomic decomposition characterizations
for $H^{p,q}_A(\rn)$ in Theorem \ref{sevent1},
we establish a criterion for the boundedness of sublinear
operators from $H_A^{p,q}(\mathbb{R}^n)$ into a quasi-Banach
space (see Theorem \ref{sevent2} below), which is of independent interest.
Moreover, using the criterion, we further obtain
the boundedness of the $\dz$-type Calder\'on-Zygmund operators
from $H_A^{p,q}(\rn)$ to $L^{p,q}(\rn)$ (or $H_A^{p,q}(\rn)$) with
$\delta\in(0,\frac{\ln\lz_-}{\ln b}]$,
$p\in(\frac1{1+\delta},1]$ and $q\in(0,\fz]$ (see Theorem \ref{sixt4} below).

As the first main result of this subsection, the following Theorem \ref{sixt3}
is the boundedness of the $\dz$-type Calder\'on-Zygmund operators
from $H^p_A(\rn)$ to $L^{p,\fz}(\rn)$ (or $H^{p,\fz}_A(\rn)$)
in the critical case.

\begin{theorem}\label{sixt3}
Let $\dz\in(0,\frac{\ln\lambda_-}{\ln b}]$ and
$p=\frac 1{1+\dz}$. If $k\in\cs'(\rn)$ coincides with a
locally integrable function on $\rn\setminus\{0_n\}$ and
there exist two positive constants $C_{13}$ and $C_{14}$,
independent of $f,\,x$ and $y$, such that
$\|k*f\|_{L^{2}(\rn)}\le C_{13}\|f\|_{L^{2}(\rn)}$
and, when $\rho(x)\ge b^{2\tau}\rho(y)$,
\begin{eqnarray}\label{sixe7}
|k(x-y)-k(x)|\le C_{14}\frac{\lf[\rho(y)\r]^\dz}{\lf[\rho(x)\r]^{1+\dz}},
\end{eqnarray}
then $T(f):=k*f$ for $f\in L^{2}(\rn)\cap H^p_A(\rn)$
has a unique extension on $H^p_A(\rn)$ and, moreover,
there exist two positive constants $C_{15}$ and $C_{16}$ such that,
for all $f\in H^p_A(\rn)$,
\begin{eqnarray}\label{sixe17}
\|T(f)\|_{L^{p,\fz}(\rn)}\le C_{15}\|f\|_{H^p_A(\rn)}
\end{eqnarray}
and
\begin{eqnarray}\label{sixe16}
\|T(f)\|_{H^{p,\fz}_A(\rn)}\le C_{16}\|f\|_{H^p_A(\rn)}.
\end{eqnarray}
\end{theorem}

To show Theorem \ref{sixt3}, we need
the following weak-type summable principle, which is from \cite[p.\,9]{fs87}
(see also \cite[p.\,114]{l91}).

\begin{lemma}\label{sixl1}
Let $p\in(0,1)$, $(X,\mu)$ be any measurable space and
$\{f_j\}_{j\in\nn}$ be a sequence of
measurable functions such that, for all
$j\in\nn$ and $\lz\in(0,\fz)$,
\begin{eqnarray*}
\mu\lf(\lf\{x\in X:\ |f_j(x)|>\lz\r\}\r)\leq C\lz^{-p},
\end{eqnarray*}
where $C$ is a positive constant independent of $\lz$ and $j$.
If $\{c_j\}_{j\in\nn}\subset\mathbb{C}$ satisfies that
$\sum_{j\in\nn}|c_j|^p<\fz$, then $\sum_{j\in\nn}c_jf_j(x)$
is absolutely convergent almost everywhere and there exists a positive
constant $\widetilde{C}$ such that, for all $\lz\in(0,\fz)$,
\begin{eqnarray*}
\mu\lf(\lf\{x\in X:\ \lf|\sum_{j\in\nn}c_jf_j(x)\r|>\lz\r\}\r)
\leq \widetilde{C}\frac{2-p}{1-p}\lf[\sum_{j\in\nn}|c_j|^p\r]\lz^{-p}.
\end{eqnarray*}
\end{lemma}

Now we show Theorem \ref{sixt3}.

\begin{proof}[Proof of Theorem \ref{sixt3}]
We first prove \eqref{sixe17}.
By Theorem \ref{tt1}, to show \eqref{sixe17},
it suffices to prove that, for $h$ being a constant multiple of a $(p,\fz,s)$-atom
associated with ball $B:=x_0+B_\ell$ for some $x_0\in\rn$ and $\ell\in\zz$,
\begin{eqnarray}\label{sixe8}
\sup_{k\in\zz}2^{kp}\lf|\lf\{x\in\rn:\ |T(h)(x)|>2^k\r\}\r|
\ls\|h\|_{L^\fz(\rn)}^p|B|.
\end{eqnarray}
Indeed,
by Theorem \ref{tt1} with $p=q$, we find that,
for any $f\in H^p_A(\rn)$, there exists a sequence of  constant multiples of
$(p,\fz,s)$-atoms, $\{h_{j}\}_{j\in\nn}$,
associated with balls $\{B_{j}\}_{j\in\nn}$,
such that $f=\sum_{j\in\nn}h_j$
in $\cs'(\rn)$ and
\begin{eqnarray*}
\|f\|_{H^p_A(\rn)}\sim\lf[\sum_{j\in\nn}
\|h_j\|_{L^\fz(\rn)}^p|B_j|\r]^{\frac1p}.
\end{eqnarray*}
From this, \eqref{sixe8} and Lemma \ref{sixl1}, we deduce that
\begin{eqnarray}\label{sixe18}
\qquad&&\sup_{k\in\zz}2^{kp}\lf|\lf\{x\in\rn:\ |T(f)(x)|>2^k\r\}\r|\\
&&\hs\leq\sup_{k\in\zz}2^{kp}
\lf|\lf\{x\in\rn:\ \sum_{j\in\nn}\lf|T(h_j)(x)\r|>2^k\r\}\r|
\ls\sum_{j\in\nn}\|h_j\|_{L^\fz(\rn)}^p|B_j|\ls\|f\|_{H^p_A(\rn)}^p,\noz
\end{eqnarray}
which implies that $\|T(f)\|_{L^{p,\fz}_A(\rn)}\ls\|f\|_{H^p_A(\rn)}$. This is as desired.

It remains to prove \eqref{sixe8}. First, by
the boundedness of $T$ and the H\"older inequality, we know that
\begin{eqnarray}\label{sixe9}
&&\sup_{k\in\zz}2^{kp}\lf|\lf\{x\in A^{4\tau}B:\
|T(h)(x)|>2^k\r\}\r|\\
&&\hs\leq\int_{A^{4\tau}B}|T(h)(x)|^p\,dx
\ls|B|^{(\frac p2)'}\lf\{\int_{A^{4\tau}B}
|T(h)(x)|^2\,dx\r\}^{\frac p2}\noz\\
&&\hs\ls|B|^{(\frac p2)'}\lf\{\int_\rn
|h(x)|^2\,dx\r\}^{\frac p2}\ls\|h\|_{L^\fz(\rn)}^p|B|\noz.
\end{eqnarray}
On the other hand, by $\int_{\rn}h(x)\,dx=0$ and \eqref{sixe7}, we find that,
for all $x\in(A^{4\tau}B)^\com$,
\begin{eqnarray*}
|T(h)(x)|
&&\leq\int_B\lf|k(x-y)-k(x-x_0)\r||h(y)|\,dy\\
&&\ls\|h\|_{L^\fz(\rn)}\int_B\frac{\lf[\rho(y-x_0)\r]^
\delta}{\lf[\rho(x-x_0)\r]^{1+\delta}}
\ls\frac{|B|^{1+\delta}}{\lf[\rho(x-x_0)\r]
^{1+\delta}}\|h\|_{L^\fz(\rn)},
\end{eqnarray*}
which further implies that
\begin{eqnarray}\label{sixe14}
&&\sup_{k\in\zz}2^{kp}\lf|\lf\{x\in \lf(A^{4\tau}B\r)^\com:\
|T(h)(x)|>2^k\r\}\r|\\
&&\hs\ls\sup_{k\in\zz}2^{kp}\lf|\lf\{x\in \lf(A^{4\tau}B\r)^\com:\
\frac{|B|^{1+\delta}}{\lf[\rho(x-x_0)\r]
^{1+\delta}}\|h\|_{L^\fz(\rn)}>2^k\r\}\r|\noz\\
&&\hs\ls\sup_{k\in\zz\cap(-\fz,\|h\|_{L^\fz(\rn)})}2^{kp}
\lf[\frac{\|h\|_{L^\fz(\rn)}}{2^k}\r]^{\frac1{1+\delta}}|B|
\sim\|h\|_{L^\fz(\rn)}^p|B|.\noz
\end{eqnarray}
Then \eqref{sixe8} follows from \eqref{sixe9} and \eqref{sixe14}, which
completes the proof of \eqref{sixe17}.

Next we prove \eqref{sixe16}. To this end, similar to \eqref{sixe18},
it suffices to prove that
\begin{eqnarray}\label{sixe19}
\sup_{k\in\zz}2^{kp}\lf|\lf\{x\in\rn:\ M_N(T(h))(x)>2^k\r\}\r|
\ls\|h\|_{L^\fz(\rn)}^p|B|.
\end{eqnarray}
First, similar to \eqref{sixe9}, by the boundedness of $T$ and $M_N$ on $L^2(\rn)$
(see Remark \ref{tr1}), we easily conclude that
\begin{eqnarray}\label{sixe20}
&&\sup_{k\in\zz}2^{kp}\lf|\lf\{x\in A^{4\tau}B:\
M_N(T(h))(x)>2^k\r\}\r|\ls\|h\|_{L^\fz(\rn)}^p|B|.
\end{eqnarray}

By $\int_\rn h(x)\,dx=0$, we know that $\wh{T(h)}(0)=\wh k(0)\wh h(0)=0$
and hence $\int_\rn T(h)(x)\,dx=0$.
By this, we find that,
for all $\phi\in\cs_N(\rn)$, $k\in\zz$ and
$x\in(A^{4\tau}B)^\com$,
\begin{eqnarray}\label{sixe10}
\ \ &&|T(h)\ast\phi_k(x)|\\
&&\hs=b^{-k}\lf|\int_\rn T(h)(y)
\lf[\phi\lf(A^{-k}(x-y)\r)-\phi\lf(A^{-k}(x-x_0)\r)\r]\,dy\r|\noz\\
&&\hs\le b^{-k}\int_\rn \lf|T(h)(y)\r|
\lf|\phi\lf(A^{-k}(x-y)\r)-\phi\lf(A^{-k}(x-x_0)\r)\r|\,dy\noz\\
&&\hs\le b^{-k}\lf\{\int_{\rho(y-x_0)<b^{2\tau}|B|}+
\int_{b^{2\tau}|B|\le\rho(y-x_0)<b^{-2\tau}\rho(x-x_0)}
+\int_{\rho(y-x_0)\ge b^{-2\tau}\rho(x-x_0)}\r\}\noz\\
&&\hs\hs\times \lf|T(h)(y)\r|
\lf|\phi\lf(A^{-k}(x-y)\r)-\phi\lf(A^{-k}(x-x_0)\r)\r|\,dy
=:\mi_1+\mi_2+\mi_3.\noz
\end{eqnarray}

For $\mi_1$, by the mean value theorem, \cite[p.\,11, Lemma 3.2]{mb03},
the H\"older inequality
and the boundedness of $T$ on $L^2(\rn)$, we conclude that
there exists $\xi(y)\in A^{2\tau}B$ such that,
for all $\phi\in\cs_N(\rn)$, $k\in\zz$ and $x\in(A^{4\tau}B)^\com$,
\begin{eqnarray}\label{sixe11}
\ \ \ \ \mi_1&&=b^{-k}\int_{\rho(y-x_0)<b^{2\tau}|B|}\lf|T(h)(y)\r|
\lf|\phi\lf(A^{-k}(x-y)\r)-\phi\lf(A^{-k}(x-x_0)\r)\r|\,dy\\
&&\leq b^{-k}\int_{\rho(y-x_0)<b^{2\tau}|B|}\lf|T(h)(y)\r|
\lf|\sum_{|\bz|=1}\partial^\bz\phi\lf(A^{-k}\lf(x-\xi(y)\r)\r)\r|
\lf|A^{-k}(y-x_0)\r|\,dy\noz\\
&&\ls b^{-k}\int_{\rho(y-x_0)<b^{2\tau}|B|}\lf|T(h)(y)\r|
\frac{b^{k(1+\widetilde{\delta})}}{\lf[\rho(x-x_0)\r]^{1+\widetilde{\delta}}}
b^{-k\widetilde{\delta}}\lf[\rho(y-x_0)\r]^{\widetilde{\delta}}\,dy\noz\\
&&\ls\frac{|B|^{\widetilde{\delta}}}{\lf[\rho(x-x_0)\r]^{1+\widetilde{\delta}}}
\lf\{\int_{\rn}\lf[T(h)\r]^2(y)\r\}^{\frac12}|B|^{\frac12}
\ls\frac{|B|^{1+\widetilde{\delta}}}{\lf[\rho(x-x_0)\r]^{1+\widetilde{\delta}}}
\|h\|_{L^\fz(\rn)},\noz
\end{eqnarray}
where
\begin{eqnarray*}
\widetilde{\delta}:=\left\{
\begin{array}{cl}
&(\ln\lz_+)/(\ln b)
\ \ \ {\rm when}\ \ \ \rho(y-x_0)\ge1,\\
&(\ln\lz_-)/(\ln b)\ \ \
{\rm when}\ \ \ \rho(y-x_0)<1.
\end{array}\r.
\end{eqnarray*}

For $\mi_2$, by $\int_\rn h(x)\,dx=0$, \eqref{sixe7}
and the mean value theorem,
we know that,
for all $\phi\in\cs_N(\rn)$, $k\in\zz$ and $x\in(A^{4\tau}B)^\com$,
\begin{eqnarray}\label{sixe12}
\mi_2
&&=b^{-k}\int_{b^{2\tau}|B|\le\rho(y-x_0)<b^{-2\tau}\rho(x-x_0)}
\lf|\int_Bh(z)k(y-z)\,dz\r|\\
&&\hs\times\lf|\phi\lf(A^{-k}(x-y)\r)-\phi\lf(A^{-k}(x-x_0)\r)\r|\,dy\noz\\
&&\ls\int_{b^{2\tau}|B|\le\rho(y-x_0)<b^{-2\tau}\rho(x-x_0)}
\lf\{\int_B|h(z)|\r.\noz\\&&\hs\times|k(y-z)-k(y-x_0)|\,dz\Bigg\}
\frac{\lf[\rho(y-x_0)\r]^{\widetilde{\delta}}}
{\lf[\rho(x-x_0)\r]^{1+\widetilde{\delta}}}\,dy\noz\\
&&\ls\frac{\|h\|_{L^\fz(\rn)}}{\lf[\rho(x-x_0)\r]^{1+\widetilde{\delta}}}
\int_{b^{2\tau}|B|\le\rho(y-x_0)<b^{-2\tau}\rho(x-x_0)}
\frac{|B|^{1+\delta}}{\lf[\rho(y-x_0)\r]^{1+\delta-\widetilde{\delta}}}\,dy\noz\\
&&\ls\frac{|B|^{1+\widetilde{\delta}}}{\lf[\rho(x-x_0)\r]^{1+\widetilde{\delta}}}
\|h\|_{L^\fz(\rn)},\noz
\end{eqnarray}
where $\widetilde{\delta}$ is as in \eqref{sixe11}.

For $\mi_3$, by the fact that $\int_\rn b(x)\,dx=0$, \eqref{sixe7}
and $\phi\in\cs_N(\rn)$, we find that,
for all $k\in\zz$ and $x\in(A^{4\tau}B)^\com$,
\begin{eqnarray}\label{sixe13}
\ \ \ \mi_3&&=\int_{\rho(y-x_0)\ge b^{-2\tau}\rho(x-x_0)}
\lf|\int_Bh(z)k(y-z)\,dz\r|\lf|\phi_k(x-y)\r|\,dy\\
&&\le\int_{\rho(y-x_0)\ge b^{-2\tau}\rho(x-x_0)}
\lf[\int_B|h(z)||k(y-z)-k(y-x_0)|\,dz\r]\lf|\phi_k(x-y)\r|\,dy\noz\\
&&\ls\|h\|_{L^\fz(\rn)}\int_{\rho(y-x_0)\ge b^{-2\tau}\rho(x-x_0)}
\lf\{\int_B\frac{\lf[\rho(z-x_0)\r]^\delta}
{\lf[\rho(y-x_0)\r]^{1+\delta}}\,dz\r\}\lf|\phi_k(x-y)\r|\,dy\noz\\
&&\ls\frac{|B|^{1+\delta}}{\lf[\rho(x-x_0)\r]
^{1+\delta}}\|h\|_{L^\fz(\rn)}\int_\rn\lf|\phi_k(x-y)\r|\,dy\ls\frac{|B|^{1+\delta}}{\lf[\rho(x-x_0)\r]
^{1+\delta}}\|h\|_{L^\fz(\rn)}.\noz
\end{eqnarray}
Combining \eqref{sixe10}, \eqref{sixe11}, \eqref{sixe12}, \eqref{sixe13}
and Proposition \ref{sp1}, we know that, for all $x\in(A^{4\tau}B)^\com$,
\begin{eqnarray*}
M_N\lf(T(h)\r)(x)\ls\sup_{\phi\in\cs_N(\rn)}\sup_{k\in\zz}\lf|T(h)\ast\phi_k(x)\r|
\ls\frac{|B|^{1+\delta}}{\lf[\rho(x-x_0)\r]
^{1+\delta}}\|h\|_{L^\fz(\rn)},
\end{eqnarray*}

From this, together with an argument parallel to \eqref{sixe14}, we further deduce that
\begin{eqnarray*}
\sup_{k\in\zz}2^{kp}\lf|\lf\{x\in \lf(A^{4\tau}B\r)^\com:\
M_N(T(h))(x)>2^k\r\}\r|\ls\|h\|_{L^\fz(\rn)}^p|B|,
\end{eqnarray*}
which, combined with \eqref{sixe20}, implies \eqref{sixe19}.
This finishes the proof of \eqref{sixe16} and hence Theorem
\ref{sixt3}.
\end{proof}

\begin{remark}\label{sixr1}
(i) If $A:=d\,{\rm I}_{n\times n}$ for some $d\in\rr$ with $|d|\in(1,\fz)$,
then $\frac{\ln\lambda_-}{\ln b}=\frac1n$ and, $H^p_A(\rn)$
and $H^{p,\fz}_A(\rn)$ become the classical isotropic Hardy and weak Hardy spaces,
respectively. In this case, we know, by Theorem \ref{sixt3}, that, if
$\delta\in(0,1]$, $p=\frac n{n+\delta}$ and $T$ is the Calder\'on-Zygmund operator
satisfying all conditions of Theorem \ref{sixt3} with \eqref{sixe7} replaced by
\begin{eqnarray*}
|k(x-y)-k(x)|\ls \frac{|y|^\dz}{|x|^{n+\dz}},\ \ \ |x|\ge 2|y|,
\end{eqnarray*}
then $T$ is bounded from $H^{\frac n{n+\dz}}(\rn)$ to
$H^{\frac n{n+\dz},\fz}(\rn)$, which is just \cite[Theorem 1]{l91}.
Here $\frac n{n+\dz}$ is called the \emph{critical index}.
In this sense, Theorem \ref{sixt3} also establishes the boundedness
of Calder\'on-Zygmund operators from $H^p_A(\rn)$ to $L^{p,\fz}(\rn)$
in the critical case under the anisotropic setting.

(ii) Let $\dz\in(0,1]$.
A \emph{non-convolutional $\dz$-type Calder\'on-Zygmund
operator} $T$ is a linear  operator  which is bounded on $L^2(\rn)$ and
satisfies that, for all $f\in L^2(\rn)$ with compact support and $x\notin\supp(f)$,
$$T(f)(x)=\int_{\supp (f)}\mathcal{K}(x,y)f(y)\,dy,$$
where $\mathcal{K}$ denotes a standard kernel on
$(\rn\times\rn)\setminus\{(x,x):\ x\in\rn\}$ in the following sense:
there exists a positive constant $C$ such that,
for all $x,\,y,\,z\in\rn$,
$$|\mathcal{K}(x,y)|\le \frac C{\rho(x-y)}\quad{\rm if}\quad x\neq y$$
and
\begin{eqnarray}\label{sixe15}
|\mathcal{K}(x,y)-\mathcal{K}(x,z)|\le C\frac{\lf[\rho(y-z)\r]^\dz}
{\lf[\rho(x-y)\r]^{1+\dz}}
\quad{\rm if}\quad \rho(x-y)>b^{2\tau}\rho(y-z).
\end{eqnarray}

By an argument similar to that used in the proof of \eqref{sixe16} in Theorem \ref{sixt3},
we find that \eqref{sixe16} also holds true for non-convolutional
$\dz$-type Calder\'on-Zygmund operators $T$ with the additional
assumption that $T^*1=0$
(namely, for all $a\in L^1(\rn)$ with compact support, if
$\int_\rn a(x)\,dx=0$, then $\int_\rn T(a)(x)\,dx=0$), the details
being omitted.

(iii)  Following the proof of
\eqref{sixe17} in Theorem \ref{sixt3},
we know that \eqref{sixe17} also holds true when  $T$
is a non-convolutional
$\dz$-type Calder\'on-Zygmund operator.

(iv) Let $\delta\in(0,\frac{\ln\lz_-}{\ln b}]$ and $p\in(\frac1{1+\delta},1]$.
If $T$ is either a convolutional
$\dz$-type Calder\'on-Zygmund operator as  in Theorem \ref{sixt3} or
a non-convolutional
$\dz$-type Calder\'on-Zygmund operator with the additional assumption
that $T^*1=0$, as in (ii) of this remark, then, by a similar proof to that of
\cite[p.\,68, Theorem 9.8]{mb03},
we conclude that
$T$ is bounded from $H^p_A(\rn)$ to $H^p_A(\rn)$.
Moreover, by an argument parallel to the proof of
\cite[p.\,69, Theorem 9.9]{mb03},
we know that, if $T$ is either a convolutional
$\dz$-type Calder\'on-Zygmund operator as  in Theorem \ref{sixt3}
or a non-convolutional
$\dz$-type Calder\'on-Zygmund operator $T$,
then $T$ is bounded from $H^p_A(\rn)$ to $L^p(\rn)$.
Comparing these with  Theorem \ref{sixt3} and (ii) and (iii) of this remark,
we know that the latter further completes the boundedness of
these operators in the critical case by establishing the bounedness
from $H^p_A(\rn)$ to $L^{p,\fz}(\rn)$ (or $H^{p,\fz}_A(\rn)$).
\end{remark}

Another interesting application of the atomic decomposition for $H^{p,q}_A(\rn)$
is to obtain the following boundedness of Calder\'on-Zygmund
operators from $H^{p,q}_A(\rn)$ to $L^{p,\fz}(\rn)$
with $p\in(0,1]$ and $q\in(p,\fz]$.

\begin{theorem}\label{sixt1}
Suppose that $p\in(0,1],\,q\in(p,\fz],\,r\in(1,\fz)$
and $T$ is a Calder\'{o}n-Zygmund
operator associated with the kernel $k$.
Moreover, assume that $T$ is bounded from
$L^r(\rn)$ to $L^{r,\fz}(\rn)$ and
$\omega_p$ satisfies a Dini-type condition of order $q/(q-p)$,
namely,
\begin{eqnarray}\label{d-c}
A_{(p,q)}:=\lf\{\int_0^1\lf[\omega_p(\delta)\r]^
{q/(q-p)}\,\frac{d\delta}\delta\r\}^{(q-p)/q}<\fz,
\end{eqnarray}
where, for $\delta\in(0,1]$,
\begin{eqnarray*}
\omega_p(\delta):=\sup_{B}
\frac1{|B|}\int_{\rho(x-y_B)>\frac{b^{2\tau}}\delta|B|}
\lf[\dis\int_{B}\lf|k(x,y)-\sum_{|\beta|\leq N}(y-y_B)
^\beta k_\beta(x,y_B)\r|\,dy\r]^p\,dx,
\end{eqnarray*}
$N:=\lfloor(1/p-1)\frac{\ln b}{\ln\lambda_-}\rfloor,
\,\beta:=(\beta_1,\cdots,\beta_n)\in\zz_+^n$,
$k_\beta(x,y_B):=
\frac1{\beta!}D^\beta k(x,y)|_{y=y_B}$
and the supremum is taken over
all dilated balls $B\in\mathfrak{B}$ centered at $y_B$.
Then $T$ is bounded from $H^{p,q}_A(\rn)$ to
$L^{p,\fz}(\rn)$ and, moreover, there exists a
positive constant $C$ such that, for all
$f\in H^{p,q}_A(\rn)$,
$$\lf\|T(f)\r\|_{L^{p,\fz}(\rn)}\leq C
\lf[A_{(p,q)}\r]^{\frac1p}\|f\|_{H^{p,q}_A(\rn)}.$$
\end{theorem}

\begin{proof}
Let $p\in(0,1],\,q\in(p,\fz]$ and $r\in(1,\fz)$.
For all $f\in H^{p,q}_A(\rn)$, by Theorem \ref{tt1} and
Definition \ref{d-aahls}, we know that there exists
a sequence of $(p,\fz,s)$-atoms,
$\{a_i^k\}_{i\in\mathbb{N},\,k\in\mathbb{Z}}$,
respectively supported on
$\{x_i^k+B_i^k\}_{i\in\mathbb{N},\,k\in\mathbb{Z}}
\subset\mathfrak{B}$ such that
$f=\sum_{k\in\mathbb{Z}}
\sum_{i\in\mathbb{N}}\lambda_i^ka_i^k$
in $\cs'(\rn)$,
$\lambda_i^k\sim2^k|B_i^k|^{1/p}$
for all $k\in\mathbb{Z}$ and $i\in\mathbb{N}$,
$\sum_{i\in\mathbb{N}}\chi_{x_i^k+B_i^k}(x)\ls1$
for all $k\in\mathbb{Z}$ and $x\in\rn$, and
\begin{eqnarray*}
\|f\|_{H^{p,q}_A(\rn)}\sim\lf\|\lf\{\mu_k\r\}
_{k\in\mathbb{Z}}\r\|_{\ell^q},
\end{eqnarray*}
where
$\mu_k:=(\sum_{i\in\mathbb{N}}
|\lambda_i^k|^p)^{1/p}$.
For all $k_0\in\mathbb{Z}$, let
$f_1:=\sum_{k=-\fz}^{k_0}
\sum_{i\in\mathbb{N}}\lambda_i^ka_i^k$
and $f_2:=f-f_1$.
Since $\frac{rq}p\in(1,\fz]$, by the H\"{o}lder inequality,
we have
\begin{eqnarray*}
\|f_1\|_{L^r(\rn)}
&&\leq\sum_{k=-\fz}^{k_0}\lf\|\sum_{i\in\mathbb{N}}
\lambda_i^ka_i^k\r\|_{L^r(\rn)}\sim\sum_{k=-\fz}^{k_0}\lf\{\int_{\bigcup_{i\in\nn}
(x_i^k+B_i^k)}\lf|\sum_{i\in\nn}\lambda_i^ka_i^k(x)
\r|^r\,dx\r\}^{1/r}\noz\\
&&\ls\sum_{k=-\fz}^{k_0}2^k\lf(\sum_{i\in\nn}\lf|
B_i^k\r|\r)^{1/r}\ls\sum_{k=-\fz}^{k_0}2^{k(1-\frac pr)}
\lf(\sum_{i\in\nn}\lf|\lambda_i^k\r|^p\r)^{1/r}\noz\\
&&\ls 2^{k_0(1-\frac pr)}\lf\|\lf\{\mu_k\r\}_{k\in\zz}\r\|
_{\ell^q}^{\frac pr}
\sim2^{k_0(1-\frac pr)}\|f\|_{H^{p,q}_A(\rn)}^{\frac pr},\noz
\end{eqnarray*}
which, together with the boundedness from $L^r(\rn)$ to
$L^{r,\fz}(\rn)$ of $T$, implies that
\begin{eqnarray}\label{sixe3}
2^{pk_0}\lf|\lf\{x\in\rn:\ |T(f_1)(x)|>
2^{k_0}\r\}\r|\ls\|f\|_{H^{p,q}_A(\rn)}^p.
\end{eqnarray}

To complete the proof of Theorem \ref{sixt1},
it suffices to prove that, for all $k_0\in\zz$
\begin{eqnarray}\label{sixe4}
2^{pk_0}\lf|\lf\{x\in\rn:\ |T(f_2)(x)|>
2^{k_0}\r\}\r|\ls A_{(p,q)}\|f\|_{H^{p,q}_A(\rn)}^p.
\end{eqnarray}
Indeed, if \eqref{sixe4} is true, then, by \eqref{sixe3},
we further conclude that
\begin{eqnarray}\label{sixe5}
&&2^{pk_0}\lf|\lf\{x\in\rn:\ |T(f)(x)|>2^{k_0}\r\}\r|\\
&&\hs\leq 2^{pk_0}\lf|\lf\{x\in\rn:\ |T(f_1)(x)|>2^{k_0-1}\r\}\r|\noz\\
&&\hs\hs+2^{pk_0}\lf|\lf\{x\in\rn:\ |T(f_2)(x)|>2^{k_0-1}\r\}\r|\noz\\
&&\hs\ls\|f\|_{H^{p,q}_A(\rn)}^p+A_{(p,q)}\|f\|_{H^{p,q}_A(\rn)}
^p\ls A_{(p,q)}\|f\|_{H^{p,q}_A(\rn)}^p.\noz
\end{eqnarray}
Taking the supremum over all
$k_0\in\mathbb{Z}$ at the left-hand side of \eqref{sixe5},
we find that
$$\lf\|T(f)\r\|_{L^{p,\fz}(\rn)}
\ls\lf[A_{(p,q)}\r]^{\frac1p}\|f\|_{H^{p,q}_A(\rn)},$$
which is the desired conclusion of Theorem \ref{sixt1}.

Finally, we give the proof of \eqref{sixe4}.
To this end, for all $i\in\mathbb{N}$
and $k\in\mathbb{Z}$,
let $B_{\ell_i^k}:=x_i^k+B_i^k$,
where $\ell_i^k\in\mathbb{Z}$.
For every $k\in(k_0,\fz]\cap\mathbb{Z}$,
there exists an $m_k\in\mathbb{N}$ such that
$b^{m_k-2\tau-1}\leq(\frac32)^{p(k-k_0)}<b^{m_k-2\tau}$.
Let
$$B_{k_0}:=\bigcup_{k=k_0+1}^{\fz}
\bigcup_{i\in\mathbb{N}}B_{\ell_i^k+m_k+\tau}.$$
Notice that
$\lambda_i^k\sim2^k|B_i^k|^{1/p}
\sim2^k|B_{\ell_i^k}|^{1/p}$.
Since
$q/p\in(1,\fz]$, from the H\"{o}lder inequality,
we deduce that,
for all $k\in\mathbb{Z}$ and $i\in\mathbb{N}$,
\begin{eqnarray}\label{sixe6}
\lf|B_{k_0}\r|
&&\leq\sum_{k=k_0+1}^{\fz}\sum_{i\in\mathbb{N}}
\lf|B_{\ell_i^k+m_k+\tau}\r|\ls2^{-pk_0}\sum_{k=k_0+1}
^{\fz}\lf(\frac34\r)^{p(k-k_0)}\sum_{i\in\mathbb{N}}
\lf|\lambda_i^k\r|^p\\
&&\ls 2^{-pk_0}\lf\|\lf\{\mu_k\r\}_{k\in\mathbb{Z}}
\r\|_{\ell^q}^p
\sim2^{-pk_0}\|f\|_{H^{p,q}_A(\rn)}^p.\noz
\end{eqnarray}
Moreover, by the cancellation condition of $a_i^k$,
we have
\begin{eqnarray*}
&&\int_{\rn\setminus B_{k_0}}\lf|T(f_2)(x)\r|^p\,dx\\
&&\hs\leq\sum_{k=k_0+1}^{\fz}\sum_{i\in\mathbb{N}}
\lf|\lambda_i^k\r|^p\int_{\rn\setminus
B_{\ell_i^k+m_k+\tau}}\lf|T(a_i^k)(x)\r|^p\,dx\\
&&\hs=\sum_{k=k_0+1}^{\fz}\sum_{i\in\mathbb{N}}
\lf|\lambda_i^k\r|^p\\
&&\hs\hs\times\lf\{\int_{\rn\setminus B_{\ell_i^k+m_k+\tau}}
\lf|\dis\int_{B_{\ell_i^k}}\lf[k(x,y)-
\sum_{|\beta|\leq N}(y-y_{B_{\ell_i^k}})^\beta
k_\beta(x,y_{B_{\ell_i^k}})\r]a_i^k(y)\,dy\r|^p\,dx\r\}.
\end{eqnarray*}
Observe that, from
$x\in(B_{\ell_i^k+m_k+\tau})^\complement$
and
\eqref{se3}, we deduce that
$$\rho(x-y_{B_{\ell_i^k}})>b^{m_k}|B_{\ell_i^k}|>
b^{2\tau}\lf(\frac32\r)^{p(k-k_0)}|B_{\ell_i^k}|.$$
Hence, by the H\"{o}lder inequality,
we find that
\begin{eqnarray*}
\int_{\rn\setminus B_{k_0}}\lf|T(f_2)(x)\r|^p\,dx
&&\leq\sum_{k=k_0+1}^{\fz}\omega_p
\lf(\lf(\frac23\r)^{p(k-k_0)}\r)\mu_k^p\\
&&\leq\lf\{\sum_{k=k_0+1}^{\fz}\lf[\omega_p
\lf(\lf(\frac23\r)^{p(k-k_0)}\r)\r]^{\frac q{q-p}}\r\}
^{\frac{q-p}q}\lf\|\lf\{\mu_k\r\}_{k\in\mathbb{Z}}\r\|
_{\ell^q}^p\noz\\
&&\ls\lf\{\int_0^1\lf[\omega_p(\delta)\r]^{\frac q{q-p}}
\,\frac{d\delta}\delta\r\}^{\frac{q-p}q}\|f\|_{H^{p,q}_A(\rn)}
^p\sim A_{(p,q)}\|f\|_{H^{p,q}_A(\rn)}^p,\noz
\end{eqnarray*}
which further implies that
$$2^{pk_0}\lf|\lf\{x\in(B_{k_0})^
\complement:\ \lf|T(f_2)(x)\r|>2^{k_0}\r\}\r|
\ls A_{(p,q)}\|f\|_{H^{p,q}_A(\rn)}^p.$$
By this and \eqref{sixe6}, we conclude that
\begin{eqnarray*}
2^{pk_0}\lf|\lf\{x\in\rn:\
\lf|T(f_2)(x)\r|>2^{k_0}\r\}\r|
&&\ls 2^{pk_0}\lf[\lf|B_{k_0}\r|+
\lf|\lf\{x\in(B_{k_0})^\complement:
\ \lf|T(f_2)(x)\r|>2^{k_0}\r\}\r|\r]\\
&&\ls\lf(1+A_{(p,q)}\r)\|f\|_
{H^{p,q}_A(\rn)}^p\ls A_{(p,q)}\|f\|
_{H^{p,q}_A(\rn)}^p,
\end{eqnarray*}
which proves \eqref{sixe4}.
This finishes the proof of Theorem \ref{sixt1}
\end{proof}

\begin{remark}\label{sixr2}
(i) If $A$ is the same as in Remark \ref{sixr1}(i), then $N=\lfloor n(1/p-1)\rfloor$,
$\frac{\ln b}{\ln\lambda_-}=n$ and, $H^{p,q}_A(\rn)$
and $L^{p,\fz}(\rn)$ become the classical isotropic Hardy-Lorentz and weak Lebesgue spaces,
respectively. In this case, Theorem \ref{sixt1} is just \cite[Theorem 2.2]{wa07}.

(ii) It is well known that the H\"{o}rmander condition implies
the boundedness of the Calder\'on-Zygmund operator $T$ from
$H^1_A(\rn)$ to $L^1(\rn)$. Observe that $H^1_A(\rn)\subsetneqq H^{1,q}_A(\rn)$ with $q\in(1,\fz]$.
Thus, to define $T$ on $H^{1,q}_A(\rn)$ with $q\in(1,\fz]$, it is natural to
requrie $T$ to satisfy some conditions stronger than the usual H\"{o}rmander condition. This
was accomplished by the
variable dilations (the Dini-type condition \eqref{d-c})
of Fefferman and Soria \cite{fs87} (see also \cite{wa07}).
Moreover, recall that we consider $p=\frac1{1+\delta}$ or $p\in(\frac1{1+\delta},1]$
with $\delta\in(0,\frac{\ln\lz_-}{\ln b}]$
in Theorem \ref{sixt3} and Remark \ref{sixr1}, which implies
$N=\lfloor \frac{\ln b}{\ln\lz_-}(1/p-1)\rfloor\le1$. But, in Theorem \ref{sixt1}, we
consider $p\in(0,1]$. If $p$ becomes smaller, then $N$ becomes larger. Thus, more
regularity of the kernel of $T$ is needed. This justifies the definition of
$\omega_p(\delta)$ in Theorem \ref{sixt1}.

(iii) If $\delta\in(0,\frac{\ln\lz_-}{\ln b}]$, $p\in(\frac1{1+\delta},1]$
and $T$ is a non-convolutional
$\dz$-type Calder\'on-Zygmund operator which satisfies all conditions in
Remark \ref{sixr1}(ii) with
\eqref{sixe15} replaced by
\begin{eqnarray*}
|\mathcal{K}(x,y)-\mathcal{K}(x,z)|\le C\frac{\lf[\rho(y-z)\r]^\dz}
{\lf[\rho(x-y)\r]^{1+\dz}}
\quad{\rm if}\quad \rho(x-y)>b^\tau\rho(y-z),
\end{eqnarray*}
then
$N=\lfloor \frac{\ln b}{\ln\lz_-}(1/p-1)\rfloor=0$
and $p(1+\delta)>1$.
Thus, for $p=q$,
\begin{eqnarray*}
A_{(p,p)}&&=\sup_{\delta\in(0,1]}\lf\{\omega_p(\delta)\r\}\\
&&=\sup_{B}\frac1{|B|}\int_{\rho(x-y_B)>b^{2\tau}|B|}
\lf[\dis\int_{B}\lf|\mathcal{K}(x,y)-\mathcal{K}(x,y_B)\r|\,dy\r]^p\,dx\\
&&\ls\sup_{B}\frac1{|B|}\int_{\rho(x-y_B)>b^{2\tau}|B|}
\lf[\dis\int_{B}\frac{\lf[\rho(y-y_B)\r]^\dz}
{\lf[\rho(x-y)\r]^{1+\dz}}\,dy\r]^p\,dx\\
&&\ls\sup_{B}\frac1{|B|}\sum_{k=0}^\fz
\int_{b^kb^{2\tau}|B|<\rho(x-y_B)\le b^{k+1}b^{2\tau}|B|}
\lf[\dis\int_{B}\frac{\lf[\rho(y-y_B)\r]^\dz}
{\lf[\rho(x-y_B)\r]^{1+\dz}}\,dy\r]^p\,dx\\
&&\ls\sup_{B}\frac1{|B|}\sum_{k=0}^\fz\lf[
\frac{|B|^{1+\delta}}{\lf(b^k|B|\r)^{1+\delta}}\r]^pb^k|B|\sim1,
\end{eqnarray*}
where the supremum is taken over
all dilated balls $B\in\mathfrak{B}$ centered at $y_B$
and $\mathfrak{B}$ is as in \eqref{se14}.
This shows that (iv) of Remark \ref{sixr1} is
the endpoint (critical) case of Theorem
\ref{sixt1} in the sense of $p=q$.
\end{remark}

Recall that a \emph{quasi-Banach space} $\mathcal{B}$ is a vector space endowed
with a quasi-norm $\|\cdot\|_{\mathcal{B}}$, which is non-negative, non-degenerate
(i.\,e., $\|f\|_{\mathcal{B}}=\theta$ if and only if $f=\theta$)
and satisfying the quasi-triangle inequality;
namely, there exists a positive constant $K\in[1,\fz)$ such that, for all
$f,\,g\in\mathcal{B}$,
$\|f+g\|_{\mathcal{B}}\le K(\|f\|_{\mathcal{B}}+\|g\|_{\mathcal{B}})$.
Clearly, a quasi-Banach space $\mathcal{B}$ is called a Banach space if $K=1$.

Let $\mathcal{B}$ be a quasi-Banach space
and $\mathcal{Y}$ a linear space. An operator
$T$ from $\mathcal{Y}$ to $\mathcal{B}$ is said to be
$\mathcal{B}$-\emph{sublinear} if
there exists a positive constant $C$ such that,
for any $\lz,\,\mu\in\mathbb{C}$ and $f,\,g\in\mathcal{Y}$,
$$\lf\|T\lf(\lz f+\mu g\r)\r\|_{\mathcal{B}}
\le C\lf[|\lz|\lf\|T(f)\r\|_{\mathcal{B}}
+|\mu|\lf\|T(g)\r\|_{\mathcal{B}}\r]$$
and
$\|T(f)-T(g)\|_{\mathcal{B}}
\le C\|T(f-g)\|_{\mathcal{B}}$.
Obviously, if $T$ is linear, then $T$ is $\mathcal{B}$-sublinear.

As an application of the finite atomic decomposition characterizations obtained in Section
\ref{s7} (see Theorem \ref{sevent1}), as well as Theorem \ref{sevent2},
we establish the following criterion for the boundedness of sublinear
operators from $H_A^{p,q}(\mathbb{R}^n)$ into a quasi-Banach
space $\mathcal{B}$, which is a variant of \cite[Theorem 5.9]{gly08};
see also \cite[Theorem 3.5]{ky14} and \cite[Theorem 1.1]{yz08}.

\begin{theorem}\label{sevent2}
Let $(p,r,s)$ be an admissible anisotropic triplet, $q\in(0,\fz)$
and $\mathcal{B}$ be a quasi-Banach space.
If one of the following statements holds true:
\begin{enumerate}
\item[{\rm (i)}] $r\in(1,\fz)$ and
$T:\ H_{A,{\rm fin}}^{p,r,s,q}(\rn)\rightarrow\mathcal{B}$
is a $\mathcal{B}$-sublinear operator satisfying that
there exists a positive constant $C_{17}$ such that,
for all $f\in H_{A,{\rm fin}}^{p,r,s,q}(\rn)$,
\begin{eqnarray}\label{sevene20}
\|T(f)\|_{\mathcal{B}}\le C_{17}\|f\|_{H_{A,{\rm fin}}^{p,r,s,q}(\rn)};
\end{eqnarray}
\item[{\rm (ii)}]
$T:\ H_{A,{\rm fin}}^{p,\fz,s,q}(\rn)\cap C(\rn)\rightarrow\mathcal{B}$
is a $\mathcal{B}$-sublinear operator satisfying that
there exists a positive constant $C_{18}$ such that,
for all $f\in H_{A,{\rm fin}}^{p,\fz,s,q}(\rn)\cap C(\rn)$,
$$\|T(f)\|_{\mathcal{B}}\le C_{18}\|f\|_{H_{A,{\rm fin}}^{p,\fz,s,q}(\rn)},$$
\end{enumerate}
then $T$ uniquely extends to a bounded sublinear operator from $H^{p,q}_A(\rn)$
into $\mathcal{B}$. Moreover, there exists a positive constant $C_{19}$ such that,
for all $f\in H^{p,q}_A(\rn)$,
$$\|T(f)\|_{\mathcal{B}}\le C_{19}\|f\|_{H^{p,q}_A(\rn)}.$$
\end{theorem}

\begin{proof}
We first prove (i).
For any given $r\in(1,\fz)$, assume that \eqref{sevene20} holds true.
Let $f\in H^{p,q}_A(\rn)$. By Theorem \ref{tt1} and the density of
$H_{A,{\rm fin}}^{p,r,s,q}(\rn)$ in $H_A^{p,r,s,q}(\rn)$, we know that
$H_{A,{\rm fin}}^{p,r,s,q}(\rn)$ is dense in $H^{p,q}_A(\rn)$. Thus, there
exists a Cauchy sequence
$\{f_j\}_{j\in\nn}\subset H_{A,{\rm fin}}^{p,r,s,q}(\rn)$ such that
$\lim_{j\to\fz}\|f_j-f\|_{H^{p,q}_A(\rn)}=0$, which, together with \eqref{sevene20}
and Theorem \ref{sevent1}(i), further implies that
\begin{eqnarray*}
\|T(f_i)-T(f_j)\|_{\mathcal{B}}
&&\ls\|T(f_i-f_j)\|_{\mathcal{B}}\ls
\|f_i-f_j\|_{H_{A,{\rm fin}}^{p,r,s,q}(\rn)}\\
&&\sim\|f_i-f_j\|_{H^{p,q}_A(\rn)}\rightarrow0,
\ \ \ {\rm as}\ \ i,\,j\rightarrow\fz.
\end{eqnarray*}
Therefore, $\{T(f_j)\}_{j\in\nn}$ is a Cauchy sequence in $\mathcal{B}$,
which, combined with the completeness of $\mathcal{B}$, implies that there
exists $F\in\mathcal{B}$ such that $F=\lim_{j\to\fz}T(f_j)$ in $\mathcal{B}$.
Let $T(f):=F$. From \eqref{sevene20} and Theorem \ref{sevent1}(i) again, we
further deduce that $T(f)$ is well defined and
\begin{eqnarray*}
\|T(f)\|_{\mathcal{B}}&&\ls\limsup_{j\to\fz}\lf[\|T(f)-T(f_j)\|_{\mathcal{B}}
+\|T(f_j)\|_{\mathcal{B}}\r]\ls\limsup_{j\to\fz}\|T(f_j)\|_{\mathcal{B}}\\
&&\ls\limsup_{j\to\fz}\|f_j\|_{H_{A,{\rm fin}}^{p,r,s,q}(\rn)}
\sim\lim_{j\to\fz}\|f_j\|_{H^{p,q}_A(\rn)}\sim\|f\|_{H^{p,q}_A(\rn)}.
\end{eqnarray*}
This finishes the proof of (i).

Now we prove (ii). For $q\in(0,\fz)$, we first claim that
$H_{A,{\rm fin}}^{p,\fz,s,q}(\rn)\cap C(\rn)$ is dense in $H^{p,q}_A(\rn)$.
Indeed, by Lemma \ref{sevenl4}(ii), we know that
$H^{p,q}_A(\rn)\cap C(\rn)$ is dense in $H^{p,q}_A(\rn)$. Thus, we only
need to show that $H_{A,{\rm fin}}^{p,\fz,s,q}(\rn)\cap C(\rn)$ is dense in
$H^{p,q}_A(\rn)\cap C(\rn)$ with respect to the quasi-norm
$\|\cdot\|_{H^{p,q}_A(\rn)}$. For any $f\in H^{p,q}_A(\rn)\cap C(\rn)$,
by an argument similar to that used in the proof of Theorem \ref{tt1} (or Lemma \ref{sevenl1}),
we find that there exists a sequence of $(p,\fz,s)$-atoms,
$\{a_i^k\}_{i\in\mathbb{N},\,k\in\mathbb{Z}}$, and
$\{\lz_i^k\}_{i\in\mathbb{N},\,k\in\mathbb{Z}}\subset\mathbb{C}$
such that
$f=\sum_{k\in\mathbb{Z}}
\sum_{i\in\mathbb{N}}\lambda_i^ka_i^k$ in $\cs'(\rn)$. Moreover,
from definitions of these $(p,\fz,s)$-atoms (see \eqref{sevene17}) and the
continuity of $f$, we further deduce that all these $(p,\fz,s)$-atoms are
continuous. Thus, for any $K\in\nn$, if we let
$f_K:=\sum_{|k|=0}^K\sum_{i=1}^K\lambda_i^ka_i^k$,
then it is easy to see that
$\{f_K\}_{K\in\nn}\subset H_{A,{\rm fin}}^{p,\fz,s,q}(\rn)\cap C(\rn)$
and
$$\lim_{K\to\fz}\|f-f_K\|_{H^{p,q}_A(\rn)}=0,$$
which implies that $H_{A,{\rm fin}}^{p,\fz,s,q}(\rn)\cap C(\rn)$ is dense in
$H^{p,q}_A(\rn)\cap C(\rn)$ with respect to the quasi-norm
$\|\cdot\|_{H^{p,q}_A(\rn)}$.

By the density of $H_{A,{\rm fin}}^{p,\fz,s,q}(\rn)\cap C(\rn)$ in $H^{p,q}_A(\rn)$
and a proof similar to (i), we conclude that (ii) holds true.
This finishes the proof of Theorem \ref{sevent2}.
\end{proof}

By Theorem \ref{sevent2}, we easily obtain the following conclusion,
the details being omitted.

\begin{corollary}\label{sevenc1}
Let $(p,r,s)$ be an admissible anisotropic triplet, $q\in(0,\fz)$
and $\mathcal{B}$ be a quasi-Banach space. If one of the following statements holds true:
\begin{enumerate}
\item[{\rm (i)}] $r\in(1,\fz)$ and $T$ is a $\mathcal{B}$-sublinear
operator from $H_{A,{\rm fin}}^{p,r,s,q}(\rn)$ to $\mathcal{B}$ satisfying that
$$\mathfrak{S}:=\sup\lf\{\|T(a)\|_{\mathcal{B}}:\
a\ is\ any\ (p,r,s){\text-}atom\r\}<\fz;$$
\item[{\rm(ii)}] $T$ is a $\mathcal{B}$-sublinear
operator defined on continuous $(p,\fz,s)$-atoms satisfying that
$$\mathfrak{S}:=\sup\lf\{\|T(a)\|_{\mathcal{B}}:\
a\ is\ any\ continuous\ (p,\fz,s){\text-}atom\r\}<\fz,$$
\end{enumerate}
then $T$ has a unique bounded $\mathcal{B}$-sublinear
extension $\widetilde{T}$ from $H^{p,q}_A(\rn)$ to $\mathcal{B}$.
\end{corollary}

\begin{remark}\label{sevenr2}
(i) Obviously, if $T$ is a bounded $\mathcal{B}$-sublinear operator
from $H^{p,q}_A(\rn)$ to $\mathcal{B}$, then, for any admissible
anisotropic triplet $(p,r,s)$, $T$ is uniformly bounded on all $(p,r,s)$-atoms
. Corollary \ref{sevenc1}(i) shows that the converse holds
true for $r\in(1,\fz)$. However, such converse conclusion is not true
in general for $r=\fz$ due to the example in \cite[Theorem 2]{mb05}.
Namely, there exists an operator $\mathcal{T}$ which is uniformly bounded
on all $(1,\fz,0)$-atoms, but does not have a bounded extension
on $H^1(\rn)$.

(ii) Corollary \ref{sevenc1}(ii) shows that the uniform boundedness
of $T$ on a smaller class of continuous $(p,\fz,s)$-atoms implies
the existence of a bounded extension on the whole space $H^{p,q}_A(\rn)$.
In particular, if we restrict the operator $\mathcal{T}$, in (i) of this remark,
to the subspace $H_{A,{\rm fin}}^{1,\fz,0,1}(\rn)\cap C(\rn)$, then
such restriction has a bounded extension, denoted by
$\widetilde{\mathcal{T}}$, to the whole space $H^1_A(\rn)$. However,
$\mathcal{T}$ itself does not have such property. Precisely,
$\mathcal{T}$ and $\widetilde{\mathcal{T}}$ coincide on all continuous
$(1,\fz,0)$ atoms, while not on all $(1,\fz,0)$ atoms; see also \cite{msv08}.
This shows that it is necessary to restrict the operator $T$ only on continuous
atoms for $r=\fz$ in Corollary \ref{sevenc1}(ii).
\end{remark}

Now we use Corollary \ref{sevenc1} and Theorem \ref{sixt2} to show the
boundedness of the
$\dz$-type Calder\'on-Zygmund operators from
$H^{p,q}_A(\rn)$ to
$L^{p,q}(\rn)$ (or $H^{p,q}_A(\rn)$).

\begin{theorem}\label{sixt4}
Let $\delta\in(0,\frac{\ln\lz_-}{\ln b}]$,
$p\in(\frac1{1+\delta},1]$ and $q\in(0,\fz]$.
\vspace{-0.25cm}
\begin{enumerate}
\item[{\rm (i)}]
If $T$ is either a convolutional
$\dz$-type Calder\'on-Zygmund operator as  in Theorem \ref{sixt3} or
a non-convolutional
$\dz$-type Calder\'on-Zygmund operator as in Remark \ref{sixr1}(ii),
then there exists a positive constant $C_{20}$
such that, for all $f\in H_A^{p,q}(\rn)$,
$$\|T(f)\|_{L^{p,q}(\rn)}\le C_{20}\|f\|_{H_A^{p,q}(\rn)}.$$
\vspace{-0.65cm}
\item[{\rm (ii)}]
If $T$ is either a convolutional
$\dz$-type Calder\'on-Zygmund operator as  in Theorem \ref{sixt3} or
a non-convolutional
$\dz$-type Calder\'on-Zygmund operator satisfying the additional assumption that $T^*1=0$
as in Remark \ref{sixr1}(ii), then there exists a positive constant $C_{21}$
such that, for all $f\in H_A^{p,q}(\rn)$,
$$\|T(f)\|_{H_A^{p,q}(\rn)}\le C_{21}\|f\|_{H_A^{p,q}(\rn)}.$$
\end{enumerate}
\end{theorem}

\begin{proof}
We first prove (i). When $\delta\in(0,\frac{\ln\lz_-}{\ln b}]$,
$p\in(\frac1{1+\delta},1)$ and $q\in(0,\fz]$,
by the proof of
\cite[p.\,69, Theorem 9.9]{mb03}, we have
$\|T(a)\|_{L^p(\rn)}\ls1$ for any $(p,2,0)$-atom $a$. From this
and Corollary \ref{sevenc1}(i), we further deduce that, for all $f\in H^{p,q}_A(\rn)$,
\begin{eqnarray}\label{sixe22}
\|T(f)\|_{L^p(\rn)}\ls\|f\|_{H^{p,q}_A(\rn)}.
\end{eqnarray}
Notice that $T$ is a linear operator. By \eqref{sixe22}, the corresponding
interpolation result of Lebesgue spaces (see, for example, \cite[Theorem 3]{m84})
and Theorem \ref{sixt2}(i), we easily conclude that (i) holds true
when $p\in(\frac1{1+\delta},1)$ and $q\in(0,\fz]$.

When $p=1$ and $q\in(0,\fz]$, combining the linearity of $T$, \eqref{sixe22}
with $p=q$, the boundedness of $T$ on $L^r(\rn)$ with $r\in(1,\fz)$
(see, for example, \cite[Theorems 5.1 and 5.10]{d01}) and Corollary \ref{sixc1},
we find that (i) also holds true in this case.

Now we turn to show (ii).
Notice that, by \cite[p.\,64, Lemma 9.5]{mb03} and Theorem \ref{tt1},
we know that $\|T(a)\|_{H^p_A(\rn)}\ls1$ for any $(p,2,0)$-atom $a$.
From this, Corollary \ref{sixc1} and an
argument similar to the proof of (i), we deduce that (ii) holds true.
This finishes the proof of Theorem \ref{sixt4}.
\end{proof}

\begin{remark}\label{sixr6}
(i) Noticing that the
$\dz$-type Calder\'on-Zygmund operators are linear operators.
By Remark \ref{sixr1}(iv), \cite[Theorems 5.1 and 5.10]{d01}, Corollary \ref{sixc1}
and the corresponding
interpolation result of Lorentz spaces (see, for example, \cite[Theorem 3]{m84}),
we can also conclude that, as in Theorem \ref{sixt4},
the boundedness of the $\dz$-type Calder\'on-Zygmund operators
from $H_A^{p,q}(\rn)$ to $L^{p,q}(\rn)$ (or $H_A^{p,q}(\rn)$) with
$\delta\in(0,\frac{\ln\lz_-}{\ln b}]$,
$p\in(\frac1{1+\delta},1]$ and $q\in(0,\fz]$, the details being omitted.

(ii) We should point out that, in (i) of this remark, the boundedness of
the $\dz$-type Calder\'on-Zygmund operators
from $H_A^p(\rn)$ to $L^p(\rn)$ (or $H_A^p(\rn)$) is a key tool, namely,
\cite[p.\,68, Theorem 9.8 and p.\,69, Theorem 9.9]{mb03}
(see Remark \ref{sixr1}(iv)). Notice that the proofs of
\cite[p.\,68, Theorem 9.8 and p.\,69, Theorem 9.9]{mb03}
also need the facts that
$$\|T(a)\|_{H^p_A(\rn)}\ls1\quad\mathrm{and}\quad \|T(a)\|_{L^p(\rn)}\ls1$$
for any $(p,2,0)$-atom $a$, respectively,
and are more complicated than the proof of Theorem \ref{sixt4}.
Thus, in this sense, the criterion established in Theorem \ref{sevent2}
is a useful tool.

(iii) If $A$ is the same as in Remark \ref{sixr1}(i),
then $\frac{\ln\lambda_-}{\ln b}=\frac1n$,
$H^{p,q}_A(\rn)$
and $L^{p,q}(\rn)$ become the classical isotropic Hardy-Lorentz and Lorentz spaces,
respectively, and $T$ becomes the classical $\dz$-type Calder\'on-Zygmund operator
correspondingly. In this case, we know that, if $\delta\in(0,1]$,
$p\in(\frac n{n+\delta},1]$ and $q\in(0,\fz]$, then Theorem \ref{sixt4}(i)
implies that $T$ is bounded from $H^{p,q}(\rn)$ to $H^{p,q}(\rn)$ and
Theorem \ref{sixt4}(ii) implies that $T$ is bounded from $H^{p,q}(\rn)$ to $L^{p,q}(\rn)$.
Moreover, when $p=q$, (i) and (ii) of Theorem \ref{sixt4} imply the boundedness of
the classical $\dz$-type Calder\'on-Zygmund operator, respectively,
from $H^p(\rn)$ to $H^p(\rn)$ and
from $H^p(\rn)$ to $L^p(\rn)$ for $\delta\in(0,1]$ and $p\in(\frac n{n+\delta},1]$,
which is a well-known result (see, for example, \cite{a86,s93,lu95}).
\end{remark}

\bigskip

\smallskip

\noindent  Jun Liu, Dachun Yang (Corresponding author) and Wen Yuan

\medskip

\noindent  School of Mathematical Sciences, Beijing Normal University,
Laboratory of Mathematics and Complex Systems, Ministry of
Education, Beijing 100875, People's Republic of China

\smallskip

\noindent {\it E-mails}: \texttt{junliu@mail.bnu.edu.cn} (J. Liu)

\hspace{1.12cm}\texttt{dcyang@bnu.edu.cn} (D. Yang)

\hspace{1.12cm}\texttt{wenyuan@bnu.edu.cn} (W. Yuan)

\end{document}